\documentclass[12pt, reqno]{amsart}

\usepackage[margin=2.4cm]{geometry}

\usepackage{amsmath,amsfonts,amssymb,amsthm,mathrsfs,microtype}
\numberwithin{equation}{section}
\usepackage{hyperref,graphicx}
\usepackage{xcolor}
\usepackage{cite}
\usepackage{enumitem}
\usepackage{ulem}
\usepackage{bm}

\makeatletter

\theoremstyle{plain}
\newtheorem{thm}{\protect\theoremname}[section]
\theoremstyle{plain}
\newtheorem{cor}[thm]{\protect\corollaryname}
\theoremstyle{plain}
\newtheorem{lem}[thm]{\protect\lemmaname}
\theoremstyle{plain}
\newtheorem{proposition}[thm]{\protect\propositionname}
\theoremstyle{plain}

\newtheorem{definition}{Definition}

\newtheorem{claim}{Claim}
\newtheorem{example}{Example}
\theoremstyle{rem}
\newtheorem{rem}{{Remark}}

\makeatother

\usepackage{babel}
\providecommand{\corollaryname}{Corollary}
\providecommand{\propositionname}{Proposition}
\providecommand{\lemmaname}{Lemma}

\providecommand{\theoremname}{Theorem}

\providecommand{\notationname}{Notation}

\DeclareMathOperator{\dimh}{dim_H}

\DeclareMathOperator{\R}{\mathbb{R}}
\DeclareMathOperator{\cS}{\mathcal{S}}
\DeclareMathOperator{\K}{\mathcal{K}}
\DeclareMathOperator{\cJ}{\mathcal{J}}
\DeclareMathOperator{\T}{\mathbb{T}}
\DeclareMathOperator{\N}{\mathbb{N}}
\DeclareMathOperator{\Z}{\mathbb{Z}}
\DeclareMathOperator{\cN}{\mathcal{N}}

\DeclareMathOperator{\cY}{\mathcal{Y}}
\DeclareMathOperator{\bfx}{{\bf{x}}}
\DeclareMathOperator{\bfy}{{\bf{y}}}
\DeclareMathOperator{\bfz}{{\bf{z}}}
\DeclareMathOperator{\bfw}{{\bf{w}}}

\DeclareMathOperator{\dH}{{\mathrm{d}_{\mathrm{H}}}}
\DeclareMathOperator{\dist}{{\mathrm{dist}}}

\DeclareMathOperator{\s}{\mathcal{S}}
\DeclareMathOperator{\m}{\ (\bmod ~1)}

\date{}

\allowdisplaybreaks[3]

\begin{document}
\begin{sloppypar}
\title[dimension theory  of inhomogeneous diophantine approximation ]{Dimension theory  of inhomogeneous Diophantine approximation with matrix sequences}
\author{Zhang-nan Hu, Junjie Huang, Bing Li*, Jun Wu}

\thanks{* Corresponding author}

\date{\today}	

\begin{abstract}

In this paper, we investigate the Hausdorff dimension of naturally occurring sets of inhomogeneous well-approximable points with a sequence of real invertible matrices $\mathcal{A}=(A_n)_{n\in\mathbb{N}}$. Specifically, for a given point $\mathbf{y}\in [0,1)^d$ and a function $\psi : \N \to \R^+$, we study the limsup set
\[ W\big(\mathcal{A},\psi,{\bf y}\big)
 =\Big\{\mathbf{x}\in [0,1)^d\colon A_n\mathbf{x}~(\bmod~1)\in B\big(\mathbf{y}, \psi(n)\big)  {\rm ~ for~ infinitely ~many}~n\in\N\Big\}.\]
The upper and lower bounds on the Hausdorff dimension of $W\big(\mathcal{A},\psi,{\bf y}\big)$ 
are determined by involving the singular values of $A_n$  and the successive minima of the lattice $A_n^{-1}\mathbb{Z}^d$, and both bounds are shown to be attainable for some matrices.
Within this framework, we unify the problem of shrinking target sets and recurrence sets, establishing the Hausdorff dimensions for such limsup sets. As applications, our corresponding upper bounds for shrinking target and recurrence sets essentially improve those appearing in the present literature. 
 Furthermore, explicit Hausdorff dimension formulas are derived for shrinking targets and recurrence sets associated with concrete classes of matrices.

We extend the Mass Transference Principle for rectangles of Li-Liao-Velani-Wang-Zorin (Adv. Math., 2025) to rectangles under local isometries. This generalization yields a general lower bound for the Hausdorff dimension of $W\big(\mathcal{A},\psi,{\bf y}\big)$. 
 
\end{abstract}

\maketitle

\section{Introduction }

We begin by setting the scene. Let $d\ge 1$. For  $\mathbf{x}=(x_1,...,x_d)\in\R^d$,  denote  $\Vert \mathbf{x}\Vert_d:=\sqrt{\Vert x_1\Vert^2+...+\Vert x_d\Vert^2}$, where $\Vert x_i\Vert:=\mathrm{dist}(x_i,\mathbb{Z})$, $1\le i\le d$. 
A ball $B(\mathbf{x},r):=\{\mathbf{y}\in [0,1)^d \colon \Vert  {\bf y}-{\bf x}\Vert_d <r\}$ is defined by a centre $\mathbf{x}\in [0,1)^d$ and radius $r>0$.
Let $\mathbb{R}^+:=(0,\infty)$ and $\psi:\mathbb{R}^+\to\mathbb{R}^+$ be a nonincreasing continuous function.  Let  $\mathcal{A}:=(A_n)_{n\in\mathbb{N}}$ be a sequence of $d\times d$ invertible matrices with real  entries. We are interested in the distribution of $(A_n{\bf x})_{n\in\mathbb{N}}$, where ${\bf x}\in [0,1)^d$, focusing on their approximation properties to a target point $\mathbf{y}\in [0,1)^d$. More precisely, for ${\bf y}\in [0,1)^d$, consider the set 
\begin{equation*}
W\big(\mathcal{A},\psi,{\bf y}\big)
 =\Big\{\mathbf{x}\in [0,1)^d\colon A_n\mathbf{x}\m\in B\big(\mathbf{y}, \psi(n)\big)  {\rm ~ for~ infinitely ~many}~n\in\N\Big\}.
 \end{equation*}
It is therefore natural to consider  the size of $W\big(\mathcal{A},\psi,{\bf y}\big)$ in terms of 
\begin{enumerate}
\item $d$-dimensional Lebesgue measure $\mathcal{L}^d$,\medskip 
\item Hausdorff dimension,\medskip
\item Hausdorff measure.
\end{enumerate}
A  straightforward consequence of the Borel-Cantelli Lemma is that if $\sum_{n\in\mathbb{N}}\mathcal{L}^d(W_n)<\infty$, where $W_n=\{\mathbf{x}\in [0,1)^d\colon A_n\mathbf{x}\m\in B\big(\mathbf{y}, \psi(n)\big)\}$, then  $\mathcal{L}^d(W\big(\mathcal{A},\psi,{\bf y}\big))=0$. Consequently, the purpose of this work is to investigate the size of $W\big(\mathcal{A},\psi,{\bf y}\big)$ in Hausdorff dimension.

 A key motivating factor is the theory of Diophantine approximation, which originated in the study of rational approximations and plays a crucial role in characterizing the arithmetic properties of numbers. 
For a point $\mathbf{y}\in [0,1)^d$ and a function $\psi:\mathbb{R}^+\to\mathbb{R}^+$, define the set
\begin{equation}\label{das}
W(\mathbf{y}, \psi) :=\{\mathbf{x} \in [0, 1)^d : \langle q\mathbf{x} -\mathbf{y}\rangle < \psi(q) {\rm ~for ~infinitely ~many ~}q \in\N\},
\end{equation}
where $\langle\cdot   \rangle $ denotes the shortest distance from a point in $\mathbb{R}^d$ to $\mathbb{Z}^d$ under the maximum norm, that is, $\langle  \mathbf{x}  \rangle =\max_{1\le i\le d}\Vert x_i\Vert$   for  $\mathbf{x}=(x_1,...,x_d)\in\R^d$. The size of $W(\mathbf{y}, \psi)$ is a core subject of the study of metric Diophantine approximation. 
According as $\mathbf{y}=\mathbf{0}$ or not,  the study of $W(\mathbf{y}, \psi)$ is referred to as homogeneous or inhomogeneous  Diophantine approximation.   In the case 
when $\mathbf{y}=\mathbf{0}$ and $d\ge 2$,   Gallagher \cite{ga} showed that  the Lebesgue measure of $W(\mathbf{0}, \psi)$ obeys zero-one law according to the convergence or divergence of $\sum_{q=1}^\infty\psi(q)^d$. When $d=1$, Duffin and Schaeffer \cite{DSC} proved that  the aforementioned conclusion holds when $\psi$ is monotonic, and this condition is crucial. 
When 
$\psi$ is monotonic, the Hausdorff measure of the set 
$W(\mathbf{0}, \psi)$ satisfies a 0-1 law depending on the convergence or divergence of a certain series, as established in  \cite{bbdv},  and the monotonicity assumption on 
$\psi$ can be dropped for $d\ge2$. For general approximation functions 
$\psi$ in the case 
$d=1$, Rynne \cite{Ryn} derived a formula for the Hausdorff dimension of 
$W(\mathbf{0}, \psi)$. 

When $\mathbf{y}\ne 0$ and $\psi$ is monotonic, Sprind{\v z}uk \cite{Sprind} established that  the Lebesgue  measure of $W(\mathbf{y}, \psi)$ has 0-1 law. Ram\'irez  \cite{ram} demonstrated that the monotonicity condition on $\psi$ is essential when $d=1$. Yu \cite{han2} proved that the monotonicity can be omitted if $d\ge3$, while Kim \cite{Kim} showed that for $d=2$, this condition can be removed provided $\mathbf{y}\in\mathbb{Q}^2$.  Levesley \cite{leve} obtained the Hausdorff dimension of 
$W(\mathbf{y}, \psi)$ under the monotonicity condition on $\psi$.  Furthermore, Beresnevich et al.\cite{bbdv} characterized the Hausdorff measure of this set, with the monotonicity condition being dispensable for $d\ge2$.
 Pollington, Velani and Zafeiropoulos \cite{PVZ} replaced $q$ by a lacunary sequence $(q_n)_{n\in\mathbb{N}}$, and considered the following set 
\begin{equation*}
W((q_n)_n,\psi,y)=\left\{x\in [0,1): \Vert q_nx-y\Vert < \psi(n) {\rm ~for~infinitely~many ~}n\in\N\right\},
\end{equation*}
and then under some condition on  the Fourier transform of $ \mu $,  they established a quantitative inhomogeneous Khintchine-type theorem for $W((q_n)_n,\psi,y)\cap E$ with $E\subset [0,1)$. Here if  the norm $\langle\cdot   \rangle  $ is replaced by  $\Vert\cdot\Vert_d$  in \eqref{das}, denoting the corresponding set by $W_d(\psi,\mathbf{y})$,  the aforementioned conclusions regarding  measure and dimension of $W(\psi,\mathbf{y})$ also hold for $W_d(\psi,\mathbf{y})$, which follows from that there exists a constant $c>1$ such that $c^{-1}\Vert \mathbf{x} \Vert_d\le \langle  \mathbf{x}  \rangle \le c\Vert \mathbf{x} \Vert_d$  for all $\mathbf{x}\in\R^d$. 

Our problem is closely related to, but differs from, the quantitative recurrence and shrinking target problems under matrix transformation of $[0,1)^d$.   Like these limsup set studies,  our primary tool is a novel Mass Transference Principle. 
Detailed comparisons with previous work appear in Section \ref{compare}.

\subsection{The problem set up and main results}\label{applications}

Let $F\subset \mathbb{R}$ and $GL_d(F)$ denote the space of $d\times d$ invertible matrices with entries in $F$. 
\begin{definition}[{\bf singular value}]
Let $A\in GL_d(\mathbb{R})$. Let $\lambda_1\le  \cdots \le\lambda_d$  denote the eigenvalues of $A^TA$, with repetitions. Let $\sigma_i(A) = \sqrt{\lambda_i}$, $1\le i\le d$. The numbers $\sigma_1(A)\le  \cdots \le\sigma_d(A)$ defined above are called the singular values of $A$.
\end{definition}
Let $\mathcal{A}=(A_n)_{n\in\N}\subset GL_d(\R)$. 
Recall that  for  ${\bf y}\in [0,1)^d$,
\begin{equation*}
W\big(\mathcal{A},\psi,{\bf y}\big)
 =\Big\{\mathbf{x}\in [0,1)^d\colon A_n\mathbf{x}\m\in B\big(\mathbf{y}, \psi(n)\big)  {\rm ~ for~ infinitely ~many}~n\in\N\Big\}.
 \end{equation*}
As mentioned above, this paper is primarily concerned with the following question:
\bigskip

\noindent\textbf{Problem 1:} {\it What is the Hausdorff dimension of $W\big(\mathcal{A},\psi,{\bf y}\big)$ ?
}

\bigskip

 Before stating our results precisely, we need to introduce some notation. 

 Let  $\Lambda_n:=A_n^{-1}\mathbb{Z}^d$, which is a lattice (see Definition \ref{lattice}) and $m_i(\Lambda_n),~1\le i\le d$ be the  successive minima of $\Lambda_n $  (see Definition \ref{minina}).
For $n\ge1$, put
$$
\tau_n:=-\frac{1}{n}\log\psi(n),\quad  l_{n,i}:=\frac{1}{n}\log\sigma_{i}(A_n)\quad {\rm and} \quad  h_{n,i}:=-\frac{1}{n}\log m_i(\Lambda_n),\quad 1\le i\le d.$$

For $1\le i\le d$, denote
\begin{equation*}
\underline{s}_n(\mathcal{A},\psi,i ):=\frac{1}{\tau_n+l_{n,i}}\Big(\sum_{j=1}^dl_{n,j}-\sum_{k\in \K_{n,2}(i)}(l_{n,j}-l_{n,i})
+\sum_{k\in\K_{n,1}(i)}(\tau_n+l_{n,i}-l_{n,j})\Big),
\end{equation*}

\begin{equation*}
\overline{s}_n(\mathcal{A},\psi,i ):=\frac{1}{\tau_n+l_{n,i}}\Big(\sum_{j=1}^dl_{n,j}-\sum_{j\in \K_{n,2}(i)}(l_{n,j}-l_{n,i})
+\sum_{j\in \Gamma_n(i)} (\tau_n+l_{n,i}-h_{n,j})\Big),
\end{equation*}
where
$$\K_{n,1}(i):=\{1\le j\le d:l_{n,j}>\tau_n+ l_{n,i}\}\quad \K_{n,2}(i):=\{1\le j\le d:l_{n,j}< l_{n,i}\}, $$
and
\begin{equation}\label{gammai}
\Gamma_n(i):=\{1\le j\le d:h_{n,j}\ge \tau_n+l_{n,i}\}. 
\end{equation}
Then put
\[\underline{s}(\mathcal{A},\psi):=\limsup_{n\to\infty}\min_{1\le i\le d}\{\underline{s}_n(\mathcal{A},\psi,i )\},\]
\begin{equation*}
\overline{s}(\mathcal{A},\psi):=\limsup_{n\to\infty}\min_{1\le i\le d}\{\overline{s}_n(\mathcal{A},\psi,i )\}.
\end{equation*}

Let us choose the common part of $\underline{s}_n(\mathcal{A},\psi,i )$ and $\overline{s}_n(\mathcal{A},\psi,i )$, and define it as $\hat{s}_n(\mathcal{A},\psi,i )$. More precisely,
 $$\hat{s}_n(\mathcal{A},\psi,i ):=\frac{1}{\tau_n+l_{n,i}}\Big(\sum_{j=1}^dl_{n,j}-\sum_{j\in \K_{n,2}(i)}(l_{n,j}-l_{n,i})\Big).$$
 By definitions, $\hat{s}_n(\mathcal{A},\psi,i )\ge \max\{\underline{s}_n(\mathcal{A},\psi,i ),\overline{s}_n(\mathcal{A},\psi,i )\}$. Put 
 \[\hat{s}(\mathcal{A},\psi):=\limsup_{n\to\infty}\min_{1\le i\le d}\{\hat{s}_n(\mathcal{A},\psi,i )\}.\]

 Let $\Gamma(\mathcal{A})$ denote the set of accumulation points $\bm{\ell}=(l_1,\dots, l_d)$ of sequences 
 $$\Big\{\Big(\frac{1}{n}\log \sigma_1(A_n),\dots,\frac{1}{n}\log \sigma_d(A_n)\Big)\Big\}_{n\in\N}.$$

Now we are fully ready to state our first theorem, which provides an estimate for  $\dim_{\rm H}W\big(\mathcal{A},\psi,{\bf y}\big)$, here $\dim_{\rm H}$ denotes the Hausdorff dimension.
For $s\ge0$, let $\mathcal{G}^{s}([0,1)^d)$ be the classes of sets in $[0,1)^d$ with large intersection property
with index $s$, and its formal definition will be given in  Section \ref{lippp}.

\begin{thm}\label{application1}
Let $\mathcal{A}\subset GL_d(\R)$ with $\Gamma(\mathcal{A}) \subset (\mathbb{R}^+)^d$. Let  $\psi : \R^+ \to \R^+$ be a positive and non-increasing function. 
Then  for any ${\bf y}\in [0,1)^d$,
\[\underline{s}(\mathcal{A},\psi)\le \dim_{\rm H} W\big(\mathcal{A},\psi,{\bf y}\big)\le \overline{s}(\mathcal{A},\psi),\]
and $W\big(\mathcal{A},\psi,{\bf y}\big)\in\mathcal{G}^{\underline{s}(\mathcal{A},\psi)}([0,1)^d)$. 
Moreover, if $\mathcal{A}$ is a sequence of  diagonal matrices, then 
\[\dim_{\rm H} W\big(\mathcal{A},\psi,{\bf y}\big)=\underline{s}(\mathcal{A},\psi) =\overline{s}(\mathcal{A},\psi).\]
\end{thm}

Throughout this paper, we denote by $\tau=\tau(\psi)$ the  lower order  at infinity of $\psi$, that is,
\begin{equation}\label{loip}
\tau=\tau(\psi):=\liminf_{n\to\infty}\frac{-\log\psi(n)}{n}.
\end{equation}

\begin{cor}\label{cor1}
As in the setting of Theorem \ref{application1}, for $\tau>\sup\limits_{\bm{\ell}=(l_1,\dots,l_d)\in \Gamma}\{\max\limits_{1\le i\le d}l_i-\min\limits_{1\le i\le d}l_i\}$, we have
\[
\dim_{\rm H}W\big(\mathcal{A},\psi,{\bf y}\big)= \overline{s}(\mathcal{A},\psi)= \underline{s}(\mathcal{A},\psi)=\hat{s}(\mathcal{A},\psi).
\]
\end{cor}

\begin{rem} 
The condition $\Gamma(\mathcal{A}) \subset (\mathbb{R}^+)^d$ means that 
\[0<\liminf_{n\to\infty}\frac{1}{n}\log\sigma_i(A_n)\le \limsup_{n\to\infty}\frac{1}{n}\log\sigma_i(A_n)<\infty,\quad 1\le i\le d. \]
Notice that
 $$0\le \max_i\big\{\limsup\limits_{n\to\infty}(h_{n,1}-\ell_{n,i})\big\}\le \sup_{\bm{\ell}\in \Gamma}\{\max_{1\le i\le d}l_i-\min_{1\le i\le d}l_i\}<\infty.$$ 
 When $\tau>\sup_{\bm{\ell}\in \Gamma}\{\max_{1\le i\le d}l_i-\min_{1\le i\le d}l_i\}$, it implies that for large $n$, we have  $\bigcup_{i=1}^d\Gamma_n(i)=\emptyset$, also $\bigcup_{i=1}^d\mathcal{K}_{n,1}(i)=\emptyset$. The different parts in the formula of $\underline{s}_n(\mathcal{A},\psi,i )$ and $\overline{s}_n(\mathcal{A},\psi,i )$ disappear, then  we obtain $\underline{s}_n(\mathcal{A},\psi,i )=\overline{s}_n(\mathcal{A},\psi,i )=\hat{s}_n(\mathcal{A},\psi,i )$. Then Corollary  \ref{cor1} follows from  Theorem \ref{application1}.
\end{rem}

\begin{rem}
For  $\mathcal{A}\subset GL_d(\mathbb{Z})$ with $\Gamma(\mathcal{A})\subset (\R^{+}\cup\{+\infty\})^d$.  Tan and Zhou \cite{tananzhou} proved that the sequence $(A_n{\bf x})_n$ is equidistributed for almost every ${\bf x}$ if  the  singular values $\inf_n\sigma_i(A_{n+1}/A_n),\,1\le i\le d$ are uniformly bounded below by some constant $C >1$. Under the same condition, one may deduce from \cite[Theorem 1.10]{tananzhou} that the Lebesgue measure of $W\big(\mathcal{A},\psi,{\bf y}\big)$ obeys 0-1 law according to  the convergence or divergence of $\sum_{n=1}^\infty\psi(n)^d$.
\end{rem}

\begin{rem}
Let $A \in {GL}_d(\mathbb{Z})$. If we take $A_n = A^n$, then $W(\mathcal{A}, \psi, \mathbf{y})$ is a shrinking target set. If instead $A_n = A^n - I$, then $W(\mathcal{A}, \psi, \mathbf{0})$ becomes a recurrence set.
\end{rem}

\begin{rem}
When $A$ is diagonalizable over $\mathbb{Z}$, Li et al. \cite{Lietal} investigated the Hausdorff dimension of $W\big((A^n)_n, \psi, \mathbf{y}\big)$. For further details on their proof strategy, we refer to \cite[Section 3.3.3]{Lietal}. It should be noted, however, that their method does not extend to the case where $A$ is diagonalizable over $\mathbb{Q}$. The  primary obstacle  arises because the operation $A^n \mathbf{x} \bmod 1$ does not coincide with $T^n \mathbf{x}$, where $A \in {GL}_d(\mathbb{R}) \setminus {GL}_d(\mathbb{Z})$ and $T \mathbf{x} = A \mathbf{x} \bmod 1$. In this situation, the set $W((A^n)_n, \psi, \mathbf{y})$ differs from both shrinking target and recurrence sets.

Theorem \ref{application1} establishes the Hausdorff dimension of $W((A^n)_n, \psi, \mathbf{y})$ when $A\in {GL}_d(\mathbb{R})$ is diagonal. This result is essential for deriving dimension formulas for both shrinking target sets and recurrence sets associated with integer matrices that are diagonalizable over $\mathbb{Q}$; see Theorem \ref{diag}.
\end{rem}

In Theorem \ref{application1}, we give an estimate on $\dim_{\rm H}W\big(\mathcal{A},\psi,{\bf y}\big)$, and moreover part indicates that the Hausdorff dimension $\dim_{\rm H}W\big(\mathcal{A},\psi,{\bf y}\big)$ is given by  $\underline{s}(\mathcal{A},\psi)=\overline{s}(\mathcal{A},\psi)$ when $A_n$ are diagonal for $n\in\mathbb{N}$. 
It is therefore natural to ask the following question:

\bigskip

\noindent\textbf{Problem 2:} {\it Does $\dim_{\rm H}W\big(\mathcal{A},\psi,{\bf y}\big)=\underline{s}(\mathcal{A},\psi)=\overline{s}(\mathcal{A},\psi)$ always hold?
}

\bigskip

However it does not always hold, as demonstrated by the following theorem. Let $\det A$ and $A^T$ denote the determinant and transpose of matrix $A$, respectively.

\begin{thm}\label{counterexample}
Assume that  $A\in GL_2(\mathbb{Z})$ satisfies $|\det A |=1$, $A^T=A$ and the modulus  of all  eigenvalues of $A$ are not equal to one. Let $\mathcal{A}=\{(\lambda A)^n\}_{n\in\N}$, where $\lambda$ is a nonzero integer for which the matrix  $\lambda A$ is expanding. 
Let  $\psi: \R^+ \to \R^+$ be a positive and non-increasing function.
Then  we obtain that for any ${\bf y}\in[0,1)^2$,  
\begin{equation}\label{upperbound}
\begin{split}
\dim_{\rm H}W\big(\mathcal{A},\psi,{\bf y}\big)&= \overline{s}(\mathcal{A},\psi)=\hat{s}(\mathcal{A},\psi).
\end{split}
\end{equation}
 Moreover $W\big(\mathcal{A},\psi,{\bf y}\big)\in\mathscr{G}^{ \overline{s}(\mathcal{A},\psi)}( [0,1)^2)$.
 \end{thm}

Theorem \ref{counterexample}  shows there exists a class of matrices such that 
 \[\dim_{\rm H}W\big((A^n)_{n\in\N},\psi,{\bf y}\big)= \overline{s}(\mathcal{A},\psi),\]
  which is the upper bound given in Theorem \ref{application1}, and  does not correspond to the lower bound in some cases, as Figure 1 shows. This demonstrates $ \underline{s}(\mathcal{A},\psi)$ given in Theorem \ref{application1} is not always the dimension formula of $W(\mathcal{A},\psi,\mathbf{y})$.  
 \begin{figure}[h]\label{upppicture}
    \begin{center}
      \includegraphics[scale=0.6]{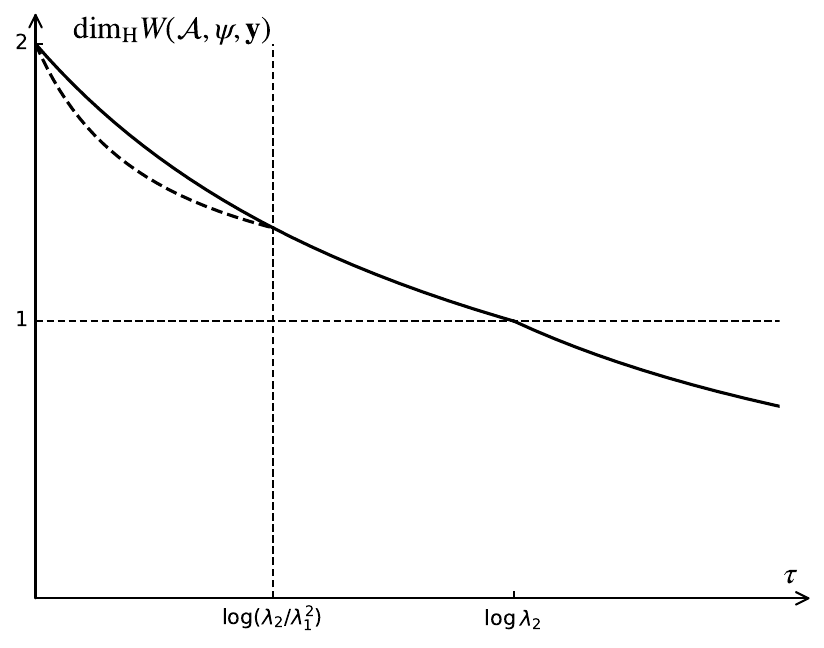}
    \end{center}
    \caption{The graph of $\tau \mapsto \overline{s}((A^n)_{n},\psi)$ in solid line for Theorem \ref{counterexample}, and the graph of $\tau \mapsto \underline{s}((A^n)_{n},\psi)$ is in dashed line, with $A=5\begin{bmatrix}2&1\\1&1\end{bmatrix}$.} 
  \end{figure}\label{fig3first}
Furthermore,  $\underline{s}(\mathcal{A},\psi)$ is derived by applying Theorem \ref{MTP1}, and then this example implicitly suggests that, at times,  the lower bound provided by the mass transference principle is not the accurate dimension formula.

We observe that in Corollary \ref{cor1} and Theorem \ref{counterexample} 
\begin{equation}\label{equality}
\begin{split}
\overline{s}(\mathcal{A},\psi)
= \hat{s}(\mathcal{A},\psi)
\end{split}
\end{equation}
Then one may conjecture that the equality \eqref{equality} always holds.
 The following example gives a negative answer, and this example is inspired by \cite[Example 7.2]{HPWZ}.

\begin{example}\label{exam1} Let $k\in \N$ and $A=\begin{bmatrix}10&5&0\\5&5&0\\0&0&k \end{bmatrix}$. The eigenvalues of $A$ are 
\[\frac{15-5\sqrt{5}}{2}=:\lambda_1,\quad \frac{15+5\sqrt{5}}{2}=:\lambda_2,\quad k=:\lambda_3.\]
Let $\mathcal{A}_1=(A^n)_{n\in\N}$ and $\psi:\mathbb{R}^+\to \mathbb{R}^+$ be a positive and
non-increasing function with the lower order at infinity $\tau$ defined in \eqref{loip}.  Write $l_i=\log\lambda_i,~1\le i\le 3$. Assume that  $k>\lambda_2^{3/2}$. Then 
\begin{equation*}
\begin{split}
\dim_{\rm H}W\big(\mathcal{A}_1,\psi,{\bf 0}\big)&
=\begin{cases}
\frac{\tau+3l_2}{\tau+l_2}\quad & \text{if~$\tau\le (l_2-l_1)/2$,}\\
\min\{\frac{\tau+2l_1+l_2}{\tau+l_1},\frac{\tau+3l_2}{\tau+l_2}\}\quad&\text{if~$ (l_2-l_1)/2<\tau\le l_3-l_2$,}\\
\min\{\frac{\tau+2l_1+l_2}{\tau+l_1},\frac{2l_2+l_3}{\tau+l_2},\frac{3l_3}{\tau+l_3}\}\quad&\text{if~$ l_3-l_2<\tau\le l_3-l_1$,}\\
\min\{\frac{l_1+l_2+l_3}{\tau+l_1},\frac{2l_2+l_3}{\tau+l_2},\frac{3l_3}{\tau+l_3}\}\quad &\text{if~$\tau\ge l_3-l_1$.}
\end{cases}\\
&=\overline{s}(\mathcal{A}_1,\psi),
\end{split}
\end{equation*}
which equals the upper bound given by Theorem \ref{application1}.
The proof of this example will be given in Section \ref{examples} later. 
\end{example}
\begin{rem}
In the above example, for simplify, we take  $\tau<l_2-2l_1$, then
$$ \dim_{\rm H}W\big(\mathcal{A}_1,\psi,{\bf 0}\big)=\overline{s}(\mathcal{A}_1,\psi)=\frac{\tau+3l_2}{\tau+l_2},$$
$$ \hat{s}(\mathcal{A}_1,\psi)=\min\Big\{\frac{3l_3}{\tau+l_3},\frac{2l_2+l_3}{\tau+l_2},\frac{l_1+l_2+l_3}{\tau+l_1}\Big\}=\frac{3l_3}{\tau+l_3},$$
and the lower bound  of $W\big(\mathcal{A}_1,\psi,{\bf 0}\big)$ given by Theorem \ref{application1} is  
$$ \underline{s}(\mathcal{A}_1,\psi)=\min\Big\{\frac{3l_3}{\tau+l_3},\frac{\tau+3l_2}{\tau+l_2},\frac{2\tau+3l_1}{\tau+l_1}\Big\}=\frac{2\tau+3l_1}{\tau+l_1}.$$

  \begin{figure}[h]\label{fig11}
    \begin{center}
      \includegraphics[scale=0.5]{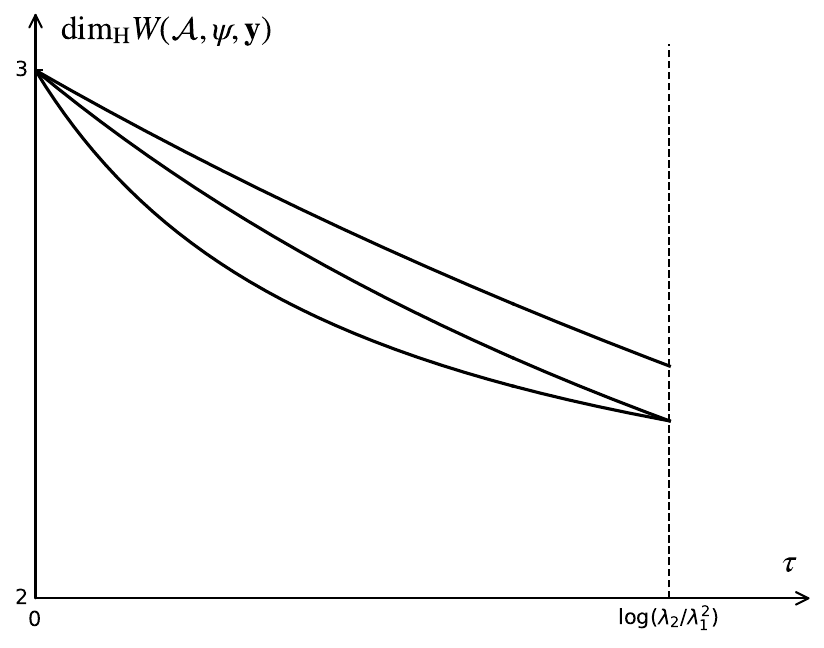}
    \end{center}
    \caption{The middle line is the graph of $\tau \mapsto \overline{s}(\mathcal{A}_1,\psi)$ for Example \ref{exam1} with $k=262$ and $\tau<l_2-2l_1$. The top line is the  graph of $\tau\mapsto  \hat{s}(\mathcal{A}_1,\psi)$, and  the lower one is the graph of $\tau\mapsto  \underline{s}(\mathcal{A}_1,\psi)$.} 
  \end{figure}

We refer to Figure 2 for a graph representing the value $\overline{s}(\mathcal{A}_1,\psi)$ of the Hausdorff dimension, along with the known bounds $  \hat{s}(\mathcal{A}_1,\psi)$ and $\underline{s}(\mathcal{A}_1,\psi)$ with $k=262$ and $\tau<l_2-2l_1$,  then  
we observe that 
\begin{equation}\label{newprevious}
\underline{s}(\mathcal{A}_1,\psi)<\dim_{\rm H}W\big(\mathcal{A}_1,\psi,{\bf 0}\big)= \overline{s}(\mathcal{A}_1,\psi)< \hat{s}(\mathcal{A}_1,\psi)
\end{equation}
holds.

The Hausdorff dimension of $W\big(\mathcal{A}_1,\psi,{\bf 0}\big)$ is $\overline{s}(\mathcal{A}_1,\psi)$ which is determined by involving the successive minima of the lattice $A^{-n}\mathbb{Z}^3$. The quantity $  \hat{s}(\mathcal{A}_1,\psi)$ is the known upper bound appearing in present literature, and \eqref{newprevious} shows that the corresponding upper bounds we obtain for shrinking target sets advance the existing results in the existing literature.
The successive minima characterize the structure of $A^{-n}\mathbb{Z}^3$, and this example also confirms that structural analysis is necessary in determining the Hausdorff dimension of shrinking target and recurrence sets. 
\end{rem}

\begin{rem}
In Theorem \ref{application1}, the lower bound $\underline{s}(\mathcal{A},\psi)$ for $W(\mathcal{A},\psi,{\bf y})$ is derived from the Mass Transference Principle (MTP), namely Theorem \ref{MTP1}. In many cases, this lower bound can be attained. However, as shown in Theorem \ref{counterexample} and Example \ref{exam1}, the value $\underline{s}(\mathcal{A},\psi)$ does not always represent the exact Hausdorff dimension, indicating that the lower bound provided by MTP may not always be the best.

More generally, when applying MTP to estimate $\dim_{\rm H}(\limsup\limits_{n\to\infty} R_n)$ for a sequence of rectangles $(R_n)_n$, it is often necessary to enlarge each $R_n$ in certain directions to form $\widetilde{R}_n$ such that the sequence $(\widetilde{R}_n)_n$ satisfies some structural conditions, thereby yielding a lower bound on the dimension. Interestingly, in the context of Theorem \ref{counterexample} or Example \ref{exam1}, one may alternatively enlarge $R_n$ in some directions while contracting it in others to obtain a modified rectangle $\hat{R}_n$ such that the corresponding condition holds, and leads to a different lower bound—which in these cases turns out to be the exact Hausdorff dimension of $\limsup\limits_{n\to\infty} R_n$.
\end{rem}

The starting point of the present work is an investigation into the Hausdorff dimension of the set $W\big(\mathcal{A},\psi,{\bf y}\big)$. As Theorem \ref{application1} establishes, this dimension ranges between $\underline{s}(\mathcal{A},\psi)$ and $\overline{s}(\mathcal{A},\psi)$. Furthermore, Theorem \ref{counterexample} and Example \ref{exam1} demonstrate the existence of  $\mathcal{A}$ for which $\dim_{\rm H}W\big(\mathcal{A},\psi,{\bf y}\big)>\underline{s}(\mathcal{A},\psi)$. Crucially, in all cases examined, the equality 
$\dim_{\rm H}W\big(\mathcal{A},\psi,{\bf y}\big)=\overline{s}(\mathcal{A},\psi)$ always holds. These observations motivate the following conjecture concerning the dimension of $W\big(\mathcal{A},\psi,{\bf y}\big)$.
\bigskip

\noindent\textbf{Conjecture:} {\it Let and $\mathcal{A} \subset GL_d(\R)$ with $\Gamma(\mathcal{A}) \subset (\R^+)^d$.  Let $\psi:\mathbb{R}^+\to \mathbb{R}^+$ be a positive and non-increasing function. Then  for any  ${\bf y}\in[0,1)^d$,
\[\dim_{\rm H}W\big(\mathcal{A},\psi,{\bf y}\big)=\overline{s}(\mathcal{A},\psi).\]}

This conjecture asserts that the Hausdorff dimension depends not only on the singular values of $A_n$ and 
 the lower order at infinity of $\psi$, but also on the structure of the lattices $A_n^{-1}{\mathbb{Z}^d}$. Various examples supporting the Conjecture are provided in Subsection \ref{example}  below.

\subsection{Applications to shrinking target sets and recurrence sets}\label{example} 

An application of Theorem \ref{application1} is to study the Hausdorff dimension of recurrence sets and shrinking target sets  for expanding matrix transformations. Here we give some related  notation.

Let  $A\in GL_d( \mathbb{Z})$ with the eigenvalues $\lambda_1,\lambda_2,\cdots,\lambda_d$. Let $T$ be the matrix transformations induced by $A$.  Given $\mathbf{y}\in [0,1)^d$, define the shrinking~ target ~set as
\[S(\psi , \mathbf{y}):=\big\{\mathbf{x}\in  [0,1)^d\colon T^n\mathbf{x}\in B(\mathbf{y},\psi(n))~{\rm  for~ infinitely ~many}~n\in\N\big\},\]
and the recurrence set is
\[R(\psi):=\big\{\mathbf{x}\in  [0,1)^d\colon T^n\mathbf{x}\in B(\mathbf{x},\psi(n))~{\rm  for~ infinitely ~many}~n\in\N\big\}.\]
For a comprehensive discussion of the background and metric properties of these sets, see Section  \ref{dda} below.

Recall that 
\begin{equation*}
\begin{split}
W((A^n)_n,\psi, \mathbf{y})&=\big\{\mathbf{x}\in  [0,1)^d\colon A^n\mathbf{x}\m\in B(\mathbf{y},\psi(n))~{\rm  for~ infinitely ~many}~n\in\N\big\}.
\end{split}
\end{equation*}
If $A\in GL_d(\mathbb{Z})$, define $T{\bf x}=A{\bf x}\pmod 1 $. Then $A^n{\bf x}\pmod 1= T^n{\bf x}$ holds for all $n\in\mathbb{N}$, and consequently, $W((A^n)_n,\psi, \mathbf{y})=S(\psi , \mathbf{y})$.

 On applying  Theorem \ref{application1} and the fact that $\overline{s}(\mathcal{A},\psi)\le \hat{s}(\mathcal{A},\psi)$, we have the following corollary.
\begin{cor}\label{cor}
  Let a matrix $A\in GL_d(\mathbb{Z})$ with all eigenvalues $\lambda_1, \lambda_2,\cdots,\lambda_d$ of modulus strictly larger than 1. Let $\psi:\mathbb{R}^+\to \mathbb{R}^+$ be a positive and non-increasing function. Then  for any  ${\bf y}\in[0,1)^d$,
 \[\underline{s}((A^n)_n,\psi)\le \dim_{\rm H}S(\psi , \mathbf{y})\le \overline{s}((A^n)_n,\psi)\le \hat{s}((A^n)_n,\psi),\]
and 
\[\underline{s}((A^n)_n,\psi)\le \dim_{\rm H}R(\psi)\le \overline{s}((A^n-I)_n,\psi)\le \hat{s}((A^n)_n,\psi). \]
\end{cor}

\begin{rem} Let $A$ be as given in Corollary \ref{cor} and assume that $1<|\lambda_1|\le \dots\le |\lambda_d|$ for simplicity. 
It is worth pointing out that $\underline{s}((A^n)_n,\psi)$ and $\hat{s}((A^n)_n,\psi)$ are determined by the singular values of $A^n$ and $\psi$. In all related work addressing  $S(\psi , \mathbf{y})$ and $R(\psi)$ under the same setting, the dimension formulae are given by eigenvalues of $A$ and $\psi$. Crucially, $\underline{s}((A^n)_n,\psi)$ and $\hat{s}((A^n)_n,\psi)$ coincide with corresponding formulae given by eigenvalues in prior related work. This equivalence follows  from Lemma \ref{supelimsup}, which establishes that
\begin{equation}\label{firstequal}
\underline{s}((A^n)_n,\psi)=\underline{s}((A^n-I)_n,\psi)=\min_i\left\{\frac{il_i}{\tau+l_i}-\sum_{j:l_j>\tau+l_i}\frac{l_j-l_i-\tau}{\tau+l_i}+\sum_{j>i}\frac{l_j}{\tau+l_i}\right\},
\end{equation}
and
\begin{equation}\label{secondequal}
\hat{s}((A^n)_n,\psi)=\hat{s}((A^n-I)_n,\psi)=\min_i\left\{\sum_{j=1}^d\frac{l_j}{\tau+l_i}-\sum_{j\in \K_2(i)}\frac{l_j-l_i}{\tau+l_i}\right\},
\end{equation}
where $l_i=\log|\lambda_i|$, $1\le i\le d$.
\end{rem}

\begin{rem}
Let $A$ be as given in Corollary \ref{cor}, Hu et. al \cite{HPWZ}  showed that 
 \begin{equation*}
\underline{s}((A^n)_n,\psi)\le \dim_{\rm H}S(\psi , \mathbf{y})\le \hat{s}((A^n)_n,\psi).
\end{equation*}
The lower bound coincides with that established in Corollary \ref{cor}, whereas the same corollary yields a novel upper bound that improves the result of \cite{HPWZ}, as demonstrated in Example \ref{exam1}.
\end{rem}

We will further investigate the dimension formula of $S(\psi , \mathbf{y})$ and $R(\psi)$.
Theorems \ref{diag}--\ref{jordan} and Corollary \ref{jorin}  substantiate the Conjecture, and in these cases,
\[\dim_{\rm H}S(\psi , \mathbf{y})=\underline{s}((A^n)_n,\psi)=\overline{s}((A^n)_n,\psi).\]
We remark that we only consider the Hausdorff dimension of $S(\psi , \mathbf{y})$ in these statements and proofs, however one can deduce that Theorems \ref{diag}--\ref{jordan} and Corollary \ref{jorin} also hold for the recurrence set  $R(\psi)$ with the same setting.

\begin{thm}\label{diag}
Let a matrix $A\in GL_d(\mathbb{Z})$ be diagonalizable over $\mathbb{Q}$, and all eigenvalues $\lambda_1,\,\lambda_2,\dots,\lambda_d$ are of modulus strictly larger than 1. Let $\psi:\mathbb{R}^+\to \mathbb{R}^+$ be a positive and non-increasing function. Then  for any  ${\bf y}\in[0,1)^d$,
\begin{eqnarray*}
     \dimh S(\psi , \mathbf{y})= \underline{s}((A^n)_n,\psi)=\overline{s}((A^n)_n,\psi).
 \end{eqnarray*}

\end{thm}


Let $d\in\N$ and $A$ be a $d$-dimensional Jordan standard form, that is $A=\mathrm{diag}(J(\lambda_1,n_1),\cdots,J(\lambda_s,n_s))$ where $\lambda_i\in\R$, $n_j\in\N$ for all $j=1,2,\cdots,s$ and $n_1+\cdot\cdot\cdot+n_s=d$. Here $J(\lambda,n)$ denote a $n\times n$ Jordan block
\begin{eqnarray*}
\begin{pmatrix}
\lambda & 1 & 0 & \ldots & 0\\
0 & \lambda & 1 & \ldots & 0\\
\vdots & \vdots & \ddots & \ddots & \vdots\\
0 & 0 & \ldots & \lambda & 1 \\
0 & 0 & \ldots  & 0 & \lambda
\end{pmatrix}_{n\times n}.
\end{eqnarray*}
We say that $A$ is expanding, if $|\lambda_j|>1$ for all $j=1,\cdots,s$. 
The Hausdorff dimension of $W((A^n)_n,\psi, \mathbf{y})$ is determined in the following theorem.
\begin{thm}\label{jordan}
    Let $A$ be a $d$-dimensional expanding Jordan standard form, 
    and $W(\psi)$ be the set which are described as above. Let $\psi:\mathbb{R}^+\to \mathbb{R}^+$ be a positive and non-increasing function. Then for any  ${\bf y}\in[0,1)^d$,
 \begin{eqnarray*}
     \dimh W((A^n)_n,\psi, \mathbf{y})= \underline{s}((A^n)_n,\psi)=\overline{s}((A^n)_n,\psi)=
     \min_{1\leq i\leq s}s( i),
 \end{eqnarray*}
 where 
 \begin{eqnarray*}
    s( i)=\sum_{j\in\K'_1(i)}n_j+\sum_{j\in\K'_2(i)}\frac{n_j\log|\lambda_i|}{\log|\lambda_i|+\tau}+\sum_{j\in\K'_3(i)}\frac{n_j\log|\lambda_j|}{\log|\lambda_i|+\tau},
\end{eqnarray*}
and
\begin{eqnarray*}
\K'_1(i)&=&\left\{1\leq j\leq s:\log|\lambda_j|>\log|\lambda_i|+\tau\right\},\\[4pt]
\K'_2(i)&=&\left\{1\leq j\leq s:|\lambda_j|<|\lambda_i|\right\},\\[4pt]
\K'_3(i)&=&\{1,\cdots,s\}\backslash(\K_1(i)\cup\K_2(i)).
\end{eqnarray*}

\end{thm}
Theorem \ref{jordan} together with the strategy of Theorem \ref{diag} gives the following result which establishes $\dimh S(\psi)$.
\begin{cor}\label{jorin}
Let $A\in GL_d(\mathbb{Z})$ be an expanding matrix. Assume that there exist 
 $P\in GL_d( \mathbb{Q})$ and a $d$-dimensional Jordan standard form $\widetilde{A}$ such that $A=P^{-1}\widetilde{A}P$. Suppose that 
$$\widetilde{A}=\mathrm{diag}(J(\lambda_1,n_1),\cdots,J(\lambda_s,n_s)),$$
 where $\lambda_i\in\R$, $n_i\in\N$ for all $i=1,2,\cdots,s$ and $n_1+\cdot\cdot\cdot+n_s=d$. Let $\psi:\mathbb{R}^+\to \mathbb{R}^+$ be a positive and non-increasing function. Then for any ${\bf y}\in[0,1)^d$,
 \[\dim_{\rm H}S(\psi , \mathbf{y})= \underline{s}((A^n)_n,\psi)=\overline{s}((A^n)_n,\psi)= \min_{1\le i\le s}s(i).\]
\end{cor}

\subsection{Connection to the previous  works}\label{compare}
 We will present several previous works that are related to our main results. 
 
\subsubsection{Dynamical Diophantine approximation}\label{dda}
In a dynamical system, dynamical Diophantine approximation investigates the quantitative properties of the distribution of orbits. More precisely, it focuses on the size of the limsup sets defined by the orbits in view of the measure-theoretic and dimensional characteristics. The quantitative study is motivated by the density of orbits and its connections to the classical Diophantine approximation.

Let $(X,\mu, T)$ be a metric probability measure preserving system and $T$ be ergodic with respect to $\mu$. Then given any ball $B$ in $X$ of positive measure, there is a direct consequence of Birkhoff's ergodic theorem that the subset 
\[\{x\in X\colon T^nx\in B~{\rm for~infinitely~many~}n\in\N\}\]
has full $\mu$-measure.   The classic shrinking target problems introduced by  Hill and Velani \cite{HL95} focus on `what happens if $B$ shrinks with time $n$?'. Let  $\psi: \R^+ \to \R^+$ be a positive and non-increasing function. Given a point $y$ in $X$, define the shrinking target set as
\[S(\psi)=S(\psi , y):=\big\{x\in X\colon T^nx\in B(y,\psi(n))~{\rm  for~ infinitely ~many}~n\in\N\big\}.\]
The dimension results for the sets are studied in several concrete dynamical systems, see \cite{AP, LWWX, BR, Ur, Huw} and references therein. References on the measure theory include \cite{AB, LLSV, KB, Ph, CK, GK} for instances.

Assume that $A\in GL_d(\mathbb{R})$  with eigenvalues $1<|\lambda_1|\le \dots\le |\lambda_d|$. We now focus on the matrix transformation $T$ on $[0,1)^d$ induced by $A$:
\[T{\bf x}=A{\bf x}\pmod 1.\]
 Recall that $\tau$ is the lower order at infinity of the function $\psi$.
 
\begin{itemize}

\item When $A\in GL_d(\mathbb{Z})$, Hill and Velani \cite{HL99} proved that for  $\tau>\log|\lambda_d/\lambda_1|$,
 \begin{equation*}
 \begin{split}
 \dim_{\rm H}S(\psi)&=\min_{1\le i\le d}\left\{\frac{i\log|\lambda_i|+\sum_{j=i+1}^d\log|\lambda_j|}{\tau+\log|\lambda_i|}\right\}\\&\overset{\eqref{secondequal}}{=}\hat{s}((A^n)_n,\psi)=\underline{s}((A^n)_n,\psi)=\overline{s}((A^n)_n,\psi),
 \end{split}
 \end{equation*}
where the last two equalities follow from Corollary \ref{cor1}. 
\item Li et al. \cite{Lietal} investigated the $\dim_{\rm H}S(\psi)$ for expanding diagonal real matrices and integer matrices which are diagonalisable over $\mathbb{Z}$, and the dimension formula is 
\begin{equation}\label{cformula}
\begin{split}
\dim_{\rm H}S(\psi)=&
\min_{1\le i\le d}\Big\{\frac{i\log|\lambda_i|-\sum_{j\colon |\lambda_j|>|\lambda_ie^\tau|}(\log|\lambda_j/\lambda_i|-\tau)+\sum_{j>i}\log|\lambda_j|}{\tau+\log|\lambda_i|}\Big\}\\
&\overset{\eqref{firstequal}}{=}\underline{s}((A^n)_n,\psi).
\end{split}
\end{equation}
\item In a recent work,  Hu et al. \cite{HPWZ} estimated $\dim_{\rm H}S(\psi)$. They showed that 
 \begin{equation}\label{estimate}
\underline{s}((A^n)_n,\psi)\le \dim_{\rm H}S(\psi)\le \hat{s}((A^n)_n,\psi).
\end{equation}
 \end{itemize}
 
Most of above results provide the same dimension formula  $\underline{s}((A^n)_n,\psi)$ for different types of matrices, and further indicate that the formula of $\dim_{\rm H}S(\psi)$ is determined by the eigenvalues of $A$ and the rate of $\psi(n)\to 0$.  It is natural to ask whether the dimension formula in \eqref{cformula}  is true for all expanding matrix transformations. 
 We now proceed to make a precise statement in the following.

 
\begin{claim}
 Let $A$ be an expanding matrix with eigenvalues $\lambda_1,\cdots,\lambda_d$. 
  Then for any ${\bf y}\in[0,1)^d$,
 \[\dim_{\rm H} S(\psi)=\min_{1\le i\le d}\Big\{\frac{i\log|\lambda_i|-\sum_{j\colon |\lambda_j|>|\lambda_ie^\tau|}(\log|\lambda_j/\lambda_i|-\tau)+\sum_{j>i}\log|\lambda_j|}{\tau+\log|\lambda_i|}\Big\},\]
 where  $\tau$ is the lower order at infinity of the function $\psi$.
  \end{claim}
   In 9
   ection \ref{example}, we present several examples showing that the claim holds true for 
 a large class of matrices. However it turns out that this claim is generally false.  In fact, we give some counter-examples for which the dimension formulae are not the expected one in this claim, see Theorem \ref{counterexample} and Example \ref{exam1}.

  There is another interesting topic closely related to the shrinking target problem, which is the so-called quantitative recurrence problem. 
It originates from the famous Poincar\'e recurrence theorem, which states that when $(X,\mu, T)$ is a metric probability measure preserving system and $X$ is separate, almost all points return close to themselves infinitely many times under $T$.  
A natural question is to quantify the rate of recurrence for general dynamical systems. In particular, one  focuses on  the
set  of points that return with the rate $\psi$,
\[R(\psi)=\big\{x\in X\colon T^nx\in B(x,\psi(n))~{\rm  for~ infinitely ~many}~n\in\N\big\}.\]
One of the earliest results regarding the rate of recurrence is due to Boshernitzan \cite{BO}. 
Since then lots of papers have appeared in establishing measure and dimension results of such sets in various dynamical systems. To name but a few, see \cite{KB,KKP,CWW,BF, vbh} for measure results, and see also \cite{TW,SW,HP} for dimension results. 
When $T$ is an expanding matrix transformation induced by $A$, He and Liao \cite{HeL} calculated $\dim_{\rm H}R(\psi)$ for the same matrices as  Li et al. \cite{Lietal} did. In \cite{HL2024}, if $\tau>\log|\lambda_d/\lambda_1|$, Hu and Li calculated the dimension formula of $R(\psi)$, which can be also deduced from Corollary \ref{cor1}.

Typically, the shrinking target sets $S(\psi)$ and recurrence sets $R(\psi)$ are studied independently. The existing results indicate that in many dynamical systems, especially for matrix transformations, the dimension formulae for these two sets are the same. Consequently, it is natural to consider whether it is possible to investigate the dimension unifying the shrinking targets and recurrence. In view of the existing results,  the dimensions of $S(\psi),~R(\psi)$ were unified in some dynamical systems by replacing the inhomogeneous part $y$ in $S(\psi)$ and $x$ in $R(\psi)$  by a Lipschitz function  \cite{Kz, wu,lvww}. In this paper, we will unify these two sets in the following way, while considering the matrix transformations $T$ on $[0,1)^d$. 
Given our focus on matrix transformations, we adopt the Euclidean distance $d$. Then define 
\[T_n=\begin{cases}T^n-\mathbf{y}\\T^n-I\end{cases},\]
for $\mathbf{y}\in [0,1)^d$, and this transformation provides a unified approach to studying the dimensional properties of
$S(\psi),R(\psi)$ via the following set
\[\big\{\mathbf{x}\in [0,1)^d\colon T_n(\mathbf{x})\in B\big(\mathbf{0}, \psi(n)\big)~{\rm  for~ infinitely ~many}~n\in\N\big\}.\] 
Therefore, our main results can be applied to the study of both shrinking target sets and recurrence sets. Let $T$ be induced by $A\in GL_d(\mathbb{Z})$ with all eigenvalues of modulus strictly larger than 1.  Corollary \ref{cor1} covers the main result of \cite[Theorem 1]{HL99}, and \cite[Theorem 1.2]{HL2024}. Theorem \ref{application1} can give a better bound for $\dim_{\rm H}S(\psi)$. 
Moreover, we extend \cite[Theorem 8]{Lietal} and \cite[Theorem 1.5]{HL2024} for any matrix diagonalizable over 
$\mathbb{Q}$ (see Theorem \ref{diag}) and to matrices of Jordan standard form (see Corollary \ref{jorin}).
 Theorem \ref{counterexample}  and Example \ref{exam1} show that both the upper bound and the lower bound given in \eqref{estimate} can not be used as a unified dimension formula of $S(\psi)$. As a matter of fact, we propose a
conjecture for the dimension that $\dim_{\rm H}S(\psi)=\overline{s}((A^n)_n,\psi),\, \dim_{\rm H}R(\psi)=\overline{s}((A^n-I)_n,\psi)$.

\subsubsection{Mass Transference Principle}


In Diophantine approximation, dynamical systems, and fractal geometry and so on, the limsup sets are of great importance. Given a sequence of sets $(E_n)_{n\in\N}$, define the limsup set 
\[\limsup_{n\to\infty}E_n:=\bigcap_{i=1}^\infty\bigcup_{i=n}^\infty E_n=\{x\colon x\in E_n~\text{for infinitely many } n\in\N\}.\]
A fundamental and extremely valuable tool in determining the Hausdorff dimension of $\limsup\limits_{n\to\infty}E_n$ is the Mass Transference Principle. The original Mass Transference Principle was introduced by Beresnevich and Velani \cite{bv}, which focused on a locally compact set $X\subset \R^d$  endowed with an Ahlfors $\delta$-regular measure, and $E_n$, $n\in\N$ are balls. It allows for the transference of a $\mathcal{H}^\delta$ -measure statement for a limsup set of balls which are obtained by enlarging $E_n$ in a certain way, to a $\mathcal{H}^s$ -measure statement for $0\le s\le \delta$, thereby offering a powerful method for dimension estimation of $\limsup\limits_{n\to\infty}E_n$.
 
 The Mass Transference Principle has been very extensively studied, and 111
 a comprehensive overview can be found in the surveys \cite{AD, DT19}. A major advancement was made by Wang and Wu \cite{WW}, who extended the Mass Transference Principle to limsup sets defined by rectangles with sides parallel to the coordinate axes. This generalization greatly broadened the applicability of the method, particularly in areas such as multiplicative Diophantine approximation, linear forms, and the dimensional theory of limsup sets defined by rectangles. Further extending the work of \cite{WW}, Li et al. \cite{LLVWZ} essentially developed a version applicable to unbounded settings—specifically, when the lower order at infinity of the side lengths of the rectangles is infinite in some direction. This extension enabled applications not only to the completion of the dimensional theory in classical simultaneous Diophantine approximation, but also to dynamical settings such as shrinking target problems. 
 
Despite these substantial developments, existing versions of the Mass Transference Principle remain insufficient for treating limsup sets defined via families of rectangles that allow rotations—a key feature in the main results of the present paper. To overcome this limitation, we introduce a further generalization of the Mass Transference Principle, building on the framework of \cite{LLVWZ}. Our main results, Theorems \ref{MTP1} and \ref{MTP2}, are valid under certain ubiquity and full measure assumptions, and apply to limsup sets formed by rectangles subjected to local isometries. Here, by `local isometries', we mean transformations that preserve distance locally, such as rotations and translations on torus. This extension provides two key innovations:

 (1)\, First, it encompasses not only the classical ball and aligned rectangle settings, but also handles the previously inaccessible case of rotated rectangles, thereby addressing a key limitation of previous work.
 
 (2)\, Second, our method simultaneously establishes that the resulting limsup sets enjoy the large intersection property, as introduced by Falconer \cite{falconer}. Such sets are the so-called $\mathcal{G}^s-$sets.  Negreira and Sequeira \cite{NS} reproduced the characterization of $\mathcal{G}^s-$sets from the Euclidean setting  \cite{falconer} to the metric space of homogeneous type $(X,\mu)$. The s-dimensional class of sets with large intersections is denoted by $\mathcal{G}^s(X)$, which consists of sets of Hausdorff dimension at least $s$. 

We thereby introduce a comprehensive framework for the analysis of a broad class of limsup sets, permitting the inclusion of rectangles via local isometries, which is directly applicable to related problems in dynamical systems.

\subsubsection{Sets arising from one-dimensional Diophantine approximation}
It is now timely to make a comparison among the main results in the case where $d=1$ and some related results about the set arising from one-dimensional Diophantine approximation.

Let $(a_n)_{n\in\N}$ be a sequence of real numbers.
Given $y\in [0,1)$, we consider
\begin{equation*}
W\big((a_n)_n,\psi,y\big)
=\big\{x\in [0,1) \colon a_nx\m\in B(y,\psi(n))~{\rm  for~ infinitely ~many}~n\in\N\big\}.
\end{equation*}

Denote
\[\alpha=\alpha((a_n)_n):=\limsup_{n\to\infty}\frac{\log |a_n|}{n}.\]

The corollary follows  by applying Theorem \ref{application1}.
\begin{cor}\label{cor2}
Let $(a_n)_{n\in\N}$ be a  sequence of real numbers with $|a_n|\ge1,n\ge1$ and $\alpha\in(0,\infty]$. Let $\psi: \mathbb{N}\to \mathbb{R}^+$ be a positive and non-increasing function.
 Then 
\[\dim_{\rm H} W\big((a_n)_n,\psi,y\big) =\limsup_{n\to\infty}\frac{\log|a_n|}{\log|a_n|-\log \psi(n)}.\]

\end{cor}

\begin{rem}\label{rem6}
If $\alpha\in [0,\infty]$, by Theorem \ref{thm5}, we have that
\[\dim_{\rm H} W\big((a_n)_n,\psi,y\big) \ge\limsup_{n\to\infty}\frac{\log|a_n|}{\log|a_n|-\log \psi(n)}.\]
\end{rem}

Now we make a comparison between Corollary \ref{cor2} and a result of  Li, Li and Wu \cite{llwu}. 
Let $W(\mathbf{y}, \psi) $ be as defined in \eqref{das} with $d=1$.  Li, Li and Wu \cite{llwu} showed that for $(a_n)_{n\in\N}\subset \mathbb{N}$,
\[\dim_{\rm H} W(\mathbf{y}, \psi)=\min\left\{1,\inf\left\{s:\sum_{n\ge1}a_n\left(\frac{\psi(n)}{a_n}\right)^s<\infty\right\}\right\}.\]
In this case, if  $\alpha\in(0,\infty]$, one can deduce from Corollary \ref{cor2} that 
\[\dim_{\rm H} W(\mathbf{y}, \psi) =\limsup_{n\to\infty}\frac{\log a_n}{\log a_n-\log \psi(n)},\] 
which is the same as the formula given by \cite{llwu}, since after some calculations, we have
\[ \limsup_{n\to\infty}\frac{\log a_n}{\log a_n-\log \psi(n)}=\inf\left\{s:\sum_{n\ge1}a_n\left(\frac{\psi(n)}{a_n}\right)^s<\infty\right\}\le 1.\]
In fact, Corollary \ref{cor1}  can deal with any sequence of real numbers $\vert a_n\vert\ge1,n\ge1$ with $\alpha>0$. But  Corollary \ref{cor1}  is not always better than that \cite{llwu}.
For example $\psi(n)=1/n$ and $a_n=n$, that is $\alpha=0$, one may deduce from \cite{llwu} that  $\dim_{\rm H} W\big((a_n)_n,\psi,y\big)=1/2$, and from Remark \ref{rem6} that 
 $$1\ge \dim_{\rm H} W\big((a_n)_n,\psi,y\big)\ge 1/2.$$
Thus, when $\alpha=0$, the result of  \cite{llwu} gives an exact value, whereas our Remark 9 only provides bounds. However the formulae coincide when $\alpha>0$.


\subsection{The organization of this paper}
The rest of this paper is organised as follows: We begin by listing our results on the mass transference principle in Section \ref{MTP}, namely Theorems \ref{MTP1} and \ref{MTP2}, which are general results for rectangles under local isometries and used to derive the lower bound on the Hausdorff dimension in most of the results. 
In Section \ref{preliminaries}, we recall some useful preliminaries. Theorems  \ref{MTP1} and \ref{MTP2} are proven in Section \ref{proofmtp}. 
The proof of our main result, Theorem \ref{application1}, is presented in full in Section \ref{proofof28}.
Section \ref{proof13} is dedicated to the proofs of Theorem \ref{counterexample} and Example \ref{exam1}. Finally, in Section \ref{examples}, we prove the results given in Section \ref{example}.

\section{Mass Transference Principle}\label{MTP} 

 Before stating our results formally, we require some preliminaries with the framework.  Given a metric space $(X,\rho_X)$, we use $B_{\rho_X}(x,r)$ to denote an open ball centered at $x\in X$ with radius $r>0$. In the following, we often write $\rho$ or $B(x,r)$ in place of  $\rho_X$ or $B_{\rho_X}(x,r)$  if the associated metric space is explicit in the context. 
 
 { Let $d\in\N$ be an integer}.
For any $i = 1, . . . , d$, {let $p_i$ be a positive integer} and  let $X_i\subseteq \R^{p_i}$, $Y_i\subseteq\R^{p_i}$ be  subsets equipped with  metrics $\rho_{X_i}$, $\rho_{Y_i}$  on $X_i$, $Y_i$ respectively. {We mention that the metric $\rho_{X_i}$ (or $\rho_{Y_i}$) is not necessarily compatible with the usual topology on $X_i$ (or $Y_i$) induced by the Euclidean distance.}  
Throughout the section, assume that  $(X_i,\rho_{X_i})$ and $(Y_i,\rho_{Y_i})$ are compact metric spaces equipped with $\delta_i$-Ahlfors regular measures $\mu_i$ and $\nu_i$ respectively; that is to say that there exist constants $0\le c_1\le 1\le c_2<\infty$ and $r_0>0$ such that 
\[c_1r^{\delta_i}\le \mu_i(B({\bf x},r))\le c_2r^{\delta_i},\quad c_1r^{\delta_i}\le \nu_i(B({\bf y},r))\le c_2r^{\delta_i}\]
for all balls $B({\bf x},r)\subseteq X_i$ and $B({\bf y},r)\subseteq Y_i$ with $r\le r_0$. Given $s>0$ and a ball $B=B({\bf x},r)$ in $X_i$ (or a ball $B=B({\bf y},r)$ in $Y_i$), we define 
\[B^s:=B({\bf x},r^{\frac{s}{\delta_i}})\qquad\text{(or $B^s:=B(\bfy,r^{\frac{s}{\delta_i}})$)}.\]Let $X=\prod_{i=1}^dX_i$, $Y=\prod_{i=1}^dY_i$ and $\mu=\prod_{i=1}^d\mu_i$, $\nu=\prod_{i=1}^d\nu_i$. The distance between  ${\bf x}=({\bf x}_1,\dots,{\bf x}_d)$ and $ {\bf y}=({\bf y}_1,\dots,{\bf y}_d)\in X$ is defined as
\begin{equation}\label{defofrhoinx}
    \rho_X({\bf x},{\bf y})=\left(\sum_{i=1}^d\rho_{X_i}({\bf x}_i,{\bf y}_i)^2\right)^{\frac{1}{2}}.
\end{equation}
We define the metric $\rho_Y$ on $Y$ similarly.

 The following notion of local isometry between metric spaces is required in the statements of main results in this section.  
 
\begin{definition}
Let $(Y,\rho_Y)$ and $(X,\rho_X)$ be metric spaces. Let $\varepsilon_0>0$. We say that a map $f:Y\to X$ is an $\varepsilon_0$-local isometry if there exists $\varepsilon_0>0$ such that for all $ y\in Y$ the map $$f|_{B(y,\varepsilon_0)}:B(y,\varepsilon_0)\to B(f(y),\varepsilon_0)$$ 
is a homeomorphism and
\begin{eqnarray}\label{locisodef}
    \rho_X\left(f(x) , f(y)\right)=\rho_Y\left(x , y\right)\qquad\forall\ x\in B(y,\epsilon_0).
\end{eqnarray}
We simply say that $f:Y\to X$ is a local isometry if there exists $\varepsilon_0>0$ such that $f$ is an $\varepsilon_0$-local isometry.
\end{definition}
 
Next, we introduce the definitions of dimensional numbers   that are also needed in the statement of our Mass Transference Principle. Put $\bm{\delta}=(\delta_1,\dots,\delta_d)$. Let ${\bf u}=(u_1,\dots,u_d)\in(\R^+)^d$ and ${\bf v}=(v_1,\dots,v_d)\in(\R^+\cup\{+\infty\})^d$  be vectors such that
\[u_i\le v_i,\quad 1\le i\le d.\]
Denote
\[\mathcal{L}({\bf v}):=\{1\le i\le d\colon v_i<+\infty\}, \quad \mathcal{L}_\infty({\bf v}):=\{1\le i\le d\colon v_i=+\infty\}.\]
Next, we define the value $s({\bf u},{\bf v},i)$ depending on whether $i\in \mathcal{L}(\bf v)$  {or} $\mathcal{L}_\infty(\bf v)$.
\begin{itemize}
\item  For each $i\in \mathcal{L}({\bf v})$, put
\[s({\bf u},{\bf v},i)= \sum_{k\in\K_1(i)}\delta_k+\sum_{k\in\K_2(i)}\delta_k\Big(1-\frac{v_k-u_k}{v_i}\Big)+\sum_{k\in\K_3(i)}\frac{\delta_ku_k}{v_i},\]
where
\[\K_1(i):=\{1\le k\le d\colon u_k>v_i\}, \quad \K_2(i):=\{k\in \mathcal{L}({\bf v})\colon v_k\le v_i\},\]
and
\[\K_3(i):=\{1,\dots,d\}\setminus \left(\K_1(i)\cup \K_2(i)\right).\vspace{1ex}\]
\item  For $i\in \mathcal{L}_\infty({\bf v})$, put
\[s({\bf u},{\bf v},i):=\sum_{k\in \mathcal{L}({\bf v})}\delta_k.\]
\end{itemize}
With the above notation in mind, we then define the \textit{dimensional number} $s({\bf u},{\bf v})$ as
\begin{equation*}\label{dimensionquantity}
s({\bf u},{\bf v}):=\min_{1\le i\le d}\{s({\bf u},{\bf v},i)\},
\end{equation*}

 In our Mass Transference Principle, we mainly focus on a sequence of balls    $(B_{i,n})_{n\in\N}$    in $Y_i$  ($1\leq i\leq d$) that satisfy the following statement: \textit{there exist sequences of positive real numbers $(r_n)_{n\in\N}$ and $(v_{i,n})_{n\in\N}$ such that for any  $1\le i\le d$  and $n\in\N$,
\begin{eqnarray}\label{conofbandv}
    r(B_{i,n})=r_n^{v_{i,n}},\quad \lim_{n\to\infty}r_n=0\quad\text{and}\quad \lim\limits_{n\to\infty}v_{i,n}=v_i\in \mathbb{R}^+\cup\{+\infty\}.
\end{eqnarray}
}Write ${\bf v}(n)=(v_{1,n},\dots,v_{d,n})$ for any $n\in\N$ and ${\bf v}=(v_1,\dots,v_d)$.
For $n\ge1$, let ${\bf u}(n)=(u_{1,n},\dots,u_{d,n})\in (\R^+)^d$ and ${\bf u}=(u_1,\dots,u_d)\in (\R^+)^d$  satisfy 
\begin{equation}\label{conditiononu}
\lim_{n\to\infty} u_{i,n}= u_i\quad and \quad u_{i,n}\le v_{i,n}.
\end{equation}
Then,  we denote
\begin{equation}\label{defofsinudv}
    s_{i,n}:=\frac{u_{i,n}\delta_i}{v_{i,n}}
\end{equation}
 for each $n\ge 1$ and $1\le i\le d$. We are now ready to state two main theorems of the section as follows.

\subsection{Mass Transference Principle  under ubiquity}

{
Building on the notion of local ubiquity condition for balls introduced by Beresnevich, Dickinson and Velani \cite{bere}, Wang and Wu \cite{WW} defined a corresponding concept of local ubiquity for rectangles.  The following presents a slight modification of notion `` local ubiquity for rectangles'' given by \cite{WW}.}

\begin{definition}[{\bf Local ubiquity system}]

 Let $(k_n)_{n\in\N}$ be a sequence of positive integers. We say that a sequence of  {sets} $(R_{n,j})_{n\in\N,1\leq j\leq k_n}$  in $Y$ is a local ubiquity system with respect to  local isometries $(f_n:Y\to X)_{n\in\N}$  if  there exists a constant $c > 0$ such that for any ball $B\subset X$,
 \begin{equation}\label{condition:measure}
 \limsup_{n\to\infty}\mu\left(\bigcup_{1\le j\le k_n }f_n(R_{n,j})\cap B\right)>c\mu(B).
 \end{equation}
\end{definition}
{For ${\bf x}\in  Y$  (or $X$) and a set $E\subset Y$ (or $X$),  throughout we use the notation ${\bf x}+E:=\{{\bf x}+{\bf y}\colon {\bf y}\in E\}$.}
\begin{thm} \label{MTP1} Let  $(B_{i,n})_{1\leq i \leq d,n\in\N}$, $({\bf v}(n))_{n\in\N}$ and ${\bf v}$ satisfy \eqref{conofbandv}. Let $({\bf u}(n))_{n\in\N}$ and ${\bf u}$ satisfy \eqref{conditiononu}. Let $(k_n)_{n\in\N} $ be a non-decreasing sequence of positive integers and let $({\bf x}_{n,j})_{n\in\N,1\leq j\leq k_n}$ be a sequence of points in $Y$. 
Let $(f_n:Y\to X)_{n\in\N}$ be a sequence of local isometries with the same constant $\varepsilon_0>0$ {associated with \eqref{locisodef}} and suppose that
\begin{eqnarray}\label{defofrectan}
    \Big({\bf x}_{n,j}+\prod_{i=1}^dB_{i,n}^{s_{i,n}}\Big)_{1\le j\le k_n, n\in\N}
\end{eqnarray}
 is a  local ubiquity system with respect to   $(f_n)_{n\in\N}$, where $s_{i,n}$ is as in \eqref{defofsinudv}. 
Then we have
\[\limsup_{n\to\infty}\bigcup_{1\le j\le k_n}f_n\Big({\bf x}_{n,j}+\prod_{i=1}^dB_{i,n}\Big) \in \mathcal{G}^{s({\bf u},{\bf v})}(X).\]
\end{thm}

\begin{rem}
{In the preceding results, we invoke the condition \eqref{condition:measure} on the sequence of rectangles as in \eqref{defofrectan}.  This condition is, however, {somewhat restrictive}. {As established in Theorem \ref{MTP2}, it can be replaced by a weaker alternative with additional constraints imposed on the sequence $(f_n)_{n\in\N}$.}}

\end{rem}

\subsection{Mass transference principle under full measure}
 
 Give any constant $c>0$  and  rectangle  $R\subset Y$ (or $R\subset X$) with the form 
\[\prod_{i=1}^dB({\bf x}_i,r_i),\] 
where $B({\bf x}_i,r_i)\subset Y_i$ (or $B({\bf x}_i,r_i)\subset X$) is a ball, we denote
\[cR:= \prod_{i=1}^dB({\bf x}_i,cr_i).\] 
Moreover, for a map $f:Y\to X$, we let
\[
cf(R):=f\left(\prod_{i=1}^dB({\bf x}_{i} , cr_i)\right).
\]


Let $(f_n:Y\to X)_{n\in\N}$ be a sequence of local isometries with the same constant $\varepsilon_0>0$ associated with \eqref{locisodef} and let $\{r_n\}_{n\in\N}$ be as in \eqref{conofbandv}.   We say that $(f_n)_{n\in\N}$ satisfies the condition ({\bf H}) if the following statement holds: \textit{there exists a sequence of real numbers $(c_n)_{n\in\N}\subset[1,+\infty)$ such that
 $$\lim_{n\to\infty}\frac{\log c_n}{-\log r_n}=0,$$
and there exists a local isometry $f$ with the constant $\varepsilon_0>0$ associated with \eqref{locisodef} such that  for  
any $n\in\N$ and any rectangle $R=\prod_{i=1}^d B({\bf x}_i,r_i)\subset Y$ with ${\rm diam}(R)<\varepsilon_0$ , there are some vectors $({\bf y}_{R,i,n})_{n\in\N}$ in $Y_i$ for any $1\leq i\leq d$ that satisfy
\[f\Big(\prod_{i=1}^d B({\bf y}_{R,i,n},c_n^{-1}r_i)\Big)\subset f_n(R)\subset f\Big(\prod_{i=1}^d B({\bf y}_{R,i,n},c_nr_i)\Big).\]}

Let $(R_n)_{n\in\N}$ be a sequence of subsets in $Y$. We say  that  $(R_n)_{n\in\N}$ satisfies \textit{full measure condition} with respect to $(f_n)_{n\in\N}$ if 
\[\mu\left(\limsup_{n\to\infty}f_n(R_n)\right)=\mu(X).\]


\begin{thm}\label{MTP2}
Let {$(B_{i,n})_{1\leq i\leq d,n\in\N}$}, $({\bf v}(n))_{n\in\N}$ and ${\bf v}$  {satisfy \eqref{conofbandv}}. Let $(f_n:Y\to X)_{n\in\N}$ be a sequence of local isometries with the same constant $\varepsilon_0>0$ associated with \eqref{locisodef} that satisfy the condition {\bf (H)}. 
Let ${\bf u}(n)=(u_{1,n},\dots,u_{d,n})\in (\R^+)^d$ and ${\bf u}=(u_1,\dots,u_d)\in (\R^+)^d$ be vectors that satisfy \eqref{conditiononu}. Assume that 
$\big(\prod_{i=1}^dB_{i,n}^{s_{i,n}}\big)_{n\in\N}$ satisfies full measure condition  with respect to $(f_n)_{n\in\N}$, where $s_{i,n}$ is defined as in \eqref{defofsinudv} for any $i=1,...,d$ and $n\in\N$.
Then 
\[\limsup_{n\to\infty}f_n\Big(\prod_{i=1}^dB_{i,n}\Big)\in\mathcal{G}^{s({\bf u},{\bf v})}(X).\]

\end{thm}

The {Mass Transference Principle} is often used in studying limsup sets, such as the shrinking target problem, and it gives the Hausdorff dimension formula in some cases (see \cite{Lietal,LLVWZ,WW,WWX}). 
However, {as shown in Theorem \ref{counterexample} and Example \ref{exam1},} the lower bound given by the {Mass Transference Principle} is not always the proper dimension formula.

\section{Preliminaries}\label{preliminaries}
 
Recall that for any $i = 1, . . . , d$, $X_i$ is a subset of $\R^{p_i}$  and   $(X_i,\rho_i)$ is a compact metric space equipped with a $\delta_i$-Ahlfors regular measure $\mu_i$. 
Denote $X=\prod_{i=1}^dX_i$ and $\mu=\prod_{i=1}^d\mu_i$.  The distance between  ${\bf x}=({\bf x}_1,\dots,{\bf x}_d)$ and ${\bf y}=({\bf y}_1,\dots,{\bf y}_d)\in X$ is  $$\rho({\bf x},{\bf y}):=\left(\sum_{i=1}^d\rho_i({\bf x}_i,{\bf y}_i)^2\right)^{\frac{1}{2}}.$$
 Then $(X,\mu)$ is an Ahlfors $\delta_0$-regular metric space,  where $\delta_0:=\delta_1+\cdots+\delta_d$.

 \subsection{Dyadic decomposition} {  Let $(X,\mu)$ be as defined above.  By \cite{krs}, for any parameter $0<b<\frac{1}{3}$,  there exists a dyadic decomposition of 
$X$ consisting of Borel sets 
$$\{Q_{n,i}\colon n\in\mathbb{N}, i\in\mathbb{N}_n\}\quad{\rm with}\quad  \mathbb{N}_n\subset \mathbb{N},$$ }
satisfying the following properties:
	\begin{enumerate}
		\item $X=\bigcup_{i\in\mathbb{N}_n}Q_{n,i}$ for every $n\in\mathbb{N}$.\medskip
		\item $Q_{n,i}\cap Q_{m,j}=\varnothing $ or $Q_{n,i}\subset Q_{m,j}$, where $n,m\in \mathbb{N}$, $n\ge m$, 
		$i\in \mathbb{N}_n $ and $j\in \mathbb{N}_m$.\medskip
		\item For every $n\in \mathbb{N}$ and $i\in\mathbb{N}_n$, there exists a point ${\bf x}_{n,i}\in X$ such that 
		\begin{equation*}\label{cube}
			B({\bf x}_{n,i},c_1b^n)\subset Q_{n,i}\subset \overline{B}({\bf x}_{n,i},c_1'b^n),
		\end{equation*}
		where $c_1=\frac{1}{2}-\frac{b}{1-b}$, $c'_1=\frac{1}{1-b}$ { and $\overline{B}$ is the closure of $B$}.\medskip
		\item There exists a point ${\bf x}_0\in X$ so that for every $n\in\mathbb{N}$, there is an index $i\in\mathbb{N}_n$ with 
		$B({\bf x}_0,c_1b^n)\subset Q_{n,i}$.\medskip
		\item $\{{\bf x}_{n,i}\colon i\in \mathbb{N}_n\}\subset \{{\bf x}_{n+1,i}\colon i\in \mathbb{N}_{n+1}\} $ for all $n\in\mathbb{N}$.
	\end{enumerate}
The family $\{Q_{n,i} \colon n \ge0, i \in \mathbb{N}_n\}$ satisfying the above properties is called ``generalized dyadic cubes".  For convenience, we write $\mathcal{D}_0=\{X\}$ and $\mathcal{D}_n=\{Q_{n,i}\colon i\in \mathbb{N}_n\}$ for $n\ge1$.

\subsection{Sets with large intersection }\label{lippp}
Define the outer net measure associated with $s$ and $\mu$ as follows:
\[\forall \ E\subset X,\quad \mathcal{M}_{\infty}^s(E):=\inf\left\{\sum_{k}\mu( D_k)^{s/\delta_0}\colon E\subset \bigcup_{k\ge1} D_k, D_k\in \bigcup_{n\ge0}\mathcal{D}_n\right\}.\]
The $s$-dimensional class of sets with large intersection, denoted by $\mathcal{G}^s( X)$,  is defined as
\begin{multline*}
\mathcal{G}^s( X)=\mathcal{G}_{\mu}^s( X):=\Big\{E\subset  X\colon E{\rm ~is ~a ~}G_{\delta}~{\rm set~with~}\mathcal{M}_{\infty}^t(E\cap  D) = \mathcal{M}_{\infty}^t( D) \\
 {\rm~for~ all~}D\in \bigcup_{n\ge0}\mathcal{D}_n{\rm ~and ~for~all~}t < s \Big\}.
 \end{multline*}
{By definition, $\dimh E\geq s$ for any $E\in\mathcal{G}^s(X)$. The term ``large intersection'' comes from the following property: \textit{if $\{E_n\}_{n\in\N}\subseteq\mathcal{G}^s(X)$, then their intersection $\bigcap_{n=1}^{\infty}E_n$ also belongs to $\mathcal{G}^s(X)$.}} 
 
 \begin{rem}
 Note that the classes $\mathcal{G}^s( X)$ do not depend on the choice of the constant $b$, even if it affects the construction of the generalized dyadic cubes $\{\mathcal{D}_n\}$, see \cite[Theorem 1.2 and Proposition 2.2]{NS}. Moreover, let $\mu'$ be another Ahlfors $\delta_0$-regular measure on $X$, then $\mathcal{G}_{\mu}^s( X)=\mathcal{G}_{\mu'}^s( X)$, see \cite[Theorem 1.4 ]{NS}.
\end{rem}

{The following lemma is a slight variation of \cite[Lemma 3.2]{HPWZ}, which is crucial for estimating the lower bound of the Hausdorff dimension of various sets in this paper.
In fact, it provides a {relatively effective} method to determine whether a limsup-set is in the class $\mathcal{G}^s( X)$,  {without} the need for estimating the outer net measure of {intersection with any generalized dyadic cube as in definition}. }

\begin{lem} \label{lemmaforlb}
Let $(X,\rho)$ be { a} compact {metric space} with an Ahlfors $\delta_0$-regular $\mu$. For all $n\in\N$, let $E_n$ be open sets in $X$ and $\mu_n$ be  probability measures with  $\mu_n(X\setminus E_n) =  0$.
  Suppose that there are constants $C\ge1$ and $s\ge 0$ such that
  \begin{equation}\label{lip}
    C^{-1} \le \liminf_{n\to\infty} \frac{\mu_n(B)}{\mu(B)}\le
    \limsup_{n\to\infty} \frac{\mu_n(B)}{\mu(B)}  \le C
  \end{equation}
  for any ball $B\subset X$, and 
  \begin{equation}\label{condition2}
  \mu_n(B) \le Cr^s
  \end{equation}
   for any $n\in\N$ and any  ball $B$ with radius $r$.  Then  $\limsup_{n\to\infty}  E_n\in\mathcal{G}^s(X).$ { In particular, we have $\dimh\limsup_{n\to\infty}E_n\geq s$.}
\end{lem}
\begin{proof}
Let $t<s$ and let $B\subset X$ be a ball with radius $r>0$. If $\mu_n(B)\leq C r^s$, then  given any ${\bf x} \in X$ we have
\begin{equation*}
\begin{split}
 \int \left(\rho({\bf x},{\bf y})\right)^{-t} \mathrm{d} \mu_n({\bf y})
    & \le 1+\int_1^{\left(\rho({\bf x},{\bf y})\right)^{-t}} \mathrm{d} u\mathrm{d} \mu_n({\bf y})
     = 1+\int_1^{\infty}\mu_n\bigl(B({\bf x}, u^{-1/t})\bigr) \mathrm{d}
      u \\
      &\le 1+
      2C\int_1^{\infty}u^{-s/t}  \mathrm{d} u
     = 1+2C\frac{t}{s-t}.
\end{split}
\end{equation*}
It derives that 
\[\int\int \left(\rho({\bf x},{\bf y})\right)^{-t} \mathrm{d} \mu_n({\bf y})\mathrm{d} \mu_n({\bf x})<1+2C\frac{t}{s-t}<+\infty.\]
Combing this with \cite[Lemma 3.3]{HLY}, we have  $\limsup\limits_{n\to\infty}
  E_n\in\mathcal{G}^t(X)$ for every $t<s$. By definition, we have  $\limsup\limits_{n\to\infty}
  E_n\in\mathcal{G}^s(X).$
\end{proof}

\subsection{Lattice}\label{lattice}

In this section, we briefly introduce the concept of a lattice in $\mathbb{R}^d$ and state some properties of lattices. One can refer to \cite{latt,lnogn} for more details.  {To begin with, we present some notation that will be used throughout. 


Write $f\lesssim g$ if there exists a constant $0<c< \infty$ (called the implied constant)  such that $f\le cg$, and $f \asymp g$ if both $f\lesssim g$ and $g\lesssim f$ hold. Since the dimension $d$ of the space $[0,1)^d$ 
will not be important for our purposes, we omit its dependence in these notations. For other parameters affecting the implied constants, we will explicitly indicate them.}


\begin{definition}[Lattices]\label{lattice} Let $\alpha_1,\dots , \alpha_d$ be $d$ linearly independent points in $\R^d$. The collection of points $\Lambda:=\{u_1 \alpha_1 + \cdots  + u_d \alpha_d \colon u_i \in\mathbb{Z}\}$ is called a lattice in $\R^d$ and  $\alpha_1,\dots, \alpha_n $
is called a basis of the lattice.
\end{definition}


{The following criterion for a set to be a lattice  follows directly from basic linear algebra and the proof is omitted.}

\begin{proposition}
 A set $\Lambda \subset \R^d$ is a lattice in $\R^d$ if and only if there exists a $d\times d$ nonsingular matrix $A$  
such that $\Lambda = A\mathbb{Z}^d$.
 \end{proposition}
 

Let $\Lambda$ be a lattice  and let $\alpha_1,\dots , \alpha_d$  be a basis of $\Lambda$. A fundamental domain for $\Lambda$ { associated with this basis} is given by the parallelepiped 
$$\{x_1 \alpha_1 + \cdots  + x_d \alpha_d \colon 0\le x_i < 1\}$$
whose  Lebesgue measure $|\det (\alpha_1,\dots, \alpha_d)|$  is independent of the choice of the basis; see \cite[Proposition 3.1.6]{lnogn}. This measure is called the covolume of $\Lambda$ denoted by ${\rm covol}(\Lambda)$.

{
\begin{definition}[Successive minima]\label{minina} Let $\Lambda \subset \R^d$ be a lattice. We define the successive minima of $\Lambda$ by
\[m_k(\Lambda) = \min\{r : \Lambda~ {\rm contains}~ k ~{\rm linearly~ independent ~vectors~ of~ Euclidean~norm~} \le r\}\]
for any $k = 1,\dots, d.$ 
\end{definition}}
Throughout, let $\vert \cdot \vert$ denote the Euclidean norm { on $\R^d$}. 
\begin{thm}[{\cite[Theorem~1.15]{latt}}]\label{thmsm}
 Let $\Lambda \subset \R^d$ be a lattice and $m_k(\Lambda)$, $k = 1,\dots, d$ be the successive minima of $\Lambda$. Then
\[m_1(\Lambda)\cdots m_d(\Lambda) \asymp \rm{covol}(\Lambda),\]
\end{thm}

\begin{cor}[Basis of a lattice{\cite[Corollary 1.16]{latt}}]\label{corsm}
 Let $\Lambda \subset \R^d$ be a lattice. Then there
is a basis ${\bf v}_1,\dots , {\bf v}_d \in  \Lambda $ of $\Lambda $ such that
\begin{equation}\label{1}
\vert {\bf v}_1\vert = m_1(\Lambda), \vert {\bf v}_2\vert \asymp m_2(\Lambda), \dots , \vert {\bf v}_d\vert \asymp m_d(\Lambda).
\end{equation}
Moreover, the projection $\pi_k({\bf v}_k)$ of ${\bf v}_k$ onto the orthogonal complement of
\[\R {\bf v}_1 + \cdots + \R {\bf v}_{k-1}\]
{satisfies} 
\begin{equation}\label{2}
\vert \pi_k({\bf v}_k)\vert \asymp m_k(\Lambda) \asymp \vert {\bf v}_k\vert
\end{equation}
for $k = 2, \dots , d$.
\end{cor}

\begin{rem}
Both Theorem \ref{thmsm} and Corollary \ref{corsm} can be found as Theorem 3.3.8 (Minkowski’s Second Theorem) in \cite{EWp}. And the implied constants in these statements depend only on the dimension $d$.

\end{rem}

\begin{rem} {Let $\mathcal{A}=(A_n)_{n\in\N}$ be a sequence of matrices in $GL_d(\mathbb{R})$ with $\Gamma(\mathcal{A})\subset (\R^+)^d$. 
{Then,} for { any} $n\ge1 $, {it follows directly from definition that} $\Lambda_n:=A_n^{-1}\mathbb{Z}^d$ is a lattice. {Moreover}, for any sufficiently large $n\in\N$, the successive minima of  $A_n^{-1}\mathbb{Z}^d$ satisfies 
$$0<m_1(\Lambda_n)\le m_2(\Lambda_n)\le \cdots \le m_k(\Lambda_n)<1 .$$
}
\end{rem}


\section{The proofs of Theorem \ref{MTP1} and Theorem \ref{MTP2}}\label{proofof13}\label{proofmtp}
\subsection{Some important lemmata}
When it comes to proving the Theorem \ref{MTP1}, it will be more convenient to use {the following more general form of dimensional number than that given in Section \ref{MTP}.}
For  any $a>0$ and any vectors ${\bf u}=(u_1,...,u_d)\in(\R^+)^d$ and ${\bf v}=(v_1,...,v_d)\in(\R^+\cup\{+\infty\})^d$ that satisfy $u_i\leq v_i$ for any $1\leq i\leq d$, define  \[ \mathfrak{s}({\bf u},{\bf v},a):=\sum_{k\in\K_1(a)}\delta_k+\sum_{k\in\K_2(a)}\delta_k\left(1-\frac{v_{k}-u_{k}}{a}\right)+\sum_{k\in\K_3(a)}\frac{\delta_ku_{k}}{a},\]
where
\[\K_1(a):=\{1\le k\le d\colon u_{k}>a\}, \quad \K_2(a):=\{1\le k\le d\colon v_{k}\le a\},\]
and
\[\K_3(a)=\{1,\dots,d\}\setminus (\K_1(a)\cup \K_2(a)).\]
{We note that the above $\mathfrak{s}({\bf u},{\bf v},a)$ is related to the previously defined value $s({\bf u},{\bf v},i)$ via the equality $s({\bf u},{\bf v}, i)=\mathfrak{s}({\bf u},{\bf v},v_i)$ for any $1\leq i\leq d$. More relationships between them can be found in the following lemma.}

\begin{lem}[Proposition 4.1 in \cite{LLVWZ}]\label{lemma1}
Suppose that ${\bf u},{\bf v}\in(\mathbb{R}^+)^d$ satisfy $u_{i}\le v_{i}$ for any $1\le i\le d$, then we have
\[\min_{a\in\mathcal{A}}\{ \mathfrak{s} ({\bf u},{\bf v},a)\}=\min_{1\le i\le d}\{s({\bf u},{\bf v},i)\},\]
where $\mathcal{A}:=\{u_i,v_i:1\leq i\leq d\}$.
\end{lem}
The next lemma  {states} that $s({\bf u},{\bf v})$ is {jointly} continuous  in ${\bf u}$ and ${\bf v}$.
\begin{lem}\label{lemma2}
Let $({\bf u}(n))_{n\in\N},~({\bf v}(n))_{n\in\N}$ be two sequences of vectors in $(\R^+)^d$ satisfying that  $u_{i,n}\le v_{i,n}$ for any $1\le i\le d$ and $n\in\N$, and suppose that they converge to ${\bf u}\in(\mathbb{R}^+)^d, ~{\bf v}\in (\mathbb{R}^+\cup \{+\infty\})^d$ respectively as $n\to\infty$. Then 
\[\lim_{n\to\infty}\min_{1\le i\le d}\{s({\bf u}(n),{\bf v}(n),i)\}=\min_{1\le i\le d}\{s({\bf u},{\bf v},i)\}.\]
\end{lem}
This lemma can be proved by using the same method as in \cite[Proposition 4.2]{LLVWZ} which implies that $s({\bf u},{\bf v})$ is continuous in ${\bf v}$. Here we include the proof for completeness. {Throughout, for any $n\in\N$, we denote $\mathcal{A}_n:=\{u_{k,n},v_{k,n}:1\leq k\leq d\}$, where ${\bf u}(n)=(u_{1,n},...,u_{d,n})$ and ${\bf v}(n)=(v_{1,n},...,v_{d,n})$ are vectors in $(\R^+)^d$ as stated in Lemma \ref{lemma2}.}
\begin{proof}[Proof of Lemma \ref{lemma2}] Observe that for any $n\ge1$ and $a\in \mathcal{A}_n$, redefining 
\[ \K_2(a):=\{1\le k\le d\colon v_{k,n}< a\}\]
makes the value of $ \mathfrak{s}({\bf u}(n),{\bf v}(n),a)$ invariant, and
the same holds also for $ \mathfrak{s}({\bf u},{\bf v},a) $.  We refer to \cite[Remark 4.1]{LLVWZ} for more details. We shall use the above new definitions {of $\mathfrak{s}({\bf u}(n),{\bf v}(n),a)$ and $\mathfrak{s}({\bf u},{\bf v},a)$} to prove this lemma. Note that
\begin{equation}\label{6142}
\min_{1\le i\le d}\{s({\bf u},{\bf v},i)\}=\min\Big\{\min_{i\in  \mathcal{L}(\bf v)}\{s({\bf u},{\bf v},i)\},\min_{i\in  \mathcal{L}_\infty(\bf v)}\{s({\bf u},{\bf v},i)\}\Big\}.
\end{equation}
{With this in mind, to prove the lemma, it suffices to show that
\begin{equation}\label{limfinite}
    \lim_{n\to\infty}s({\bf u}(n),{\bf v}(n),i)=s({\bf u},{\bf v},i), \quad\forall \ i\in\mathcal{L}({\bf v})
\end{equation}
and
\begin{equation}\label{liminfinte}
\lim_{n\to\infty}\min_{i\in \mathcal{L}_\infty(\bf v)}s({\bf u}(n),{\bf v}(n),i)=\min_{i\in \mathcal{L}_\infty(\bf v)}s({\bf u},{\bf v},i).
\end{equation}
We are going to prove \eqref{limfinite} and \eqref{liminfinte} separately in the following.}
\medskip

\noindent\textit{$\bullet$ Proving the equality \eqref{limfinite}.}
Fix any $i\in\mathcal{L}({\bf v})$, then by the definition, $v_i<\infty$. Recall that 
\[\K_1(v_i):=\{1\le k\le d\colon u_k>v_i\}, \quad \K_1(v_{i,n}):=\{1\le k\le d\colon u_{k,n}>v_{i,n}\},\]
and
\[\K_2(v_i):=\{ k\in\mathcal{L}({\bf v})\colon v_k<v_i\}, \quad \K_2(v_{i,n}):=\{1\le k\le d\colon v_{k,n}<v_{i,n}\}.\]
Since ${\bf v}(n)\to {\bf v}$ and ${\bf u}(n)\to {\bf u}$ as $n\to\infty$, it follows that for all sufficiently large $n\in\N$,
\[\K_1(v_i)\subset \K_1(v_{i,n}), \quad \K_2(v_i)\subset \K_2(v_{i,n}).\]
{For such $n\in\N$,} put
\[I_1=\K_1(v_{i,n})\setminus \K_1(v_i), \quad I_2=\K_2(v_{i,n})\setminus \K_2(v_i).\]
Then
\begin{equation}\label{expreofs}
\begin{split}
s({\bf u}(n),{\bf v}(n),i)&= \sum_{k\in\K_1(v_i)\cup\K_2(v_i)}\delta_k+\frac{1}{v_{i,n}}\Big(\sum_{k\in\K_2(v_i)}\delta_k(u_{k,n}-v_{k,n})+\sum_{k\in\K_3(v_i)}\delta_ku_{k,n}\Big)\\[4pt]
&\quad\qquad +\sum_{k\in I_1}\delta_k\big(1-\frac{u_{k,n}}{v_{i,n}}\big)+\sum_{k\in I_2}\delta_k\big(1-\frac{v_{k,n}}{v_{i,n}}\big).
\end{split}
\end{equation}
Denote 
\[T_n:=\sum_{k\in\K_1(v_i)\cup\K_2(v_i)}\delta_k+\frac{1}{v_{i,n}}\Big(\sum_{k\in\K_2(v_i)}\delta_k(u_{k,n}-v_{k,n})+\sum_{k\in\K_3(v_i)}\delta_ku_{k,n}\Big),\]
and notice that 
\begin{eqnarray}\label{limtnsuvi}
    \lim\limits_{n\to\infty}T_n=s({\bf u},{\bf v},i).
\end{eqnarray}
Now, for any $k\in I_1$, we have 
\[u_{k,n}>v_{i,n}, \quad u_k\le v_i,\]
which implies that 
\begin{equation}\label{kini1lim}
    \lim_{n\to\infty}\big(1-\frac{u_{k,n}}{v_{i,n}}\big)=0.
\end{equation}
If $k\in I_2$, then we have 
\[v_{k,n}<v_{i,n}, \quad v_k\ge v_i,\]
which yields 
\begin{equation}\label{kini2lim}
    \lim_{n\to\infty}\big(1-\frac{v_{k,n}}{v_{i,n}}\big)=0.
\end{equation}
Combining \eqref{expreofs}-\eqref{kini2lim}, we obtain the desired equality \eqref{limfinite} for any fixed $i\in\mathcal{L}({\bf v})$. 
\medskip

\noindent\textit{$\bullet$ Proving the equality \eqref{liminfinte}.}
For $i\in \mathcal{L}_\infty(\bf v)$, we have $\lim\limits_{n\to\infty}v_{i,n}=v_i=+\infty$. It follows that  for any sufficiently large $n\in\N$, 
\[\K_1(v_i)=\K_1(v_{i,n})=\emptyset,\quad\K_2(v_i)=\mathcal{L}({\bf v})\quad\text{and}\quad \K_2(v_{i,n})\supset \mathcal{L}({\bf v}).\]
Then, {for such $n\in\N$,} 
\begin{equation}\label{suv}
\begin{split}
s({\bf u}(n),{\bf v}(n),i)&=\sum_{k\in\K_2(v_{i,n})}\delta_k\left(1-\frac{v_{k,n}-u_{k,n}}{v_{i,n}}\right)+\sum_{k\in\K_3(v_{i,n})}\frac{\delta_ku_{k,n}}{v_{i,n}}\\
&=\sum_{k\in \mathcal{L}({\bf v})}\delta_k\left(1-\frac{v_{k,n}-u_{k,n}}{v_{i,n}}\right)+\sum_{k\in \K_2(v_{i,n})\setminus\mathcal{L}({\bf v})}\delta_k\left(1-\frac{v_{k,n}-u_{k,n}}{v_{i,n}}\right)\\
&\quad +\sum_{k\in\K_3(v_{i,n})}\frac{\delta_ku_{k,n}}{v_{i,n}}.
\end{split}
\end{equation}
Notice that  in \eqref{suv}, {since $v_{i,n}\to\infty$ as $n\to\infty$},
\[\lim_{n\to\infty}\frac{1}{v_{i,n}}\Big(\sum_{k\in\K_3(v_{i,n})}\delta_ku_{k,n}\Big)=0\quad\text{and}\quad\lim_{n\to\infty}\frac{1}{v_{i,n}}\Big(\sum_{k\in \mathcal{L}({\bf v})}\delta_k\frac{v_{k,n}-u_{k,n}}{v_{i,n}}\Big)=0.\]
{With this in mind, to prove the desired equality \eqref{liminfinte}, we are required  to analyze the asymptotic behavior of the term
\[
\sum_{k\in \K_2(v_{i,n})\setminus\mathcal{L}({\bf v})}\delta_k\left(1-\frac{v_{k,n}-u_{k,n}}{v_{i,n}}\right).
\]}
Observe that if $k\in \K_2(v_{i,n})\setminus\mathcal{L}({\bf v})$, then $v_{k,n}<v_{i,n}$. This implies that 
\[\sum_{k\in \K_2(v_{i,n})\setminus\mathcal{L}({\bf v})}\delta_k\left(1-\frac{v_{k,n}-u_{k,n}}{v_{i,n}}\right)>0 \qquad\forall\ n\in\N.\]
Therefore, for any $\varepsilon>0$, we have
\begin{equation*}
s({\bf u}(n),{\bf v}(n),i)\ge \sum_{k\in \mathcal{L}({\bf v})}\delta_k-\varepsilon
\end{equation*}
for any sufficiently large $n\in\N$ and all $i\in  \mathcal{L}_\infty(\bf v)$. This establishes  
\begin{equation}\label{lowerofsuv}
\liminf_{n\to\infty}\min_{i\in  \mathcal{L}_\infty({\bf v})}s({\bf u}(n),{\bf v}(n),i)\ge  \sum_{k\in \mathcal{L}({\bf v})}\delta_k.
\end{equation}
On the other hand, for any $n\ge1$, {we choose $i_n\in\mathcal{L}_\infty({\bf v})$ such that 
\[v_{i_n,n}=\min_{i\in \mathcal{L}_\infty({\bf v})}v_{i,n}.\]}Then it is easily verified that $\lim\limits_{n\to\infty}v_{i_n,n}=+\infty$. Moreover, for sufficiently large $n$, it follows that 
\[\K_1(v_{i_n,n})=\emptyset \quad{\rm and}\quad \K_2(v_{i_n,n})=\mathcal{L}({\bf v}).\]
Hence we get
\begin{equation*}\label{614}
\begin{split}
s({\bf u}(n),{\bf v}(n),i_n)&= \sum_{k\in \mathcal{L}({\bf v})}\delta_k+\sum_{k\in \mathcal{L}({\bf v})}\delta_k\frac{u_{k,n}-v_{k,n}}{v_{i_n,n}}+\sum_{k\in\K_3(v_{i_n,n})}\frac{\delta_ku_{k,n}}{v_{i_n,n}}.
\end{split}
\end{equation*}
Note that 
\[\lim_{n\to\infty}\frac{u_{k,n}}{v_{i,n}}=0,\qquad\forall\ 1\leq k\leq d,\]
and
\[\lim_{n\to\infty}\frac{u_{k,n}-v_{k,n}}{v_{i_n,n}}=0,\qquad\forall \ k\in \mathcal{L}({\bf v})\]
Therefore,
\begin{equation}\label{6141}
\limsup_{n\to\infty}\min_{i\in  \mathcal{L}_\infty({\bf v})}s({\bf u}(n),{\bf v}(n),i)\le \lim_{n\to\infty}s({\bf u}(n),{\bf v}(n),i_n)=  \sum_{k\in \mathcal{L}({\bf v})}\delta_k.
\end{equation}
By combining \eqref{lowerofsuv} and \eqref{6141}, we obtain the desired equality \eqref{liminfinte}.
\end{proof}

{The forthcoming auxiliary lemma is helpful to the proofs of Theorem \ref{MTP1} and Theorem \ref{MTP2}.}

\begin{lem}\label{elements in an}
For any $n\ge1$, arrange the elements in $\mathcal{A}_n$ in ascending order denoted by
\[
t_1<t_2<\cdots< t_{\#\mathcal{A}_n}.
\]
Then, for any $1\le k <\#\mathcal{A}_n$, we have
\begin{equation*}
\mathfrak{s}({\bf u}(n),{\bf v}(n),t_{k+1})\ =\  
\sum_{i\in \K_1(t_k)}\delta_i+\sum_{i\in  \K_2(t_k)}\delta_i\left(1-\frac{v_{i,n}-u_{i,n}}{t_{k+1}}\right)+\sum_{i\in \K_3(t_k)}\frac{\delta_iu_{i,n}}{t_{k+1}}.
\end{equation*}
\end{lem}

\begin{proof}
Since $t_k<t_{k+1}$  are distinct consecutive elements in $\mathcal{A}_n$, we get
\[
\K_1(t_k)=\{1\le i\le d\colon u_{i,n}\ge t_{k+1}\},
\]
and 
\[
\K_2(t_k)=\{1\le i\le d\colon v_{i,n}<t_{k+1}\}.
\]
It follows that 
\[\K_1(t_k)=\K_1(t_{k+1})\cup \{1\le i\le d\colon u_{i,n}=t_{k+1}\},\medskip\]
\[\K_2(t_k)=\K_2(t_{k+1})\setminus \{1\le i\le d\colon v_{i,n}=t_{k+1}\},\]
and
\[\K_3(t_k)=\K_3(t_{k+1})\cup \{1\le i\le d\colon v_{i,n}=t_{k+1}\}\setminus \{1\le i\le d\colon u_{i,n}=t_{k+1}\}.\]
We write 
\[s_1:= \sum_{i\in \K_1(t_k)}\delta_i+\sum_{i\in  \K_2(t_k)}\delta_i\left(1-\frac{v_{i,n}-u_{i,n}}{t_{k+1}}\right)+\sum_{i\in \K_3(t_k)}\frac{\delta_iu_{i,n}}{t_{k+1}}.\]
Then 
\begin{equation*}
\begin{split}
s_1&= \sum_{i\in \K_1(t_{k+1})}\delta_i +\sum_{\{i\colon u_{i,n}=t_{k+1}\}}\delta_i +\sum_{i\in  \K_2(t_{k+1})}\delta_i\left(1-\frac{v_{i,n}-u_{i,n}}{t_{k+1}}\right)\\
&\quad \qquad\qquad-\sum_{\{i\colon v_{i,n}=t_{k+1}\}}\delta_i\left(1-\frac{v_{i,n}-u_{i,n}}{t_{k+1}}\right)+\sum_{i\in \K_3(t_{k+1})}\frac{\delta_iu_{i,n}}{t_{k+1}}\\
&\quad\qquad\qquad +\sum_{\{i\colon v_{i,n}=t_{k+1}\}}\frac{\delta_iu_{i,n}}{t_{k+1}}-\sum_{\{i\colon u_{i,n}=t_{k+1}\}}\frac{\delta_iu_{i,n}}{t_{k+1}}\\
&=\mathfrak{s}({\bf u}(n),{\bf v}(n),t_{k+1})-\sum_{\{i\colon v_{i,n}=t_{k+1}\}}\delta_i\left(1-\frac{v_{i,n}-u_{i,n}}{t_{k+1}}-\frac{u_{i,n}}{t_{k+1}}\right)\\
&\quad \qquad\qquad+\sum_{\{i\colon u_{i,n}=t_{k+1}\}}\delta_i\left(1-\frac{u_{i,n}}{t_{k+1}}\right)\\
&=\mathfrak{s}({\bf u}(n),{\bf v}(n),t_{k+1}).
\end{split}
\end{equation*}
 The proof is complete.
\end{proof}

The following covering lemma  for rectangles is a slight modification of \cite[Lemma 5.1]{LLVWZ}. 
\begin{lem}[Covering Lemma]\label{coverlemma}
 Let $f:Y\to X$ be a local isometry with the constant $\varepsilon_0>0$ associated with \eqref{locisodef}, ${\bf a} = (a_1, \dots, a_d )$ be a point in $(\R^+)^d$ and $\cY$  be a collection of points in $Y=\prod_{i=1}^d Y_i$. For any ${\bf y}\in\cY$, let $r_{\bf y}$ be a positive real number in $(0,1)$ that satisfies
\begin{equation}\label{condofrn2ai}
    \Big(\sum_{i=1}^d r_{\bf y}^{2a_i}\Big)^{\frac{1}{2}}<\frac{\varepsilon_0}{4},
\end{equation}
where $\varepsilon_0>0$ is  the index of the local isometry $f$ associated with \eqref{locisodef}. 
 Let $\mathcal{E}$ be a collection of sets in $X$ consisting of
$$f\Big(\prod_{i=1}^dB({\bf y} , r_{\bf y}^{a_i})\Big) \qquad\forall   \ {\bf y} \in\cY.$$
 Then there is a finite or countable sub-collection $\mathcal{E}' \subset \mathcal{E}$ such that the sets in $\mathcal{E}' $ are disjoint and 
\begin{equation}\label{resof5rcov}
    \bigcup_{R\in \mathcal{E}}R\subset \bigcup_{R\in \mathcal{E}'}5^{\tilde{c}}R,\qquad\text{where}\ \  \tilde{c}:=\frac{\max_i a_i}{\min_i a_i}.
\end{equation}
\end{lem}

\begin{proof}
    Denote $a_{\max}:=\max_{1\leq i\leq d}a_i$ and $a_{\min}:=\min_{1\leq i\leq d}a_i$. Let $\rho_{Y}'$ be a different metric than $\rho_{Y}$ on $Y$ as
    \begin{eqnarray*}
        \rho_{Y}'({\bf x},{\bf y}):=\max\left\{\rho_{Y_i}({\bf x}_i,{\bf y}_i)^{\frac{a_{\min}}{a_{i}}}:1\leq i\leq d\right\}\qquad\forall\ {\bf x}, {\bf y}\in Y.
    \end{eqnarray*}
    Then, by a direct calculation, we obtain that
     \begin{eqnarray}\label{prodbrhoy}
       \prod_{i=1}^d B({\bf y},r^{a_i})=B_{\rho_Y'}({\bf y},r^{a_{\min}})\qquad\forall\ {\bf y}\in \cY,\ \ \forall \ r>0.
    \end{eqnarray}
    
     Let $\rho_{X}'$ be the metric on $X$ induced by $\rho_{Y}'$ and the local isometry $f:Y\to X$ via the definition
    \begin{eqnarray*}
        \rho_X'({\bf x},{\bf y}):=\dH(f^{-1}({\bf x}),f^{-1}({\bf y}))\qquad\forall\ {\bf x},{\bf y}\in X,
    \end{eqnarray*}
    where $\dH$ is the Hausdorff distance associated with $\rho_Y'$. Specifically, for any subsets $E,F\subseteq Y$，
    \begin{eqnarray*}
        \dH(E,F):=\sup\big\{\inf\{\rho_Y'({\bf x},{\bf z}_1):{\bf z}_1\in F\},\,\inf\{\rho_Y'({\bf z}_2,{\bf {y}}):{\bf z}_2\in E\}:{\bf x}\in E,\,{\bf y}\in F\big\}.
    \end{eqnarray*}
     Observe that
    \begin{itemize}
    \item for any ${\bf y}\in Y$ and $r>0$, by definition, we have
\begin{equation}\label{ballrectangle}
         B_{\rho_X'}(f({\bf y}),r)\subseteq f\left(B_{\rho_Y'}({\bf y},r)\right).\medskip
\end{equation}
    
        \item for any ${\bf y}\in\cY$, since $r_{\bf y}$ satisfies \eqref{condofrn2ai}, it follows from \eqref{prodbrhoy},  \eqref{ballrectangle} and the locally isometric property of $f:Y\to X$ that
        \begin{equation*}
            f\Big(\prod_{i=1}^dB({\bf y}_i,r_{\bf y}^{a_i})\Big)=B_{\rho_X'}(f({\bf y}),r_{\bf y}^{a_{\min}}).\medskip
        \end{equation*}

    \end{itemize}
    Now, on applying $5r$-covering lemma \cite[Theorem 1.2]{Hein01} to the collection $$\left\{B_{\rho_X'}(f({\bf y}),r_{\bf y}^{a_{\min}})=f\Big(\prod_{i=1}^dB({\bf y},r_{\bf y}^{a_i})\Big)\right\}_{{\bf y}\in\cY},$$ there exists a finite or countable subset $\cY'\subseteq \cY$ such that
    \begin{eqnarray*}
        f\Big(\prod_{i=1}^dB({\bf y},r_{\bf y}^{a_i})\Big)\cap f\Big(\prod_{i=1}^dB({\bf y}',r_{\bf y'}^{a_i})\Big)=\emptyset \qquad\forall\ {\bf y}\neq{\bf y}'\in\cY'
    \end{eqnarray*}
    and
    \begin{eqnarray*}
        \bigcup_{{\bf y}\in\cY}f\Big(\prod_{i=1}^dB({\bf y},r_{\bf y}^{a_i})\Big)&=& \bigcup_{{\bf y}\in\cY}B_{\rho_X'}(f({\bf y}),r_{\bf y}^{a_{\min}})\\
        &\subseteq&\bigcup_{{\bf y}\in\cY'}B_{\rho_X'}(f({\bf y}),5r_{\bf y}^{a_{\min}})\\
        &\subseteq&\bigcup_{{\bf y}\in\cY'}f\big(B_{\rho_Y'}({\bf y},5r_{\bf y}^{a_{\min}})\big)\qquad(\text{by \eqref{ballrectangle}})\\
    &=&\bigcup_{{\bf y}\in\cY'}f\Big(\prod_{i=1}^d B\big({\bf y}_i,5^{\frac{a_{\max}}{a_{\min}}}r^{a_i}\big)\Big).
    \end{eqnarray*}
    The proof is thereby complete by letting $$\mathcal{E}'=\left\{f\Big(\prod_{i=1}^dB({\bf y},r_{\bf y}^{a_i})\Big) \right\}_{{\bf y}\in\cY'}.$$
\end{proof}


\subsection{Proof of Theorem \ref{MTP1} }

 We prove the theorem by  establishing the following more general statement than that of Theorem \ref{MTP1} with weaker conditions imposed on sequences $({\bf u}(n))_{n\in\N}$ and $({\bf v}(n))_{n\in\N}$.

Given $1\leq i\leq d$,  let $(B_{i,n})_{n\in\N}$ be a sequence of balls in $Y_i$ and suppose that there exist sequences of positive real numbers $(r_n)_{n\in\N}$ and $(v_{i,n})_{n\in\N}$ such that
\[r(B_{i,n})=r_n^{v_{i,n}}.\]
For any $n\in\N$, let  ${\bf v}(n)=(v_{1,n},\dots,v_{d,n})\in (\R^+)^d$ and ${\bf u}(n)=(u_{1,n},\dots,u_{d,n})\in (\R^+)^d$ satisfy 
$$u_{i,n}\leq v_{i,n}\quad\forall \ i=1,2,...,d\qquad\&\qquad{0<\inf_{n\in\N}\min_{1\leq i\leq d}u_{i,n}\leq\sup_{n\in\N}\max_{1\leq i\leq d}u_{i,n}<+\infty}$$
and denote
\[s_{i,n}:=\frac{u_{i,n}\delta_i}{v_{i,n}}.\]

\begin{thm}\label{thm5}


Let $(B_{i,n})_{1\leq i\leq d,n\in\N},({\bf v}(n))_{n\in\N},({\bf u}(n))_{n\in\N}$  be as given above. Let $(f_n:Y\to X)_{n\in\N}$ be a sequence of local isometries with the same constant $\varepsilon_0>0$ {associated with \eqref{locisodef}}. Let $(k_n)_{n\in\N} $ be a non-decreasing sequence of positive integers and $\{{\bf x}_{n,j}\}_{1\leq j\leq k_n,n\in\N}\subset Y$.
Moreover, assume that $$\left\{{\bf x}_{n,j}+\prod_{i=1}^dB_{i,n}^{s_{i,n}}\right\}_{1\le j\le k_n, n\in\N}$$
is a  local ubiquity system with respect to  local isometries $(f_n)_{n\in\N}$. Then
\[\limsup_{n\to\infty}\bigcup_{1\le j\le k_n}f_n\left({\bf x}_{n,j}+\prod_{i=1}^dB_{i,n}\right) \in \mathcal{G}^{t}(X),\]
where 
\begin{eqnarray}\label{defoftuv}
   t=t(({\bf u})_{n\in\N},({\bf v})_{n\in\N}):=\liminf_{n\to\infty}\min_{1\le i\le d}\{s({\bf u}(n),{\bf v}(n),i)\}.
\end{eqnarray}
\end{thm}

We first use Theorem \ref{thm5} to prove Theorem \ref{MTP1}.

{\begin{proof}[Proof of Theorem \ref{MTP1} modulo Theorem \ref{thm5}]
Let $({\bf u}(n))_n$ and $({\bf v}(n))_n$ be as given in Theorem \ref{MTP1}, then
\[u_{i,n}\le v_{i,n},\quad\forall \  1\le i\le d,\ n\geq1\]
and
\[
0<\inf_{n\in\N}\min_{1\leq i\leq d}u_{i,n}\leq\sup_{n\in\N}\max_{1\leq i\leq d}u_{i,n}<+\infty.
\]
Since $u_{i,n}\to{ u_i}$ and $v_{i,n}\to{v_i }$ as $n\to\infty$ for any $1\leq i\leq d$,
by Lemma \ref{lemma2}, we have
\[\lim\limits_{n\to\infty}\min_{1\le i\le d}\{s({\bf u}(n),{\bf v}(n),i)\}=s({\bf u},{\bf v}).\]
On exploiting Theorem \ref{thm5}, we finish the proof of Theorem \ref{MTP1}.
\end{proof}
}

\begin{proof}[Proof of Theorem \ref{thm5}]
For any $n\in\N$ and $1\le j\le k_n$, put
\[\widetilde{R}_{n,j}=f_n\Big({\bf x}_{n,j}+\prod_{i=1}^dB_{i,n}^{s_{n,i}}\Big)\quad {\rm and} \quad R_{n,j}=f_n\Big({\bf x}_{n,j}+\prod_{i=1}^dB_{i,n}\Big).\]
{Since $r_n\to 0$ as $n\to\infty$, without loss of generality, assume that \eqref{condofrn2ai} holds with $r_{\bf y}$ replaced by $r_n$ and $a_i$ replaced by $u_{i,n}$ for any $n\in\N$ and $1\leq i\leq d$; that is to say that
\begin{eqnarray}\label{rangofrn}
    \Big(\sum_{i=1}^d r_n^{2u_{i,n}}\Big)^{1/2}<\frac{\varepsilon_0}{4}.
\end{eqnarray}}Then, given large $n\ge1$,  by Lemma \ref{coverlemma}, 
there is a finite sub-collection $\mathcal{S}_n\subset \{1,2,\dots, k_n\}$ such that 
\[\big\{\widetilde{R}_{n,j}\big\}_{j\in \mathcal{S}_n} \]
is a collection of disjoint rectangles and 
\[\bigcup_{j=1}^{k_n}\widetilde{R}_{n,j}\subset \bigcup_{j\in \mathcal{S}_n}5^{\hat{c}}\widetilde{R}_{n,j},\]
where {$\hat{c}=\frac{2\sup_{n\in\N}\max_{1\leq i\leq d}u_{i,n}}{\inf_{n\in\N}\min_{1\leq i\leq d}u_{i,n}}<+\infty$}.

By the local ubiquity condition, there exists a constant $c>0$ such that for any ball $B_0$ in $X$, we have
 \[\limsup_{n\to\infty}\mu\Big(\bigcup_{j=1 }^{k_n }\widetilde{R}_{n,j}\cap B_0\Big)>c\mu(B_0).\]
Taking $B_0= {X}$ gives that
\begin{equation}\label{measureofen}
\mu\Big(\bigcup_{j=1 }^{k_n }\widetilde{R}_{n,j}\Big)\asymp \mu( {X})
\end{equation}
for infinitely many $n$, denoted by $\mathcal{N}$.
Let
\[\widetilde{E}_n=\bigcup_{j\in\mathcal{S}_n}\widetilde{R}_{n,j},
\quad {\rm and}
\quad E_n=\bigcup_{j\in\mathcal{S}_n}R_{n,j},\]
then  $E_n\subset \widetilde{E}_n$ since $R_{n,j}\subset \widetilde{R}_{n,j}$ for any $1\leq j\leq k_n$. Observe that 
\[\limsup_{n\in\mathcal{N}\atop n\to \infty}E_n\subset \limsup_{n\to\infty}\bigcup_{j=1}^{k_n}{R}_{n,j}.\]
Therefore, it suffices to consider the large intersection property  of $ \limsup_{n\in\mathcal{N}, n\to \infty}E_n$.

Note that  isometries preserve the Hausdorff measure. Then, for a local isometry $f:Y\to X$ with the constant $\varepsilon_0>0$ associated with \eqref{locisodef} and any set $E$ with ${\rm diam} E<\varepsilon_0$,
\begin{eqnarray}\label{locisohas}
    \mathcal{H}^{s} (f(E))=\mathcal{H}^{s}(E),\qquad\forall \ s>0.
\end{eqnarray}
Since $\mu$ and $\nu$ are Ahlfors $\delta_0$-regular measures on $X$ and $Y$ respectively,  where $\delta_0=\sum_{i=1}^d\delta_i$, it is known that $\mu\asymp \mathcal{H}^{\delta_0}|_X$ and $\nu\asymp\mathcal{H}^{\delta_0}|_Y$; see e.g.  \cite[Section 8.7]{Hein01}. This together with \eqref{rangofrn} and \eqref{locisohas} gives that
\begin{eqnarray*}
    \mu(R_{n,j})\asymp \mathcal{H}^{\delta_0}(R_{n,j}) =\mathcal{H}^{\delta_0}\Big({\bf x}_{n,j}+\prod_{i=1}^d B_{i,n}\Big)\asymp \nu\Big({\bf x}_{n,j}+\prod_{i=1}^d B_{i,n}\Big)\asymp\prod_{i=1}^d r_n^{\delta_iv_{i,n}}
\end{eqnarray*}
for any $n\in\N$ and $1\leq j\leq k_n$.  Similarly, 
\[
\mu(\widetilde{R}_{n,j})\asymp\nu\Big({\bf x}_{n,j}+\prod_{i=1}^d B_{i,n}^{s_{i,n}}\Big)\asymp\prod_{i=1}^dr_n^{\delta_iu_{i,n}}.
\]
Combing this and the choice of $n\in\mathcal{N}$, we deduce that
\begin{equation}\label{meaofen}
    \mu(\widetilde{E}_n)\asymp 1\quad {\rm and}\quad\mu(E_n)\asymp \prod_{i=1}^dr_n^{\delta_i(v_{i,n}-u_{i,n})}.
\end{equation}

Define a probability measure supported on $E_n$ as
\[\mu_n:=\sum_{j\in\s_n}\frac{\mu(\widetilde{R}_{n,j})}{\mu(\widetilde{E}_n)}\frac{\mu|_{R_{n,j}}}{\mu(R_{n,j})}=\frac{1}{\mu(E_n)}\mu|_{E_n}.\]
In the following, in order  to obtain  the large intersection property  of $\limsup_{n\in\mathcal{N}, n\to \infty}E_n$, we shall verify the conditions in Lemma \ref{lemmaforlb} for $(\mu_n)_{n\in\N}$.

For any ball $B\subset  X$, it follows that
\begin{equation}\label{limsup}
 \begin{split}
 \limsup_{n\in\mathcal{N},n\to \infty}\frac{\mu_n(B)}{\mu(B)}&= \limsup_{n\in\mathcal{N},n\to \infty}\frac{1}{\mu(B)}\sum_{j\in\s_n\atop \widetilde{R}_{n,j}\cap B\ne\varnothing}\frac{\mu(\widetilde{R}_{n,j})}{\mu(\widetilde{E}_n)}\frac{\mu(B\cap R_{n,j})}{\mu(R_{n,j})}\\
 &\lesssim \limsup_{ n\in\mathcal{N},n\to \infty}\frac{1}{\mu(B)}\sum_{j\in\s_n\atop \widetilde{R}_{n,j}\cap B\ne\varnothing}\mu(\widetilde{R}_{n,j})\\
 &\le  \limsup_{n\in\mathcal{N},n\to \infty}\frac{1}{\mu(B)}\mu(3B)\asymp 1.
 \end{split}
 \end{equation}
 On the other hand, assume without loss of generality that $\max_{1\le i\le d}r_n^{u_{i,n}}<\frac{1}{5}|B|$ when $n$ is sufficiently large.  Then, for such $n\in\N$, 
 \[\Big\{j\in\s_n\colon 5^{\hat{c}}\widetilde{R}_{n,j}\cap \frac{1}{5}B\ne\emptyset\Big\}\subset \Big\{j\in\s_n\colon \widetilde{R}_{n,j}\cap \frac{1}{2}B\ne\emptyset\Big\}.\]
It follows that
\begin{equation}\label{liminf}
 \begin{split}
\liminf_{n\in\mathcal{N}\atop n\to \infty}\frac{\mu_n(B)}{\mu(B)} &\gtrsim \liminf_{n\in\mathcal{N}\atop n\to \infty}\frac{1}{\mu(B)}\sum_{j\in\s_n\atop \widetilde{R}_{n,j}\cap \frac{1}{2}B\ne \varnothing}\mu(\widetilde{R}_{n,j})\frac{\mu(B\cap R_{n,j})}{\mu(R_{n,j})}\\
 &= \liminf_{n\in\mathcal{N}\atop n\to \infty}\frac{1}{\mu(B)} \sum_{j\in\s_n\atop \widetilde{R}_{n,j}\cap \frac{1}{2}B\ne \varnothing}\mu(\widetilde{R}_{n,j})\\
&\gtrsim \liminf_{n\in\mathcal{N}\atop n\to \infty}\frac{1}{\mu(B)}\sum_{j\in\s_n\atop 5^{\hat{c}}\widetilde{R}_{n,j}\cap \frac{1}{5}B\ne\varnothing}\mu(5^{\hat{c}}\widetilde{R}_{n,j} \cap B)\\
 &\gtrsim\liminf_{n\in\mathcal{N}\atop n\to \infty}\frac{1}{\mu(B)}\mu\Big(\bigcup_{j\in\s_n}5^{\hat{c}}\widetilde{R}_{n,j}\cap \frac{1}{5}B\Big)\\
 &\ge \liminf_{n\in\mathcal{N}\atop n\to \infty}\frac{1}{\mu(B)}\mu\Big(\bigcup_{j=1}^{k_n}\widetilde{R}_{n,j}\cap \frac{1}{5}B\Big)\gtrsim 1,
 \end{split}
 \end{equation}
 where the last inequality follows from
\eqref{measureofen}. Therefore, by inequalities \eqref{limsup} and \eqref{liminf}, we verify that $(\mu_n)_{n\in\mathcal{N}}$ satisfies \eqref{lip}. Specifically, we will show that for any $s<t$, where $t=t(({\bf u}(n))_{n\in\N},({\bf v}(n))_{n\in\N})$ is defined as in \eqref{defoftuv}, there exists $C>0$ such that
 \[\mu_n(B({\bf x},r))\le Cr^s\]
holds for all ${\bf x}\in X$, $r>0$ and $n\in\mathcal{N}$.

For any  $n\in\mathcal{N}$, let
\[\mathcal{A}_n=\{v_{i,n},u_{i,n}\colon 1\le i\le d\}\]
Assume that $v_{d,n}\in \mathcal{A}_n$ is the largest element in $\mathcal{A}_n$ and $u_{1,n}\in \mathcal{A}_n$ is the smallest one. In the following, fix  $n\in\mathcal{N}$, ${\bf x}\in X$, $r>0$ and write $B=B({\bf x},r)$. We then estimate the $\mu_n$-measure of $B$ by considering the following cases depending on the size of $r$.\medskip

\noindent $\bullet$ \textit{Estimating $\mu_n(B)$ when $0<r<r_n^{v_{d,n}}$}. 
In this case, by \eqref{meaofen} and the definition of $\mu_n$, we have 
\begin{equation*}
\begin{split}
\mu_n(B)=\frac{\mu(B\cap E_n)}{\mu(E_n)}\lesssim r^{\delta_0}\prod_{i=1}^dr_n^{-\delta_i(v_{i,n}-u_{i,n})}<r^{\sum_{i=1}^d\delta_i\big(1-\frac{v_{i,n}-u_{i,n}}{v_{d,n}}\big)}.
\end{split}
\end{equation*}
\medskip

\noindent $\bullet$ \textit{Estimating $\mu_n(B)$ when $r\ge r_n^{u_{1,n}}$}. In this case, observe that 
\[\bigcup_{j\in\mathcal{S}_n:\,\widetilde{R}_{n,j}\cap B\ne \emptyset }\widetilde{R}_{n,j}\subset (1+2\sqrt{d})B,\]
and thus
\[\#\{j:\widetilde{R}_{n,j}\cap B\ne \emptyset \}\cdot\mu(\widetilde{R}_{n,j})\le \mu\left((1+ 2\sqrt{d})B\right).\]
It follows that
\begin{eqnarray}\label{numofjcapb}
    \#\{j:\widetilde{R}_{n,j}\cap B\ne \emptyset \}\le \frac{\mu\left((1+  2\sqrt{d})B\right)}{\mu(\widetilde{R}_{n,j})}\asymp \frac{r^{\delta_0}}{\prod_{i=1}^dr_n^{\delta_iu_{i,n}}}.
\end{eqnarray}
Since $R_{n,j}\subseteq \widetilde{R}_{n,j}$ for any $1\leq j\leq k_n$, the inequality \eqref{numofjcapb} yields that the ball $B$  intersects  
\[\lesssim \frac{r^{\delta_0}}{\prod_{i=1}^dr_n^{\delta_iu_{i,n}}}\]
rectangles in $\{R_{n,j}\}_{j\in\mathcal{S}_n}$. Thus
\begin{equation*}
\begin{split}
\mu_n(B)&\le \frac{1}{\mu(E_n)}\sum_{j\in\s_n\atop R_{n,j}\cap B\ne\varnothing}\mu( R_{n,j})\lesssim \prod_{i=1}^dr_n^{\delta_i(u_{i,n}-v_{i,n})}\cdot \frac{r^{\delta_0}}{\prod_{i=1}^dr_n^{\delta_iu_{i,n}}}\cdot \prod_{i=1}^dr_n^{v_{i,n}\delta_i}
=r^{\delta_0}.
\end{split}
\end{equation*}
\medskip

\noindent $\bullet$ \textit{Estimating $\mu_n(B)$ when $r_n^{v_{d,n}}\leq r<r_{n}^{u_{1,n}}$}. Let
\[
\mathcal{S}_n(B):=\left\{j\in\mathcal{S}_n:B\cap R_{n,j}\neq\emptyset\right\}.
\]
By \eqref{rangofrn} and the locally isometric property of $f_n$, it is easily seen that  $$\big\{\big(f_n|_{B({\bf x}_{n,j},\varepsilon_0)}\big)^{-1}(B)\big\}_{j\in\mathcal{S}_n(B)}$$ are the same ball in $Y$ with radius $r$, denoted as   $B({\bf y},r)$, where ${\bf y}=({\bf y}_1,...,{\bf y}_d)\in Y$  and $f_n({\bf y})$ is the center of $B$.
Note that
\begin{equation}\label{case3munb}
\begin{split}
\mu_n(B)&=\frac{\mu(B\cap E_n)}{\mu(E_n)}= \frac{1}{\mu(E_n)}\cdot\sum_{j\in\mathcal{S}_n(B)}\mu(B\cap R_{n,j})\\
&\asymp\frac{1}{\mu(E_n)}\cdot\sum_{j\in \cS_n(B)}\nu\left(\big(f_n|_{B({\bf x}_{n,j},\varepsilon_0)}\big)^{-1}(B)\cap\big({\bf x
}_{n,j}+\prod_{i=1}^d B_{i,n}\big)\right)\\
&=\frac{1}{\mu(E_n)}\cdot\#\cS_n(B)\cdot\prod_{i=1}^d\big(\min\{r,r_n^{v_{i,n}}\}\big)^{\delta_i}.
\end{split}
\end{equation}
{Clearly, to bound the $\mu_n$-measure of $B$, it is required to estimate the number of elements in $\cS_n(B)$ and the value of $\min\{r,r_{n}^{v_{i,n}}\}$ for any $1\leq i\leq d$}.
Recall that 
\[\mathcal{U}_n=\{v_{i,n},u_{i,n}\colon 1\le i\le d\}.\]
Arrange the elements in $\mathcal{U}_n$ as follows
\[
t_1<t_2<\cdots<t_{\#\mathcal{U}_n}.
\]
Since $r_{n}^{v_{d,n}}\leq r<r_{n}^{u_{1,n}}$,  there exists $1\leq k\leq \#\mathcal{U}_n$ such that 
\[r_n^{t_{k+1}}\le r<r_n^{t_k}.\]
For such $k$, recall that 
\[\K_1(t_k)=:\{i\colon u_{i,n}>t_k\}
,\quad \K_2(t_k):=\{i\colon v_{i,n} \le t_k\},\]
and
\[\K_3(t_k):=\{1,2,\dots,d\}\setminus (\K_1(t_k)\cup \K_2(t_k)).\]
By definitions, we also have
\begin{equation}\label{ni1}
\K_1(t_k)=\{i\colon u_{i,n}\ge t_{k+1}\}
\end{equation}
and 
\begin{equation}\label{ni2}
\K_2(t_k)=\{i\colon t_{k+1}>v_{i,n}\}.
\end{equation}
{With these notations in mind, observe that
\begin{itemize}
    \item[(i)] by \eqref{ni1}, for any $i\in\K_1(t_k)$, we have $r\geq r_n^{u_{i,n}}$ and thus $\min\{r,r_n^{v_{i,n}}\}=r_n^{v_{i,n}}$.
    \medskip

    \item[(ii)] for any $i\in\K_2(t_k)$, we have $r\leq r_n^{v_{i,n}}$ and thus $\min\{r,r_n^{v_{i,n}}\}=r$.
    \medskip

    \item[(iii)] by \eqref{ni2}, for any $i\in\K_3(t_k)$, we have $r_n^{v_{i,n}}\leq r<r_n^{u_{i,n}}$ and thus $\min\{r,r_n^{v_{i,n}}\}=r_n^{v_{i,n}}$. 
\end{itemize}
The above observations yield that the rectangles $$\left\{{\bf x}_{n,j}+\prod_{i=1}^d B_{i,n}^{s_{i,n}}\right\}_{j\in\cS_n(B)}$$ is contained within
\begin{eqnarray*}
    H:=\prod_{i=1}^d B({\bf y}_i,\varepsilon_i),\qquad\text{where} \ \varepsilon_i=\left\{
    \begin{aligned}
        & 2r, \quad \ \ \text{if} \  i\in\K_1(t_k),\\
        & 2r_n^{u_{i,n}}, \ \text{if} \ i\in\K_2(t_k)\cup\K_3(t_k).
    \end{aligned}\right.
\end{eqnarray*}
Then, by the volume argument, we obtain that
\begin{eqnarray*}
    \#\cS_n(B)\lesssim\frac{\prod_{i=1}^d\varepsilon_i^{\delta_i}}{\prod_{i=1}^d r_n^{\delta_iu_{i,n}}}\asymp\prod_{i\in\K_1(t_k)}\Big(\frac{r}{r_n^{u_{i,n}}}\Big)^{\delta_i}.
\end{eqnarray*}
Combining it with the formulas \eqref{meaofen} and \eqref{case3munb}  gives that
\begin{eqnarray*}
    \mu_n(B)&\lesssim& \prod_{i=1}^d r_n^{-\delta_i(v_{i,n}-u_{i,n})}\cdot\prod_{i\in\K_1(t_k)}\Big(\frac{r}{r_n^{u_{i,n}}}\Big)^{\delta_i}\cdot\prod_{i\in\K_1(t_k)\cup\K_3(t_k)}r_n^{\delta_iv_{i,n}}\cdot\prod_{i\in\K_2(t_k)}r^{\delta_i}\\[4pt]
    &=&r^{\sum_{i\in \K_1(t_k)\cup \K_2(t_k)}\delta_i}\cdot r_n^{\sum_{i\in \K_2(t_k)}\delta_i(u_{i,n}-v_{i,n})+\sum_{i\in \K_3(t_k)}\delta_iu_{i,n}}.
\end{eqnarray*}}Next, we finish estimating $\mu_n(B)$ by considering the following two cases:
\begin{itemize}
\item [$\circ$]
If  $\sum_{i\in \K_2(t_k)}\delta_i(u_{i,n}-v_{i,n})+\sum_{i\in \K_3(t_k)}\delta_iu_{i,n}\ge 0$, recall that $r_n^{t_{k+1}}\le r$, then 
\begin{equation*}
\begin{split}
\mu_n(B)&\lesssim  r^{\sum_{i\in \K_1(t_k)\cup \K_2(t_k)}\delta_i+\frac{\sum_{i\in \K_2(t_k)}\delta_i(u_{i,n}-v_{i,n})+\sum_{i\in \K_3(t_k)}\delta_iu_{n,i}}{t_{k+1}}}=: r^{s_1}.
\end{split}
\end{equation*}
It follows from Lemma \ref{elements in an} that $s_1= \mathfrak{s} ({\bf u}(n),{\bf v}(n),t_{k+1})$.
\vspace*{2ex}

\item [$\circ$]
If  $\sum_{i\in \K_2(t_k)}\delta_i(u_{i,n}-v_{i,n})+\sum_{i\in \K_3(t_k)}\delta_iu_{i,n}< 0$, since $r<r_n^{t_k}$, one has
\begin{equation*}
\begin{split}
\mu_n(B)&\lesssim  r^{\sum_{i\in \K_1(t_k)\cup \K_2(t_k)}\delta_i+\frac{\sum_{i\in \K_2(t_k)}\delta_i(u_{i,n}-v_{i,n})+\sum_{i\in \K_3(t_k)}\delta_iu_{n,i}}{t_k}}=r^{ \mathfrak{s} ({\bf u}(n),{\bf v}(n),t_k)}.
\end{split}
\end{equation*}
\end{itemize}
This shows that $\mu_n(B)\lesssim r^{ \min_{a\in\mathcal{U}_n}\{\mathfrak{s}({\bf u}(n),{\bf v}(n),a)\}}$ when $r_n^{v_{d,n}}\leq r<r_n^{v_{1,n}}$.
\medskip

Combining the above arguments, we derive that
\[\mu_n(B)\lesssim r^{ \min_{a\in\mathcal{U}_n}\{\mathfrak{s}({\bf u}(n),{\bf v}(n),a)\}}=r^{ \min_{1\le i\le d}\{s({\bf u}(n),{\bf v}(n),i)\}},\]
where the last equality follows from Lemma \ref{lemma1}.  
Hence, for any $s<t=t(({\bf u}(n))_{n\in\N},({\bf v}(n))_{n\in\N})$ and sufficiently large $n\in\mathcal{N}$, by the definition of $t$, we have 
\[\mu_n(B)\lesssim r^s\]
for all balls $B$ in $X$. It follows from Lemma \ref{lemmaforlb} that $\limsup_{n\in\mathcal{N},n\to \infty}E_n\in\mathcal{G}^s(X)$ for any $ s<t$. Then, by definition, $\limsup_{n\in\mathcal{N},n\to \infty}E_n\in\mathcal{G}^{t}(X)$, and thus
\[\limsup_{n\to\infty}\bigcup_{1\le j\le k_n}f_n\Big({\bf x}_{n,j}+\prod_{i=1}^dB_{i,n}\Big)\in \mathcal{G}^{t}(X).\]
The proof is complete.
\end{proof}

\subsection{Proof of Theorem \ref{MTP2}}

Since $v_{i,n}\to v_i$, $u_{i,n}\to u_i$ as $n\to\infty$ for any $1\le i\le d$, $r_n\to0$ as $n\to\infty$ and the condition ({\bf H}) holds,
it follows that for any $\eta >0$ there exist an integer $M_1=M_1(\eta)\in \mathbb{N}$, a sequence of points $\{{\bf y}_{i,n}\}_{n\in\N}\subset Y_i$ for any $1\leq i\leq d$, a sequence of real numbers $\{c_n\}_{n\in\N}\subset[1,+\infty)$ and  a local isometry $f:Y\to X$ with the same constant $\varepsilon_0>0$ as $(f_n)_{n\in\N}$ associated with \eqref{locisodef} that satisfy the following statements:
\begin{itemize}
    \item[(i)] for all $n\ge M_1$, we have
   \begin{eqnarray}\label{uinge1eta2}
     (1-\frac{\eta}{2})\cdot u_i < u_{i,n} <(1+\frac{\eta}{2})\cdot u_i\quad \forall \  1\le i\le d\quad{\rm and}\quad \left(\sum_{i=1}^dr_n^{2(1-\eta)\cdot u_i}\right)^{1/2}<\frac{\varepsilon_0}{4}.\medskip
    \end{eqnarray}
\item[(ii)] $c_n\leq r_n^{-\eta\cdot\min_{1\leq i\leq d}u_i/2}$ for any $n\geq M_1$.
\medskip

\item[(iii)] for  
any $n\in\N$ and any rectangle $R=\prod_{i=1}^d B({\bf x}_i,r_i)\subset Y$ with ${\rm diam}(R)<\varepsilon_0$, there are some vectors $({\bf y}_{R,i,n})_{n\in\N}$ in $Y_i$ for any $1\leq i\leq d$ that satisfy
\begin{eqnarray*}\label{repalacefnbyf}
    f\Big(\prod_{i=1}^d B({\bf y}_{R,i,n},c_n^{-1}r_i)\Big)\subset f_n(R)\subset f\Big(\prod_{i=1}^d B({\bf y}_{R,i,n},c_nr_i)\Big).
\end{eqnarray*}
\end{itemize}
Fix $\eta>0$. For the sake of simplicity, we assume that $M_1=M_1(\eta)=1$ henceforth. By the above (i)-(iii), there exists $({\bf y}_{i,n})_{n\in\N}\subset Y_i$ for any $1\leq i\leq d$ such that
\begin{eqnarray}\label{1etavin}
    f\Big(\prod_{i=1}^d B({\bf y}_{i,n},r_n^{(1+\eta)\cdot v_{i,n}})\Big)\subset f_n\Big(\prod_{i=1}^dB_{i,n}\Big)\subset f\Big(\prod_{i=1}^d B({\bf y}_{i,n},r_n^{(1-\eta)\cdot v_{i,n}})\Big)
\end{eqnarray}
and
\begin{eqnarray}\label{fbyinuinsin}
    f\Big(\prod_{i=1}^d B({\bf y}_{i,n},r_n^{(1+\eta)\cdot u_{i}})\Big)\subset f_n\Big(\prod_{i=1}^dB_{i,n}^{s_{i,n}}\Big)\subset f\Big(\prod_{i=1}^d B({\bf y}_{i,n},r_n^{(1-\eta)\cdot u_{i}})\Big).
\end{eqnarray}
For any $n\ge 1$, denote
\[R_n:=f\Big(\prod_{i=1}^dB({\bf y}_{i,n},r_n^{(1-\eta)\cdot u_i})\Big).
\]
By \eqref{fbyinuinsin} and the full measure condition for $\big(\prod_{i=1}^dB_{i,n}^{s_{i,n}}\big)_{n\in\N}$, we have
$$\mu\Big(\limsup_{n\to\infty}R_n\Big)=\mu(X)$$ and thus $$\mu\Big(\bigcup_{n=m}^\infty R_n\Big)=\mu(X)$$ for every $m\ge1$.
Therefore, for any $m\ge1$, there exists $N_m\ge m$ such that 
\[\mu\Big(\bigcup_{n=m}^{N_m}R_n\Big)>\mu(X)-\frac{1}{m}.\]
By the second inequality in \eqref{uinge1eta2}, the hypothesis of Lemma \ref{coverlemma} is satisfied for $(R_n)_{m\leq n\leq N_m}$. Then, this lemma shows that there exists $\mathcal{I}_m\subset \{m,m+1,\dots,N_m\}$ such that the sets in $(R_n) _{n\in \mathcal{I}_m}$ are disjoint and 
\[\bigcup_{n=m}^{N_m}R_n\subset \bigcup_{n\in \mathcal{I}_m}5^{\hat{c}}R_n,\qquad\text{where $\hat{c}=\frac{\max_{1\leq i\leq d}u_i}{\min_{1\leq i\leq d}u_i}$}.\]

For any $1\le i\le d$ and $m\geq1$,  denote 
\[\tilde{v}_{i,m}:=\max_{n\in\mathcal{I}_m}\{v_{i,n}\}.\]
Given any $n\in\N$, let
\[U_n:=f\Big(\prod_{i=1}^dB\big({\bf y}_{i,n},r_n^{(1+\eta)\cdot\tilde{v}_{i,m}}\big)\Big).\]
By \eqref{1etavin}, we have
\begin{eqnarray}\label{unsubsetrn}
    U_n\subset f\Big(\prod_{i=1}^dB\big({\bf y}_{i,n},r_n^{(1+\eta)\cdot v_{i,n}}\big)\Big)\subset f_n\Big(\prod_{i=1}^dB_{i,n}\Big)\subset R_n,\qquad\forall \ n\geq 1.
\end{eqnarray}
Let
$$E_m:=\bigcup_{n\in \mathcal{I}_m}U_n\qquad\forall \ m\in\N.$$ Then, by \eqref{unsubsetrn}, 
$$\limsup\limits_{m\to\infty}E_m\subset \limsup\limits_{n\to\infty} f_n\Big(\prod_{i=1}^dB_{i,n}\Big).$$
We next prove that  $\limsup_{m\to\infty}E_m\in\mathcal{G}^{s}(X)$ for any $s<s((1-\eta){\bf u},(1+\eta){\bf v})$ making use of Lemma \ref{lemmaforlb}. 

For any $m\in\N$,  define the probability measure $\mu_m$ supported on $E_m$ as
\begin{equation*}\label{measure}
\begin{split}
\mu_m:&=\sum_{n\in \mathcal{I}_m}\frac{\mu(R_n)}{\mu(\bigcup_{k\in \mathcal{I}_m}R_k)}\frac{\mu|_{U_n}}{\mu(U_n)}\\[4pt]
&\asymp \sum_{n\in \mathcal{I}_m}\frac{\mu(R_n)}{\mu(U_n)}\mu|_{U_n}.
\end{split}
\end{equation*}
 We first prove that $(\mu_m)_{m\in\N}$ satisfies inequality \eqref{lip}.  For any ball $B:=B({\bf x},r)\subset  X$ and  sufficiently large $m\in\N$, we have
\begin{equation}\label{measure1}
\begin{split}
\mu_m(B)\asymp\sum_{n\in \mathcal{I}_m}\mu(R_n) \frac{\mu(U_n\cap B)}{\mu(U_n)}
\le \sum_{n\in \mathcal{I}_m\atop R_n\cap B\ne\emptyset}\mu(R_n)\le \mu\left(\frac{3}{2}B\right)\asymp\mu(B),
\end{split}
\end{equation}
where the inequality from the third term to the fourth term follows from the fact that $\max_{n\in\mathcal{I}_m}\{{\rm diam} (R_n)\}<r/2$ when $m\in\N$ is sufficiently large. On the other hand, since $ {\rm diam} (5^{\hat{c}}R_m)<\frac{1}{4}r$ for sufficiently large $m\in\N$, we have 
\[\Big\{n\in\mathcal{I}_m\colon \frac{1}{2}B\cap 5^{\hat{c}}R_n\ne\emptyset\Big\}\subset \Big\{n\in\mathcal{I}_m\colon \frac{3}{4}B\cap U_n\ne\emptyset\Big\}.\]
Then, for any large $m\in\N$,
\begin{equation}\label{measure2}
\begin{split}
\mu_m(B)&\gtrsim \sum_{n\in \mathcal{I}_m\atop U_n\cap \frac{3}{4}B\ne\emptyset}\mu(R_n) \frac{\mu(U_n\cap B)}{\mu(U_n)}= \sum_{n\in \mathcal{I}_m\atop U_n\cap \frac{3}{4}B\ne\emptyset}\mu( R_n)\\[4pt]
&\asymp \sum_{n\in \mathcal{I}_m\atop U_n\cap \frac{3}{4}B\ne\emptyset}\mu( 5^{\hat{c}}R_n)\ge \sum_{n\in \mathcal{I}_m\atop 5^{\hat{c}}R_n\cap \frac{1}{2}B\ne\emptyset}\mu\Big(\frac{1}{2}B\cap 5^{\hat{c}}R_n\Big)\\[4pt]
&> \mu\Big(\frac{1}{2}B\Big)-\frac{1}{m}.
\end{split}
\end{equation}
Therefore,
\[1\lesssim \liminf_{m\to\infty}\frac{\mu_m(B)}{\mu(B)}\le \limsup_{m\to\infty}\frac{\mu_m(B)}{\mu(B)}\lesssim1,\]
where the implicit constant is independent of $B$.

In view of Lemma \ref{lemmaforlb}, it is left to show that for any $s<s((1-\eta){\bf u},(1+\eta){\bf v})$  there exist  $C>0$ and $N\in\N$ such that $\mu_m(B)\le  Cr^{s}$ holds for any $m\geq N$ and any  ball $B=B({\bf x},r)\subset X$. Fix $s<s((1-\eta){\bf u},(1+\eta){\bf v})$. Recall that $v_{i,n}\to v_i$ as $n\to\infty$ for any $1\leq i\leq d$. Then, by Lemma \ref{lemma2}, there exists $M_2=M_2(s,\eta)\in\mathbb{N}$ such that 
\begin{equation}\label{t}
s<\min_{1\le i\le d}\{s((1-\eta){\bf u}, (1+\eta)\widetilde{\bf v}(m),i)\}\qquad\forall \ m\ge M_2,
\end{equation} 
where $\widetilde{\bf v}(m):=(\tilde{v}_{1,m},\dots,\tilde{v}_{d,m})$.
In the following, fix $m\geq M_2$ and a ball $B=B({\bf x},r)\subset X$ with radius $r>0$. We now turn to estimating $\mu_m(B)$. 
 We divide $\mathcal{I}_m$ into several subclasses 
\[\mathcal{F}_p:=\{n\in \mathcal{I}_m\colon 2^{-p}\le r_n<2^{-p+1}\}\qquad\forall \ p\in\N.\]
Denote $p_0$ as the smallest integer among those $p\in\N$ such that $\mathcal{F}_{p}\ne\emptyset$. 
Then 
\begin{equation}\label{measure3}
\begin{split}
\mu_m(B)&\asymp\sum_{p\ge p_0}\sum_{n\in \mathcal{F}_p\atop U_n\cap B\ne \emptyset}\mu(R_n)\frac{\mu(U_n\cap B)}{\mu(U_n)}.
\end{split}
\end{equation}
Throughout, we use $\tilde{v}_{i}$ instead of $\tilde{v}_{i,m}$ and $\widetilde{{\bf v}}$ instead of $\widetilde{\bf v}(m)$ for short. Let 
\[\widetilde{\mathcal{A}}=\big\{(1+\eta)\cdot\tilde{v}_{i},\,(1-\eta)\cdot u_i\colon 1\le i\le d\big\}\]
Assume that $(1-\eta)\cdot u_1$ is the smallest element in $\widetilde{\mathcal{A}}$ and $(1+\eta')\cdot\tilde{v}_d$ is the largest.
Without loss of generality, assume that $r<1$. Then there exists unique integer $p_1\ge1$   such that 
\[2^{-p_1\cdot(1-\eta)\cdot u_1}\le r<2^{-(p_1-1)\cdot(1-\eta)\cdot u_1}.\]
We estimate the $\mu_m$-measure of $B$ by considering the following cases.\medskip

\noindent $\bullet$ \textit{Estimating $\mu_m(B)$ when $p_1<p_0$}.
In this case, $r>\max_{n\in\mathcal{I}_m}r_n^{(1-\eta)u_1}.$ It follows that any rectangle $R_n$  that intersects $B$ is contained within $3B$.
Hence,
\[\mu_m(B)\le \sum_{n\in\mathcal{I}_m\atop U_n\cap B\ne\emptyset }\mu(R_n)\le \mu(3B)\lesssim r^{\delta_0}.\]
Observe that $s<s((1-\eta){\bf u},(1+\eta){\bf v})\leq \delta_0$,  then 
\begin{equation}\label{estimate1}
\mu_m(B)\lesssim r^{\delta_0}< r^s.\medskip
\end{equation}

\noindent $\bullet$ \textit{Estimating $\mu_m(B)$ when $p_1\ge p_0$}. 
For any $p\geq p_0$, denote 
\[I_p:=\sum_{n\in \mathcal{F}_p}\mu(R_n)\frac{\mu(U_n\cap B)}{\mu(U_n)}.\]
Then, 
\begin{equation*}
\begin{split}
\mu_m(B)&\asymp\sum_{p\ge p_0}\sum_{n\in \mathcal{F}_p\atop U_n\cap B\ne \emptyset}\mu(R_n)\frac{\mu(U_n\cap B)}{\mu(U_n)}\\
&=\sum_{p= p_0}^{p_1-1}I_p+\sum_{p= p_1}^{\infty}I_p.
\end{split}
\end{equation*}
For the sake of convenience, given any $p\geq p_0$ and  $n\in \mathcal{F}_p$,  we assume that $r_n=2^{-p}$, then $\{R_n\}_{n\in\mathcal{F}_p} $ and $\{U_n\}_{n\in\mathcal{F}_p}$ share the same shape respectively. Next, we bound the value of $I_p$ according to the range of $p$.\medskip

\textit{Case 1: $p\ge p_1$}. The method is similar to that used in the case $p_1< p_0$. Note that any rectangle $R_n$ with  $n\in\bigcup_{p\ge p_1} \mathcal{F}_p$  that intersects $B$ is contained in $3B$.
From this, we derive that
\begin{equation}\label{estimate2}
\sum_{p= p_1}^{\infty}I_p \le \mu(3B)\asymp r^{\delta_0}< r^{s}.\medskip
\end{equation}

\textit{Case 2: $p<p_1$}.
In this case, $ r<2^{-(p_1-1)\cdot(1-\eta)\cdot u_1}\le 2^{-p\cdot(1-\eta)\cdot u_1}=r_n^{(1-\eta)\cdot u_1}$ for any  $n\in\mathcal{F}_p$; that is to say that $r$ is less than the largest  side of $R_n$.
Next, we estimate the value of $I_p$ 
 depending on the relationship between the size of $r$ and that of the smallest side of the rectangle $U_n$ (namely $2^{-p\cdot(1+\eta)\cdot\widetilde{v}_d}$) with $n\in\mathcal{F}_p$.

\textit{Subcase 1}:  $r<2^{-p\cdot(1+\eta)\cdot\tilde{v}_d}$. Observe that there exists $M>0$ that satisfies the following statement: \textit{if $\{R\}_{R\in\mathcal{R}}$ is a collection of  rectangles in $Y$ whose diameters are less than $\varepsilon_0$ such that $\{f(R)\}_{R\in\mathcal{R}}$ are disjoint, and $B\subseteq X$ is a ball whose radius is less than the size of the smallest side of $R$ for any $R\in\mathcal{R}$, then $\#\{R\in\mathcal{R}:f(R)\cap B\neq\emptyset\}\leq M$}. 
With this observation in mind,  we have
\begin{equation*}
\begin{split}
I_p&\leq\sum_{n\in\mathcal{F}_p\atop U_n\cap B\neq\emptyset}\mu(B)\cdot\frac{\mu(R_n)}{\mu(U_n)}\leq M\cdot\mu(B)\cdot\max_{n\in\mathcal{F}_p}\frac{\mu(R_n)}{\mu(U_n)}\\[4pt]
&\asymp  r^{\delta_0-\frac{1}{(1+\eta)\cdot\tilde{v}_d}\sum_{i=1}^d\delta_i((1-\eta)\cdot u_i-(1+\eta)\cdot\tilde{v}_i)}\\[4pt]
&= r^{s((1-\eta)\cdot{\bf u},(1+\eta)\cdot\widetilde{{\bf v}},d)}<r^{s}.
\end{split}
\end{equation*}

\textit{Subcase 2}: $2^{-p\cdot(1+\eta)\cdot\tilde{v}_d}\le r< 2^{-p\cdot(1-\eta)\cdot u_1}$.
Arrange the elements in $\widetilde{\mathcal{A}}$ as the following
\[
t_1<t_2<\cdots t_{\#\widetilde{\mathcal{A}}}.
\]
In this case, there exists $1\leq k\leq \#\widetilde{\mathcal{A}}$ such that 
\[2^{-p\cdot t_{k+1}}\le r<2^{-p\cdot t_k}.\]
Note that $t_k$, $t_{k+1}$ are consecutive, then we have
\[\mathcal{K}_1:=\mathcal{K}_1(t_k)=\{i:(1-\eta)\cdot u_i>t_k\}=\{i:(1-\eta)\cdot u_i\ge t_{k+1}\},\]
\[\mathcal{K}_2:=\mathcal{K}_2(t_k)=\{i:(1+\eta)\cdot\tilde{v}_i\le t_k\}=\{i:(1+\eta)\cdot\tilde{v}_i< t_{k+1}\},\]
and
\[\mathcal{K}_3:=\mathcal{K}_3(t_k)=\{1,2,\dots,d\}\setminus (\mathcal{K}_1\cup \mathcal{K}_2).\]
By definition, it follows that
\begin{itemize}
    \item[(a)]  $r\ge 2^{-p\cdot t_{k+1}}\ge  2^{-p\cdot(1-\eta)\cdot u_i}$ for any $i\in \mathcal{K}_1$.
    \medskip

    \item[(b)]  $r<2^{-p\cdot t_k}\le 2^{-p\cdot (1+\eta)\cdot\tilde{v}_i}$ for any $i\in \mathcal{K}_2$.
    \medskip

    \item[(c)]  $2^{-p\cdot(1+\eta)\cdot\tilde{v}_i}<r<2^{-p\cdot(1-\eta)\cdot u_i}$ for any $i\in \mathcal{K}_3$.
\end{itemize} 
Let $\cS_p(B):=\{n\in\mathcal{F}_p:U_n\cap B\neq\emptyset\}$. Since $f:Y\to X$ is a local isometry, it is easily seen via \eqref{uinge1eta2} that $\big\{\big(f|_{B({\bf y}_n,\varepsilon_0)}\big)^{-1}(B)\big\}_{n\in\mathcal{S}_p(B)}$ are the same ball with radius $r$, where ${\bf y}_n=({\bf y}_{1,n},{\bf y}_{2,n},...,{\bf y}_{d,n})$. Denote the ball as $B({\bf x}',r)$, where ${\bf x}'=({\bf x}'_1,{\bf x}'_2,...,{\bf x}'_d)\in Y$ satisfies $f({\bf x}')={\bf x}$ and ${\bf x}$ is the center of the ball $B$. With these notations in mind, we obtain that
\begin{eqnarray}\label{reformultecspb}
    \cS_p(B)\subseteq\left\{n\in\mathcal{F}_p:B({\bf x}',r)\cap\prod_{i=1}^d B({\bf y}_{i,n},2^{-p\cdot(1-\eta)\cdot u_i})\neq\emptyset \right\}.
\end{eqnarray}
Note that the above (a)-(c) yields that all rectangles $\prod_{i=1}^d B({\bf y}_{i,n},2^{-p\cdot(1-\eta)\cdot{u}_i})$ with $n\in\mathcal{S}_p(B)$ is contained within
\begin{eqnarray*}
    H:=\prod_{i=1}^d B({\bf x}'_i,\varepsilon_i),\qquad\text{where}\ \varepsilon_i=\left\{
    \begin{aligned}
        &2r,\qquad\qquad\ \ \, \, \text{if $i\in\K_1$},\\
        &2\cdot 2^{-p\cdot(1-\eta)\cdot u_i},\ \text{if $i\in\K_2\cup\K_3$}.
    \end{aligned}
    \right.
\end{eqnarray*}
 This together with \eqref{reformultecspb} and  the volume argument, we obtain that
\begin{eqnarray*}
    \#\cS_p(B)\lesssim\prod_{i\in\K_1}\Big(\frac{r}{2^{-p\cdot(1-\eta)\cdot u_i}}\Big)^{\delta_i}.
\end{eqnarray*}
 Now, we are ready to estimate the value of $I_p$:
\begin{eqnarray*}
    I_p&=&\sum_{n\in\cS_p(B)}\mu(R_n)\frac{\mu(U_n\cap B)}{\mu(U_n)}\\[4pt]
    &\leq&\#\cS_p(B)\cdot\max\left\{\mu(R_n)\cdot\frac{\mu(U_n\cap B)}{\mu(U_n)}:n\in\cS_p(B)\right\}\\[4pt]
    &\asymp& \#\cS_p(B)\cdot\max\Big\{\nu\big(\prod_{i=1}^d B({\bf y}_{i,n},2^{-p\cdot(1-\eta)\cdot u_i})\big)\\[4pt]
    &~&\hspace{25ex}\times\frac{\nu\big(\prod_{i=1}^{d}B({\bf y}_{i,n},2^{-p\cdot(1+\eta)\cdot\widetilde{v}_i})\cap B({\bf x}',r))\big)}{\nu\big(\prod_{i=1}^{d}B({\bf y}_{i,n},2^{-p\cdot(1+\eta)\cdot\widetilde{v}_i})\big)}:n\in\cS_p(B)\Big\}\\[4pt]
    &\lesssim& \prod_{i\in\K_1}\Big(\frac{r}{2^{-p\cdot(1-\eta)\cdot u_i}}\Big)^{\delta_i}\cdot\prod_{i=1}^d 2^{-p\cdot(1-\eta)\cdot u_i\cdot\delta_i}\cdot\prod_{i=1}^d 2^{p\cdot(1+\eta)\cdot\widetilde{ v}_i\cdot\delta_i}\cdot\prod_{i=1}^d\big(\min_{1\leq i\leq d}\{2^{-p\cdot(1+\eta)\cdot\widetilde{v}_i\cdot\delta_i},r\}\big)^{\delta_i}\\[4pt]
    &=& \prod_{i\in \mathcal{K}_1}\frac{r^{\delta_i}}{2^{-p\cdot(1-\eta)\cdot u_i\cdot\delta_i}}\cdot\prod_{i\in \mathcal{K}_1\cup \mathcal{K}_3}2^{-p\cdot(1+\eta)\cdot\tilde{v}_i\delta_i}\cdot\prod_{i\in \mathcal{K}_2}r\cdot\prod_{i=1}^d2^{-p\cdot\left((1-\eta)\cdot u_i-(1+\eta)\cdot\tilde{v}_i\right)\cdot\delta_i}\\[4pt]
    &=& 2^{-p\big(\sum_{i\in \mathcal{K}_2}\delta_i((1-\eta)\cdot u_i-(1+\eta)\cdot\tilde{v}_i)+\sum_{i\in \mathcal{K}_3}\delta_i(1-\eta)u_i\big)}\cdot \prod_{i\in \mathcal{K}_1\cup \mathcal{K}_2}r^{\delta_i}
\end{eqnarray*}
We finish estimating $I_p$ by considering two cases:
\begin{itemize}
    \item[$\circ$]  If $\sum_{i\in \mathcal{K}_2}\delta_i((1-\eta)\cdot u_i-(1+\eta)\cdot\tilde{v}_i)+\sum_{i\in \mathcal{K}_3}\delta_i(1-\eta)u_i\ge 0$, then by combining Lemma \ref{elements in an} and the fact that $r\ge 2^{-p\cdot t_{k+1}}$, we obtain that
\begin{equation*}
\begin{split}
I_p&\lesssim r^{\sum_{i\in \mathcal{K}_1\cup \mathcal{K}_2}\delta_i+\frac{\sum_{i\in \mathcal{K}_2}\delta_i(u_i(1-\eta)-\tilde{v}_i(1+\eta))+\sum_{i\in \mathcal{K}_3}\delta_iu_i(1-\eta)}{t_{k+1}}}\\[4pt]
&= r^{\mathfrak{s}((1-\eta){\bf u},(1+\eta)\widetilde{{\bf v}},t_{k+1})}.
\end{split}
\end{equation*}
\vspace*{1ex}

\item[$\circ$]  If $\sum_{i\in \mathcal{K}_2}\delta_i((1-\eta)u_i-(1+\eta)\tilde{v}_i)+\sum_{i\in \mathcal{K}_3}\delta_i(1-\eta)u_i< 0$, since $r\le 2^{-p\cdot t_k}$, we have
\begin{equation*}
\begin{split}
I_p&\lesssim r^{\sum_{i\in  \mathcal{K}_1\cup  \mathcal{K}_2}\delta_i+\frac{\sum_{i\in  \mathcal{K}_2}\delta_i((1-\eta)u_i-(1+\eta)\tilde{v}_i)+\sum_{i\in  \mathcal{K}_3}\delta_i(1-\eta)u_i}{t_{k}}}\\[4pt]
&=r^{\mathfrak{s}((1-\eta){\bf u},(1+\eta)\widetilde{{\bf v}},t_k)}.
\end{split}
\end{equation*}
\end{itemize}
\medskip

On combining the above Subcase 1, Subcase 2, Lemma \ref{lemma1} and \eqref{t}, we obtain that for any $p<p_1$, 
\[I_p\lesssim r^{\min_{a\in \widetilde{\mathcal{A}}}\{\mathfrak{s}((1-\eta){\bf u},(1+\eta)\widetilde{{\bf v}},a)\}}=r^{\min_{1\leq i\leq d}s((1-\eta){\bf u},(1+\eta)\widetilde{v}(m),i)}<r^{s}.\]
Therefore, for any $\varepsilon>0$,
\begin{equation}\label{estimate3} 
\sum_{p= p_0}^{p_1-1}I_p\lesssim p_1\cdot r^s\asymp\log 1/r\cdot r^s\lesssim r^{s-\varepsilon}.
\end{equation}
Then, for any $s<s((1-\eta){\bf u},(1+\eta){\bf v})$ and $\varepsilon>0$, \eqref{estimate1}--\eqref{estimate3} give that 
\[\mu_m(B)\lesssim r^{s-\varepsilon}.\]
Applying Lemma \ref{lemmaforlb}, we obtain that $\limsup_{m\to\infty}E_m\in \mathcal{G}^s( X)$ for any $s< s((1-\eta){\bf u},(1+\eta){\bf v})$ for any $\eta>0$ and thus $\limsup_{m\to\infty}E_m\in \mathcal{G}^{ s({\bf u},{\bf v}) }( X)$. 
The proof is complete.

\section{Proof of Theorem \ref{application1}}\label{proofof28}

Recall that$\mathcal{A}=(A_n)_{n\in\N}$ is a sequence of $d\times d$ invertible real matrices, $\{\sigma_i(A_n)\}_{i=1}^d$ is the singular values of $A_n$ and $l_{n,i}:=\frac{1}{n}\log \sigma_i(A_n)$ for any $n\in\N$ and $1\le i\le d$. Given any ${\bf y}\in[0,1)^d$, denote
\[F_n({\bf y}):=\{\mathbf{x}\in  [0,1)^d\colon A_n\mathbf{x}~(\bmod~1) =\mathbf{y}\}.\]
For any ${\bf x}\in\R^d$ and $r>0$, denote $B_{\R^d}({\bf x},r)$ the open ball in $\R^d$ centered at ${\bf x}$ with radius $r$  under the Euclidean norm, namely the $L^2$ norm.
\begin{lem}\label{number:ellipse}
Let $\mathcal{A}=(A_n)_{n\in\N}
$ be a sequence of $d\times d$ invertible real matrices whose singular values are denoted by $(\sigma_i(A_n))_{i=1}^d$ for any $n\in\N$.   Assume that $\Gamma(\mathcal{A})\subset (\R^+)^d$. Then 
\begin{itemize}
\item[(1)] for sufficiently large $n\in\N$ {and for any ${\bf y}\in[0,1)^d$}, we have
\begin{equation*}
\#F_n({\bf y})\asymp \prod_{i=1}^de^{nl_{n,i}};
\end{equation*}
\item[(2)]given $B:=B(\mathbf{x},r)\subset  [0,1)^d$, for sufficiently large $n\in\N$, we have
\[\#(B\cap F_n({\bf y}))\asymp r^d\prod_{i=1}^de^{nl_{n,i}}.\]
\end{itemize}
\end{lem}

\begin{proof}
{We first prove the statement (2).  
Let $B_{\R^d}$ be the unique Euclidean ball  with the same centre  as $B$ such that $B_{\R^d}~(\bmod~1)=B$. 
For any integer $n\ge 1$, by definition, we have
\[\#F_n({\bf y})=\#\{\mathbf{x}\in B_{\R^d}\colon A_n\mathbf{x}-\mathbf{y}\in\mathbb{Z}^d\}=\#\big((A_n B_{\R^d}-\mathbf{y})\cap \mathbb{Z}^d\big),\]
where $A_n B_{\R^d}-\mathbf{y}:=\{A_n{\bf z}-{\bf y}: {\bf z}\in  B_{\R^d}\}$.
Notice that $A_nB_{\R^d}-\mathbf{y}$ contains a rectangle $R$ whose side lengths are $cre^{nl_{n,i}}$ for $1\leq i\leq d$, where $c>0$ is a constant only depending on $\mathcal{A}$. Since $\Gamma(\mathcal{A})\subset (\R^+)^d$, it follows that $l_{n,i}>0,~1\le i\le d$ for sufficiently large $n\in\N$. Therefore, the maximal number of disjoint squares with side length $2$ contained within $R$ is
\[\gtrsim \prod_{i=1}^d re^{nl_{n,i}}.\]
Combining this with the observation that any such square contains at least one integral point, we conclude that 
\[\#\big(\mathbb{Z}^d\cap (A_nB_{\R^d}-\mathbf{y})\big)\gtrsim r^d\prod_{i=1}^d e^{nl_{n,i}}.\]
To obtain the reverse inequality, note that the minimal number of unit squares required to cover $A_n B_{\R^d}-\mathbf{y}$ is
\[\lesssim  \prod_{i=1}^d re^{nl_{n,i}}.\]
 Since each of these squares contains at most one integral point, we have
\[\#\big(\mathbb{Z}^d\cap (A_n B_{\R^d}-\mathbf{y})\big)\lesssim r^d\prod_{i=1}^d e^{nl_{n,i}}.\]
Statement (1) is derived from statement (2) by taking $B=B({\bf 0},\sqrt{d})$. This completes the proof of the lemma.}
\end{proof}

{ The following arguments are helpful to verify the condition appearing in the Theorem \ref{MTP1} for shrinking target sets induced by matrix transformations.}
 
\begin{rem}\label{ellipsoidrec}
Let $r>0$ and 
{$B_{\R^d}:=B_{\R^d}(\mathbf{0}, r)$ be a Euclidean ball in $\R^d$}. Let $\{\sigma_{n,i}\}_{i=1}^d$ be singular values of $A_n$ and $({\bm \alpha}_{n,i})_{i=1}^d$ be unit eigenvectors of $A_n^TA_n$ with
\[A_n^TA_n{\bm \alpha}_{n,i}=\sigma_{n,i}^2{\bm \alpha}_{n,i},\quad 1\le i\le d.\]
Denote ${\bm \beta}_{n,i}:=\frac{(A_n^{T})^{-1}{\bm \alpha}_{n,i}}{|(A_n^{T})^{-1}{\bm \alpha}_{n,i}|}$, $1\le i\le d$. Consider cubes of $\R^d$
\[R:=\left\{x_1{\bm\beta}_{n,1}+\dots+x_d{\bm\beta}_{n,d}\colon |x_i|<\frac{r}{d}\right\}\subset B_{\R^d}\]
and
\[\widetilde{R}:=\left\{x_1{\bm \beta}_{n,1}+\dots+x_d{\bm \beta}_{n,d}\colon |x_i|<r\right\}\supset B_{\R^d}.\] Then, by the left multiplication of the matrix $A_n^{-1}$, we have
\[A_n^{-1}R=\left\{\sigma_{n,1}^{-1}x_1{\bm \alpha}_{n,1}+\dots+\sigma_{n,d}^{-1}x_d{\bm \alpha}_{n,d}\colon |x_i|<\frac{r}{d}\right\}\subset A_n^{-1}B_{\R^d},\]
\[A_n^{-1}\widetilde{R}=\left\{\sigma_{n,1}^{-1}x_1{\bm \alpha}_{n,1}+\dots+\sigma_{n,d}^{-1}x_d{\bm \alpha}_{n,d}\colon |x_i|<r\right\}\supset A_n^{-1}B_{\R^d}.\]
Note that $A_n^{-1}R$ and $A_n^{-1}\widetilde{R}$ are rectangles with side lengths $2\sigma_{n,i}^{-1}r=2re^{-nl_{n,i}}$, $1\le i\le d$. Therefore, the above inclusions allow us to view
each ellipsoid in $\{\mathbf{x}+A_n^{-1} B_{\R^d}(\mathbf{0}, r):~\mathbf{x}\in F_n\}$ as a rectangle with side lengths  $2re^{-nl_{n,k}}$, $1\le k\le d$. 
\end{rem}
 
 \begin{rem}\label{localiso}

Consider the metric space $([0,1)^d,\Vert\cdot\Vert_d)$, where 
\[\Vert{\bf x}\Vert_d=\sqrt{\sum_{i=1}^d{\Vert x_i\Vert}^2},\quad \Vert x_i\Vert={\rm dist}(x,\mathbb{Z}),\quad\forall \ \bfx\in[0,1)^d.\]
 Let $O\in GL_d(\mathbb{R})$ be a rotation matrix, that is, an orthogonal matrix with $\det O=1$. Let $f:\R^d\to[0,1)^d$ be defined as 
 \begin{eqnarray}\label{rotationmod1}
     f({\bf x})=O{\bf x}\m,\qquad\forall \ {\bfx}\in\R^d.
 \end{eqnarray}
 For any ${\bf x},{\bf y}\in\R^d$ with $|{\bfx}-{\bfy}|<1/2$, it is easily verified that $\|\bfy-\bfy\|_d=|{\bf x}-{\bfy}|$, where $|\cdot |$ is the Euclidean norm. Since $O$ preserves Euclidean distance on $\R^d$, it follows that 
 \[\Vert f({\bf x})-f({\bfy})\Vert=\Vert O({\bf x}-{\bfy})\m\Vert=|O({\bf x}-{\bfy})|=|{\bf x}-{\bfy}|.\]
  Therefore, $f$ is a local isometry from $\R^d$ onto $[0,1)^d$ with $\varepsilon_0=\frac{1}{2}$. 
For any rectangle $\hat{R}\subset [0,1)^d$  (such as $A_n^{-1}B_{\R^d}({\bf 0},r_n)\m$ associated with Remark \ref{ellipsoidrec}), there exist a rotation transformation $f_{\hat{R}}$ of the form \eqref{rotationmod1} and a Euclidean rectangle $\prod_{i=1}^d B_i\subseteq[-2,2]^d$ such that for any $\bfx\in[0,1)^d$,
\[\bfx+\hat{R}=f_{\hat{R}}\left(\bfx'+\prod_{i=1}^dB_i\right)\]
for some $\bfx'\in[-2,2]^d$.  
 \end{rem}
Given an integer $n\ge1$, let
\begin{equation*}
\begin{split}
W_n (\mathcal{A},\psi,{\bfy})&:=\{\mathbf{x}\in  [0,1)^d\colon A_n\mathbf{x}~(\bmod~ 1)\in B(\mathbf{y}, \psi(n))\}.
\end{split}
\end{equation*}
Then by Lemma \ref{number:ellipse}, the set $W_n (\mathcal{A},\psi,{\bfy})$ consists of $L_n\asymp \prod_{k=1}^de^{nl_{n,k}}$ disjoint ellipsoids with lengths of semi-axes $\psi(n)e^{-nl_{n,k}}$, $1\le k\le d$. In the following, the estimate of the Hausdorff dimension of $ W\big(\mathcal{A},\psi,\mathbf{y}\big)$ will be divided into two parts: the lower bound and the upper bound.

\subsection{The lower bound in  Theorem \ref{application1}}


{
For any integer $n\ge 1$, as mentioned in Remark~\ref{ellipsoidrec}, the sets $W_{n}(\mathcal{A},\psi,{\bfy})$  can be viewed as the union of rectangles (module one) with side lengths $2\psi(n)e^{-nl_{n,k}}$, $1\leq k\leq d$. With this in mind, let $(W_{n,j})_{j=1}^{L_{n}}$  be the corresponding rectangles associated with $W_{n}(\mathcal{A},\psi,{\bfy})$. As illustrated in the Remark \ref{localiso}, there exist $(\bfx_{n,j})_{j=1}^{L_n}\subseteq[-2,2]$, a rectangle $\prod_{k=1}^dB_{k,n}\subseteq[-2,2]^d$ with $r(B_{k,n})=\psi(n)e^{-nl_{n,k}}$ for any $1\leq k\leq d$ and a local isometry $f_n:[-2,2]^d\to[0,1)^d$ with $\varepsilon_0=1/2$ such that
\begin{equation*}
    W_{n,j}=f_n\Big(\bfx_{n,j}+\prod_{k=1}^dB_{k,n}\Big),\qquad\forall \ 1\leq j\leq L_n.
\end{equation*}
Let $r_n=e^{-n}$, then $r(B_{k,n})=r_n^{v_{k,n}}$ for any $1\leq k\leq d$, where
\begin{equation*}
    v_{k,n}:=l_{n,k}+\frac{\log1/\psi(n)}{n},\qquad\forall \ 1\leq k\leq d.
\end{equation*}
Let $M>\sqrt{d}$ and let $\psi_M:\R^+\to\R^+$ be the constant function  $x \mapsto M$. We denote $\{R_{n,j}\}_{j=1}^{L_n}$ the corresponding rectangles associated with $W_n(\mathcal{A},\psi_M,\bfy)$. Then, we have
\begin{equation*}
    R_{n,j}=f_n\Big(\bfx_{n,j}+\prod_{k=1}^dB_{k,n}^{s_{k,n}}\Big),\qquad\forall \ 1\leq j\leq L_n,
\end{equation*}
where for any $1\leq k\leq d$,
\begin{equation*}
    s_{k,n}:=\frac{u_{k,n}}{v_{k,n}},\qquad u_{k,n}:=l_{n,k}+\frac{\log1/M}{n}.
\end{equation*}
The range of $M$ implies that $(R_{n,j})_{j=1}^{L_n}$ covers $[0,1)^d$ for any $n\in\N$ and thus $\big(\bfx_{n,j}+\prod_{k=1}^dB_{k,n}^{s_{k,n}}\big)_{n\in\N,\,1\leq j\leq L_n}$  satisfies the local ubiquity condition \eqref{condition:measure}. Without loss of generality, assume that $\psi(n)\leq M$ for all $n\in\N$. Then $u_{k,n}\leq v_{k,n}$ for any $1\leq k\leq d$ and $n\in\N$.
}

Let us choose a subsequence $(n_m)_{m\in\N}$ such that 
\begin{equation}\label{subseqconv}
    \lim_{m\to\infty}\min_{1\le k\le d}\{\underline{s}_{n_m}(\mathcal{A},\psi,k)\}=\underline{s}(\mathcal{A},\psi)
\end{equation}
and
\[
\lim_{m\to\infty}u_{k,n_m}=u_k, \quad\lim_{m\to\infty}v_{k,n_m}=v_k,\quad \forall \ 1\leq k\leq d.
\]
for some ${\bf u}=(u_1,...,u_d)\in(\R^+)^d$ and ${\bf v}=(v_1,...,v_d)\in(\R^+\cup\{\infty\})^d$.
In the following, we will use Theorem \ref{MTP1} to estimate the Hausdorff dimensional lower bound of subset of $W(\mathcal{A},\psi,{\bfy})$ defined as
\begin{equation*}
\begin{split}
\widetilde{W} (\mathcal{A},\psi,{\bfy})&:=\{\mathbf{x}\in  [0,1)^d\colon A_{n_m}\mathbf{x}\m\in B(\mathbf{y}, \psi(n_m))~ {\rm for~infinitely~many}~m\in\N\}\\[4pt]
&=\limsup_{m\to\infty}W_{n_m}(\mathcal{A},\psi,{\bfy}).
\end{split}
\end{equation*}
Note that previous arguments demonstrate that
\begin{itemize}
    \item[(i)]  $(B_{i,n_m})_{1\leq i \leq d,m\in\N}$, $({\bf v}(n_m))_{m\in\N}$ and ${\bf v}$ satisfy \eqref{conofbandv};\\[4pt]
    \item[(ii)] $({\bf u}(n_m))_{m\in\N}$ and ${\bf u}$ satisfy \eqref{conditiononu};\\[4pt]
    \item[(iii)]the sequence $\big({\bf x}_{n,j}+\prod_{i=1}^dB_{i,n}^{s_{i,n}}\big)_{1\le j\le k_n, n\in\N}$ satisfies  \eqref{condition:measure}. 
    
\end{itemize}
That is to say that all conditions of Theorem \ref{MTP1} are satisfied.
Then, on combining Lemma \ref{lemma2} with Theorem \ref{MTP1}, we obtain that
\begin{eqnarray*}
\begin{aligned}
    \dim_{\rm H} W\big(\mathcal{A},\psi,\mathbf{y}\big)&\ge\dim_{\rm H}\widetilde{W}(\mathcal{A},\psi,{\bfy})  \ge s({\bf u},{\bf v})=\lim_{m\to\infty}\min_{1\le k\le d}\{s({\bf u}(n_m),{\bf v}(n_m),k)\}.
\end{aligned}
\end{eqnarray*}
Observe that
\[s({\bf u}(n_m),{\bf v}(n_m),k)=\underline{s}_{n_m}(\mathcal{A},\psi,k),\qquad\forall \ 1\leq k\leq d,\quad \forall \ m\in\N.\]
This together with \eqref{subseqconv} implies that $\dimh W(\mathcal{A},\psi,{\bf y})\ge \underline{s}(\mathcal{A},\psi)$.

\subsection{The upper bound in  Theorem \ref{application1}}\label{poupthm1}

Let $r_{n,k}=e^{-n(\tau_n+l_{n,k})}$, $1\le k\le d$.
Recall that
 \begin{eqnarray*}
W_n(\psi)&:=&\big\{\mathbf{x}\in [0,1)^d\colon A_n\mathbf{x}\m\in B\big(\mathbf{y}, \psi(n)\big)\big\}\\[4pt]
&=&\bigcup_{\mathbf{z}\in F_n}\Big(\mathbf{z}+A_n^{-1}B_{\R^d}(\mathbf{y},\psi(n))\Big)\m,
 \end{eqnarray*}
 where $F_n=F_n({\bf 0})=\{\mathbf{z}\in[0,1)^d:A_n\mathbf{z}\in\mathbb{Z}^d\}=[0,1)^d\cap A_n^{-1}\Z^d${.
These ellipsoids are denoted by $W_{n,j}$ $(1\le j\le L_n)$, where 
\begin{eqnarray}\label{estmofln}
L_n:=\#F_n\asymp\prod_{i=1}^de^{nl_{n,k}}.
\end{eqnarray}
The asymptotic estimate in \eqref{number:ellipse} follows from Lemma \ref{number:ellipse}

For any $n\in\N$, let $\Lambda_n=A_n^{-1}\mathbb{Z}^d$,} and let $\{{\bf v}_{n,i}:1\le i\le d\}$ be the basis of $\Lambda_n$ given in  Corollary \ref{corsm}. Given $1\leq k\leq d$, consider the parallelepiped
\begin{eqnarray}\label{defofpnk}
    P_{n,k}=\left\{\sum_{j=1}^dr{\bf v}_{n,j}:0\le r\le r_{n,k}\right\}.
\end{eqnarray}
By Corollary \ref{corsm}, the vectors $\{{\bf v}_{n,i}:1\leq i\leq d\}$ are ``almost orthogonal'', which implies that $|P_{n,k} |\asymp r_{n,k}$. Here, $|E|$ stands for the diameter of $E\subseteq\R^d$ under the Euclidean norm.

Fix $n\in\N$ and $1\leq k\leq d$. A standard covering argument yields that there exists a finite set $\mathcal{P}_{n,k}\subseteq\R^d$ such that
\begin{eqnarray}\label{stacovarg}
    A_n^{-1}B_{\R^d}(\bfy,\psi(n))\subseteq\bigcup_{\bfx\in\mathcal{P}_{n,k}}(\bfx+P_{n,k})\qquad\mathrm{and}\qquad\#\mathcal{P}_{n,k}\lesssim\prod_{i\in \mathcal{K}_{n,2}(k)}\frac{r_{n,i}}{r_{n,k}}.
\end{eqnarray}
Let $\widetilde{P}_{n,k}$ be the set defined similarly as \eqref{defofpnk} with $r_{n,k}$ replaced by $2r_{n,k}$. Note that  ${\bf v}+P_{n,k}\subseteq\widetilde{P}_{n,k}$ for any ${\bf v}\in P_{n,k}$. It follows that
\begin{eqnarray}\label{total}
\bigcup_{{\bfx}\in\mathcal{P}_{n,k}}(\bfx+\widetilde{P}_{n,k})\supseteq\bigcup_{{\bf v}\in P_{n,k}\cap\Lambda_n}\bigcup_{\bfx\in\mathcal{P}_{n,k}}({\bf v}+\bfx+P_{n,k})\supseteq\bigcup_{{\bf v}\in P_{n,k}\cap\Lambda_n}\big({\bf v}+A_{n}^{-1}B_{\R^d}(\bfy,\psi(n)\big).
\end{eqnarray}
Recall from Corollary \ref{corsm} that $|{\bf v}_{n,i}|\asymp m_i(\Lambda_n)$.  Then, we have
\begin{eqnarray}\label{local}
   \#(P_{n,k}\cap\Lambda_n)\asymp\prod_{i\in\Gamma_n(k)}\frac{r_{n,k}}{m_i(\Lambda_n)}. 
\end{eqnarray}
On combining \eqref{stacovarg}, \eqref{total} and \eqref{local}, we obtained that  the number of sets of the form ${\bfx}+\widetilde{P}_{n,k}$ needed to cover $W_n(\psi)$ is  
\begin{equation*}
\begin{split}
&\lesssim  \prod_{i\in L_n(k)}\frac{r_{n,i}}{r_{n,k}}\cdot \prod_{i=1}^de^{nl_{n,i}}\cdot \prod_{i\in \Gamma_n(k)}\frac{m_i(\Lambda_n)}{r_{n,k}}\\
&=  \prod_{i\in L_n(k)}e^{n(l_{n,k}-l_{n,i})}\cdot \prod_{i=1}^de^{nl_{n,j}}\cdot \prod_{i\in \Gamma_n(k)}\frac{m_i(\Lambda_n)}{r_{n,k}},
\end{split}
\end{equation*}
denoted by $N_n$. With this in mind, given $\delta >0$ and  $N$ sufficiently large such that $r_{n,k}<\delta$ for any $n \ge N$, it follows 
that for any $t >0$,
\begin{equation*}
 \begin{split}
    \mathcal{H}_{\delta}^t( W\big(\mathcal{A},\psi,\mathbf{y}\big) )&\le  \sum_{n\ge N}\mathcal{H}_{\delta}^t(W_n(\psi))\le  \sum_{n\ge N} N_nr_{n,k}^t.
 \end{split}
\end{equation*}
Therefore, for any 
\begin{equation*}
\begin{split}
t& >\limsup_{n\to\infty} \frac{1}{\tau_n+l_{n,k}}\Big\{\sum_{j=1}^dl_{n,j}-\sum_{j\in L_n(k)}(l_{n,j}-l_{n,k})+\sum_{j\in \Gamma_n(k)}(\tau_n+l_{n,k}+\frac{1}{n}\log m_j(\Lambda_n))\Big\}\\[4pt]
&=\limsup_{n\to\infty}\overline{s}_n(\mathcal{A},\psi,k),
\end{split}
\end{equation*}
we have 
\begin{equation*}
\mathcal{H}^s( W\big(\mathcal{A},\psi,\mathbf{y}\big) ) \le \liminf_{N\to\infty}\sum_{n\ge N}N_nr_{n,k}^t<\infty,
\end{equation*}
which implies that
\begin{equation}\label{upb}
\begin{split}
\dim_{\rm H}  W\big(\mathcal{A},\psi,\mathbf{y}\big) & \le \min_{1\le k\le d} \Big\{\limsup_{n\to\infty}\overline{s}_n(\mathcal{A},\psi,k)\Big\}\\[4pt]
&=\overline{s}(\mathcal{A},\psi).
\end{split}
\end{equation}


\subsection{ Dimension formula for diagonal matrices}
{Let $\mathcal{A}$ be a sequence of diagonal matrices. In the previous subsections, we have established that
\[\underline{s}(\mathcal{A},\psi)\le \dim_{\rm H} W\big(\mathcal{A},\psi,{\bf y}\big)\le \overline{s}(\mathcal{A},\psi).\]
To prove the desired equality, it suffices to show that  $\underline{s}(\mathcal{A},\psi)=\overline{s}(\mathcal{A},\psi)$} in the context of diagonal matrices. 
For any $n\in\N$, write $A_n={\rm diag}(e^{nl_{n,1}}, \dots, e^{nl_{n,d}})$ and let
\[F_n({\bf y})=\{{\bf x}\in  [0,1)^d\colon A_n{\bf x}\m ={\bf y}\},\]
Note that for any ${\bf x}=(x_1,x_2,\dots,x_d)\in F_n({\bf y})$ and ${\bf y}=(y_1,y_2,\dots,y_d)$, we have 
\begin{equation*}
e^{nl_{n,i}}x_i~(\bmod~1)=y_i\quad \Longleftrightarrow\quad x_i=e^{-nl_{n,i}}y_i+e^{-nl_{n,i}}\mathbb{Z},\quad 1\le i\le d.
\end{equation*}
It follows that 
\[F_n({\bf y})=\Big\{\Big(\frac{y_1+k_1}{e^{nl_{n,1}}},\frac{y_2+k_2}{e^{nl_{n,2}}},\dots,\frac{y_d+k_d}{e^{nl_{n,d}}}\Big)\colon 0\le k_j<e^{nl_{n,j}},\, 1\le j\le d\Big\}\]
{and $\# F_n({\bf y})=\prod_{i=1}^d\lfloor e^{nl_{n,j}} \rfloor $, where $\lfloor x\rfloor:=\max\{z\in\Z:z\leq x\}$ for any $x\in\R$.}

Without loss of generality, we assume that
\[l_{n,1}\le l_{n,2}\le \cdots\le l_{n,d}.\]
{Recall that $m_i(A_n^{-1}\mathbb{Z}^d)$, $1\le i\le d$ denote the successive minima of the lattice $A_n^{-1}\mathbb{Z}^d$.} In the diagonal setting, it is easily seen that
\[m_i(A_n^{-1}\mathbb{Z}^d)=e^{-nl_{n,i}}, \quad 1\le i\le d.\]
It follows that
\begin{align*}
\Gamma_n(i)&=\{j\colon l_{n,j}\ge \tau_n+l_{n,i}\}\\[4pt]
&=\mathcal{K}_{n,1}(i)\cup \{j\colon l_{n,j}= \tau_n+l_{n,i}\}.
\end{align*}
Hence, for any $1\le i\le d$,
\begin{equation}\label{equal}
\begin{split}
\overline{s}_n(\mathcal{A},\psi,i)&=\frac{1}{\tau_n+l_{n,i}}\cdot\Big(\sum_{j=1}^dl_{n,j}-\sum_{j\in \mathcal{K}_{n,2}(i)}(l_{n,j}-l_{n,i})\\[4pt]
&\quad\hspace{18ex} +\sum_{j\in \Gamma_n(i)}(\log m_j(A_n^{-1}\mathbb{Z}^d)+\tau_n+l_{n,i})\Big)\\[4pt]
&=\underline{s}_n(\mathcal{A},\psi,i)+\frac{1}{\tau_n+l_{n,i}}\cdot\Big(\sum_{1\leq j\leq d, l_{n,j}= \tau_n+l_{n,i}}(-l_{n,j}+\tau_n+l_{n,i})\Big)\\[4pt]
&=\underline{s}_n(\mathcal{A},\psi,i).
\end{split}
\end{equation}
The proof is complete.


\section{Proofs of Theorem \ref{counterexample}, and Example  \ref{exam1}}\label{proof13}
\subsection{Proof of Theorem \ref{counterexample}} 

Let $\psi:\R^+\to\R^+$ be a positive and non-increasing function. Recall that $\tau_n=\frac{1}{n}\log\psi(n)$ and $\tau$ is the lower order at infinity of $\psi$. 
 From now on, without loss of generality, we assume that  ${\bf y= 0}$. With this in mind, for any $(A_n)_{n\in\N}$, we have
 \[W_n(\psi)=\{\mathbf{x}\in  [0,1)^d\colon A_n\mathbf{x}\m\in B(\mathbf{0},\psi(n))\},\]
 and 
 \[F_n=F_n({\bf 0})=\{\mathbf{x}\in  [0,1)^d\colon A_n\mathbf{x}\m=0\}.\]
 
Before giving the proof of Theorem \ref{counterexample}, we prove two lemmas. Recall $l_{n,i}=\frac{1}{n}\log\sigma_i(A_n)$. 

 \begin{lem}\label{lambdalni} 
 Assume that  $A\in GL_d(\mathbb{Z})$ satisfies $|\det A |=1$. For $n\in\N$, Let $A_n=(\lambda A)^n$, where $\lambda\in\mathbb{N}$. 
 Then we have the following statements:
 \begin{itemize}
 \item[(1)]  For $n\in\N$, $\lambda=\exp\{\frac{1}{d}\sum_{i=1}^d l_{i,n} \}.$
 \item[(2)] $\# F_n=|\det A_n|=e^{\sum_{i=1}^d l_{i,n}},$  for $n\in\N$.
\item[(3)] For $n\in\N$, $F_n=\Big\{\big(\frac{i_1}{\lambda^n},\frac{i_2}{\lambda^n},\cdots,\frac{i_d}{\lambda^n}\big)^T\colon 0\le i_k<\lambda^n,~1\le k\le d\Big\}.$
 \end{itemize}
 \end{lem}
 \begin{proof}
For $n\in\N$, we have $\det A_n=\prod_{i=1}^d\sigma_i(A_n)=e^{n\sum_{i=1}^d l_{i,n}}=\lambda^{dn}$, which implies that $\lambda=\exp\{\frac{1}{d}\sum_{i=1}^d l_{i,n} \}.$
 The second statement (2) follows from Lemma 2.3 in \cite{EW}.  
 
 For $n\ge1$ and $\mathbf{x}=\big(\frac{i_1}{\lambda^n},\frac{i_2}{\lambda^n},\cdots,\frac{i_d}{\lambda^n}\big)^T$, $A_n\mathbf{x}=\lambda^nA^n\mathbf{x}=A^n{\bf i}\in\mathbb{Z}^d$, where ${\bf i}=\big(i_1,i_2,\cdots,i_d\big)^T$.  Also by (2), $\#\Big\{\big(\frac{i_1}{\lambda^n},\frac{i_2}{\lambda^n},\cdots,\frac{i_d}{\lambda^n}\big)\colon 0\le i_k<\lambda^n,~1\le k\le d\Big\}=\lambda^{dn}=\#F_n$, which completes the proof of (3).
 \end{proof}
 
 \begin{lem}\label{formulaequal}
Assume that  $A\in GL_d(\mathbb{Z})$ satisfies $|\det A |=1$. For $n\in\N$, Let $A_n=(\lambda A)^n$, where $\lambda\in\mathbb{N}$ such that the modulus  of eigenvalues of $\lambda A$ is strictly larger than 1. Let $\psi:\R^+\to\R^+$ be a positive and non-increasing function. Then
 \[\overline{s}(\mathcal{A},\psi)=\hat{s}(\mathcal{A},\psi).\]
 \end{lem}
 \begin{proof}
 
 By Lemma \ref{lambdalni} (3), we have
 \[\frac{1}{n}\log m_k(\Lambda_n)=\log\lambda,\quad 1\le k\le d.\]
 It implies that $\Gamma_n(i)=\{1,2,\dots,d\}$ or $\emptyset$ for $1\le i\le d$. 
 
 Assume that $l_{n,1}\le l_{n,2}\le \cdots\le l_{n,d}$. Since $\log\lambda=\frac{1}{d}\sum_{j=1}^dl_{n,j}\le l_{n,d}$, we have $ \log\lambda\le \tau_n+l_{n,d}$. Depending on the size of $\tau_n$, there are three cases to
consider: $\tau_n+l_{n,k_0}\le \log\lambda<\tau_n+l_{n,k_0+1}$ for some $1\le k_0<d$, $\log\lambda<\tau_n+l_{n,1}$ and $\log\lambda=\tau+l_{n,d}$.
 
{\bf Case (1)}:\,  there is some $1\le k_0<d$ such that 
 \[\tau_n+l_{n,k_0}\le \log\lambda<\tau_n+l_{n,k_0+1}.\]
 For $i>k_0$, we get $\Gamma_n(i)=\emptyset$, then 
 \begin{equation}\label{snwsnh}
 \overline{s}_n(\mathcal{A},\psi,i)=\frac{1}{\tau_n+l_{n,i}}\left\{\sum_{j=1}^dl_{n,j}-\sum_{j\in\mathcal{K}_{n,2}(i)}(l_{n,j}-l_{n,i})\right\}=: \hat{s}_n(\mathcal{A},\psi,i).
 \end{equation}
 For $1\le i\le k_0$, we get $\Gamma_n(i)=\{1,2,\dots,d\}$. For such $i$,
 \begin{equation*}
 \begin{split}
 \overline{s}_n(\mathcal{A},\psi,i)&=\frac{1}{\tau_n+l_{n,i}}\Big(\sum_{j=1}^dl_{n,j}-\sum_{j\in \K_{n,2}(i)}(l_{n,j}-l_{n,i})
+\sum_{j=1}^d(-\log\lambda+l_{n,i}+\tau_n)\Big)\\
&=\frac{1}{\tau_n+l_{n,i}}\Big(-\sum_{j\in \K_{n,2}(i)}(l_{n,j}-l_{n,i})+d(l_{n,i}+\tau_n)\Big)\\
&=d-\frac{\sum_{j\in \K_{n,2}(i)}(l_{n,j}-l_{n,i})}{\tau_n+l_{n,i}}\ge d.
 \end{split}
 \end{equation*}
 Then
 \[\min_{1\le i\le d} \overline{s}_n(\mathcal{A},\psi,i)=\min_{k_0+1\le i\le d} \overline{s}_n(\mathcal{A},\psi,i)=\min_{k_0+1\le i\le d}\hat{s}_n(\mathcal{A},\psi,i).\]

 In the following we will show that for $\tau_n\le \log\lambda-l_{n,k_0}$,
 \begin{equation}\label{721}
 \min_{1\le i\le d}\hat{s}_n(\mathcal{A},\psi,i)=\min_{k_0+1\le i\le d}\hat{s}_n(\mathcal{A},\psi,i),
 \end{equation}
 and it suffices to prove  that for any $1\le i\le k_0$
 \[\hat{s}_n(\mathcal{A},\psi,i)\ge \hat{s}_n(\mathcal{A},\psi,k_0+1).\]
 Write $t_k:=\frac{1}{k}\sum_{j=k+1}^dl_{n,j}$, then $t_{k_0}\le t_{k_0-1}\le \cdots\le t_1$. We observe that when $\tau_n<t_k$,
 \[\hat{s}_n(\mathcal{A},\psi,k)\ge \hat{s}_n(\mathcal{A},\psi,k+1).\]
 Therefore if $\tau_n<t_{k_0}$,
 \begin{equation}\label{snk}
 \hat{s}_n(\mathcal{A},\psi,k_0+1)\le \hat{s}_n(\mathcal{A},\psi,k_0)\le \cdots\le \hat{s}_n(\mathcal{A},\psi,1).
 \end{equation}
 Since 
 \begin{equation*}
 \begin{split}
 t_{k_0}-(\log\lambda-l_{n,k_0})&=\frac{1}{k_0}\sum_{j=k_0+1}^dl_{n,j}+l_{n,k_0}-\frac{1}{d}\sum_{j=1}^dl_{n,j}\\
 &=(\frac{1}{k_0}-\frac{1}{d})\sum_{j=k_0+1}^dl_{n,j}+\frac{1}{k_0}\sum_{j=1}^{k_0}(l_{n,k_0}-l_{n,j})\ge 0,
 \end{split}
 \end{equation*}
we have that inequalities \eqref{snk} also hold when $\tau_n\le \log\lambda-l_{n,k_0}$, which gives \eqref{721}.

{\bf Case (2)}:\, If $\log\lambda<\tau_n+l_{n,1}$, then $\bigcup_{i=1}^d\Gamma_n(i)=\emptyset$, implying that \eqref{snwsnh} holds for any $i$.

{\bf Case (3)}:\,  If $\log\lambda=\tau_n+l_{n,d}$, then $\Gamma_n(i)=\{1,2,\dots,d\}$, $1\le i\le d$. By Case (1),  we have
\[ \min_{1\le i\le d}\hat{s}_n(\mathcal{A},\psi,i)=\hat{s}_n(\mathcal{A},\psi,d)= \overline{s}_n(\mathcal{A},\psi,d)= \min_{1\le i\le d}\overline{s}_n(\mathcal{A},\psi,i)\]

We finish the proof.
 \end{proof}
 
Now we are ready to prove Theorem \ref{counterexample}. Recall that in Theorem \ref{counterexample}, we consider $d=2$, $A\in GL_2(\mathbb{Z})$ satisfies $|\det A |=1$, $A^T=A$ and all the modulus  of eigenvalues of $A$ are not equal to one.  Let $\lambda\in\mathbb{Z}$ be an integer such that  $\lambda A$ is expanding, and put $\mathcal{A}=\{(\lambda A)^n\}_{n\in\N}$. For notational convenience, we assume throughout the rest of this subsection that $\lambda>1$.

 \begin{proof}[Proof of Theorem \ref{counterexample}]
It follows from Theorem \ref{application1} and Lemma \ref{formulaequal} that 
\[\dim_{\rm H}W(\mathcal{A},\psi,{\bf 0})\le \overline{s}(\mathcal{A},\psi)=\hat{s}(\mathcal{A},\psi).\]
In the following we prove that the value $\hat{s}(\mathcal{A},\psi)$ is also the lower bound on $\dim_{\rm H}W(\mathcal{A},\psi,{\bf 0})$.
Rewrite $W_n(\psi)$ as
\begin{equation*}
W_n(\psi)=\bigcup_{\mathbf{x}\in F_n}\left(A_n^{-1}B(\mathbf{0},e^{-n\tau_n})+\{\mathbf{x}\}\right)\m,
\end{equation*}
then $W_n(\psi)$ is the union of finite disjoint ellipsoids with centres in $F_n$, denoted by $\{W_{n,j}\}_{j=1}^{\# F_n}$. Depending on the value of $\tau$, we will consider two cases: $\min_{i=1,2}\{\tau+l_i\}\ge\log\lambda$ and $\min_{i=1,2}\{\tau+l_i\}<\log\lambda$.





\subsubsection{{\bf Case 1:} $\min_{i=1,2}\{\tau+l_i\}\ge\log\lambda$}

In this case, given $\epsilon>0$, there exists $\mathcal{N}$ with $\#\mathcal{N}=+\infty$ such that 
\[\log\lambda<\min_i\{\tau_n+l_{n,i}\}+\epsilon\quad n\in\mathcal{N}.\]
Then for $n\in\mathcal{N}$ and $\mathbf{x}_j\in F_n$
\[W_{n,j}:=\left(\{\mathbf{x}_j\}+A_n^{-1}B(\mathbf{0},e^{-n\tau_n})\right)\m\subset B\left(\mathbf{x}_j, 2\sqrt{2}e^{-n(\log\lambda-\epsilon)}\right). \]
Notice that  $ [0,1)^2=\bigcup_{\mathbf{x}_j\in F_n}B\left(\mathbf{x}_j,2\sqrt{2}e^{-n(\log\lambda-\epsilon)}\right)$, then applying Theorem \ref{MTP1}, we obtain that 
\[s=
\min_{1\le i\le 2}\left\{\sum_{k=1}^2\frac{l_k}{\tau+l_i}+\sum_{k\in\K_2(i)}\frac{l_i-l_k}{\tau+l_i}\right\},\]
where  $\K_2(i)=\{k:l_k\le l_i\}.$


\subsubsection{{\bf Case 2:} $\min_{i=1,2}\{\tau+l_i\}<\log\lambda$}

Assume that $l_1<l_2$, then $\min_{i=1,2}\{\tau+l_i\}=\tau+l_1$. This together with Lemma \ref{lambdalni} (1)     implies that $\tau<\frac{l_2-l_1}{2}$. {Recall that $\tau=\liminf\limits_{n\to\infty}\tau_n$, then there exists an infinite subset $\mathcal{N}\subseteq\N$ such that
\begin{eqnarray}\label{tannlell2}
    \tau_n<\frac{l_2-l_1}{2},\qquad\forall\ n\in\mathcal{N}.
\end{eqnarray}
 }In this case, the lower bound may be obtained in much the same way as  that used in proving \cite[Theorem 4.2]{HP}. 

In the following,  let $\bm{\beta}_1, \bm{\beta}_2$ be the unit eigenvectors (that is $|\bm{\beta}_1|=|\bm{\beta}_2|=1$) of the matrix $\lambda A$ associated with its eigenvalues $\lambda_1, \lambda_2$ respectively.  Let us first list some basic facts that are helpful in the proof:

    \begin{itemize}

    \item[(F1)] Since $A=A^T$, the eigenvalues $\lambda_1,\lambda_2$ of the matrix $\lambda A$ are real numbers and the associated eigenvectors $\bm{\beta}_1$, $\bm{\beta}_2$ are orthogonal; that is to say that their inner product $\bm{\beta}_1\cdot\bm{\beta}_2$ is equal to zero. It follows that the vectors ${\bm\beta}_1$ and ${\bm\beta}_2$ constitute an orthogonal basis of the Euclidean space $\R^2$.
    \medskip

        \item[(F2)]  Since $|\det A|=1$ and the modules of all the eigenvalues of $A$ have modules  are not equal to one, it follows that there exists an eigenvalue of $A$ with modules strictly less than one. This together with  Lemma 5.1 in \cite{HPWZ} implies that both $\lambda_1$ and $\lambda_2$ are irrational quadratic algebraic numbers. Write $\bm{\beta}_i=(\beta_{i,1},\beta_{i,2})$ where $i=1,2$, then it is easily verified that $\frac{\beta_{i,2}}{\beta_{i,1}}$ $(i=1,2)$ are also irrational quadratic algebraic numbers.\medskip

        \item[(F3)] From the above (F2), the unit eigenvector $\bm{\beta}_i$ $(i=1,2)$ is not equal to $(\pm1,0)$ or $(0,\pm1)$. With this in mind,   let $\alpha\in(0,\pi/2)$ be the angle between   $\{0\}\times\R$ and the straight line through $\bm{\beta}_1$.  Recall that $A_n=\lambda^n A^n$ for any $n\in\N$. Then, by a straightforward calculation, we obtain that
\begin{eqnarray}\label{lenofline}
\begin{aligned}
    |(\{x\}\times\R)\,\cap\, A_n^{-1}B_{\R^2}({\bf{y}},e^{-n\tau_n})|\ &\leq\ \frac{2}{\sqrt{\cos^2\alpha\cdot e^{2n(l_1+\tau_n)}+\sin^2\alpha\cdot e^{2n(l_2+\tau_n)}}}\\[4pt]
    \ &\leq\ \frac{2}{\sin\alpha}\cdot e^{-n(l_2+\tau_n)}
    \end{aligned}
\end{eqnarray}
for any $x\in\R$, ${\bf{y}}\in\R^d$ and $n\in\N$.\medskip
    \end{itemize} 

Now, given $x\in\R$, let
\[S_{x}:=\{x\}\times \R.\]
{Since $\lambda A$ is expanding, without loss of generality, we assume that $e^{-n(\tau_n+l_1)}<1/2$ for any $n\in\N$. Then the diameter of the ellipsoid $A_n^{-1}B_{\R^2}(\bm{0},e^{-n\tau_n})$ is less than one for all $n\in\N$. It follows that for any $x\in[0,1)$, $n\in\N$ and ${\bf{x}}_j\in F_n$, there exists at most one integer $z\in\mathbb{Z}\setminus\{0\}$ such that  
\[
\big({\bf{x}}_j+A_n^{-1}B_{\mathbb{R}^2}(\bm{0},e^{-n\tau_n})\big)\cap S_{x+z}\neq\emptyset.
\]
The upshot of the above discussion is that the set
\[
W_{n,j}\cap S_x=\bigcup_{z\in\mathbb{Z}}\big({\bf{x}}_j+A_n^{-1}B_{\mathbb{R}^2}(\bm{0},e^{-n\tau_n})\big)\cap S_{x+z}\m\]
is either an empty set or an interval in $\{x\}\times [0,1)$.} Denote
\begin{equation}\label{defofwnx}
    \begin{aligned}
        W_n(x)&:=\bigcup\left\{W_{n,j}\cap S_x\colon  |W_{n,j}\cap S_x|\geq\frac{1}{2}e^{-n(l_2+\tau_n)},~1\le j\le \#F_n\right\}\\[4pt]
        &=\bigcup_{j\in\cJ_n(x)}W_{n,j}\cap S_x\\[4pt]
        &=\bigcup_{j\in\cJ_n(x)}I_{n,j},
    \end{aligned}
\end{equation}
{where $\mathcal{J}_n(x)$ denotes the set of $j\in\{1,2,...,\# F_n\}$ that satisfies the inequality appearing in \eqref{defofwnx} and $I_{n,j} = I_{n,j}(x):=W_{n,j}\cap S_x$ for any $j\in\cJ_n(x)$. Then, for any $j\in\mathcal{J}_n(x)$, as mentioned above, the set $I_{n,j}$ is an interval in $\{x\}\times[0,1)$. Moreover, by \eqref{lenofline} and the definitions of $W_n(x)$ and $I_{n,j}$, the lengths of the intervals $I_{n,j}$ are comparable to $e^{-n(l_2+\tau_n)}$.} The following theorem gives a lower bound on the Hausdorff dimension of $\limsup\limits_{n\to\infty}W_n(x)$, which is crucial to obtain $\dim_{\rm H}W(\mathcal{A},\psi,\bm{0})$.

\begin{proposition}\label{dimWx}
For $n\in\N$ and $x\in [0,1)$, let $W_n(x)$ be the set defined as above. If $\tau=\liminf\limits_{n\to\infty}\tau_n<\frac{1}{2}(l_2-l_1)$, then for any $x\in[0,1)$, we have $\limsup\limits_{n\to\infty}W_n(x)\in\mathcal{G}^{s_0}([0,1))$, where
\[s_0=\limsup_{n\to\infty}\frac{l_{n,2}-\tau_n}{\tau_n+l_{n,2}}.\]


\end{proposition}

Before proving Theorem \ref{dimWx}, we need to establish the following several lemmas. The first result, namely Lemma \ref{DISTR2PL}, concerns the distance between integer points and a line with irrational slope. It is required for estimating the gaps between the intervals in $\{I_{n,j}\}_{j\in\cJ_n(x)}$ (see Lemma \ref{distance}). Let $\mathrm{dist}_{\R^2}(\cdot)$  denote the Euclidean distance on $\R^2$.

\begin{lem}\label{DISTR2PL}
    Let ${\bm\beta}=(\beta_1,\beta_2)\in\R^2\setminus\mathbb{Q}^2$ be a  vector satisfying that $\frac{\beta_2}{\beta_1}$ is an irrational quadratic algebraic number. Then there exist constants $M_0=M_0({\bm\beta})>0$ and $C=C({\bm\beta})\in(0,1)$ such that for any $M>M_0$ and any ${\bf z}\in\mathbb{Z}^2$,
    \begin{equation}\label{distofpl}
         \mathrm{dist}_{\R^2}({\bf p},L)>M^{-1},\qquad\forall\ {\bf p}\in\mathbb{Z}^2\setminus\{{\bf z}\},
    \end{equation}
    where  $L=L({\bf z},\bm{\beta},M)\subseteq\R^2$ is the line defined as
\begin{equation}\label{defoflil}
      L:=\left\{{\bf z}+t{\bm\beta}:t\in\R,\,|t|\leq CM\right\}.
\end{equation}
\end{lem}
\begin{proof}
    Since $\frac{\beta_2}{\beta_1}$ is an irrational quadratic algebraic number, by Liouville's theorem, there exists a constant $C_1=C_1({\bm\beta})>0$ such that
    \begin{equation}\label{byliouthm}
        \left|\frac{\beta_2}{\beta_1}\cdot p-q\right|>\frac{C_1}{|p|},\qquad\forall\ p\in\mathbb{Z}\setminus\{0\},\ \forall\  q\in\mathbb{Z}.
    \end{equation}
    Denote $M_0=\max\{2|\bm{\beta}|/(C_1\beta_1),|\bm{\beta}|/\beta_1\}$. Throughout, fix ${\bm z}=(z_1,z_2)\in\mathbb{Z}$, $M>M_0$ and $\bm{p}=(p_1,p_2)\in\mathbb{Z}^2\setminus\{{\bm z}\}$.  Let $L\subseteq\R^2$ be the line defined as in \eqref{defoflil} with $C=\frac{C_1}{2{\bm|\beta|}}$. It is clear that the slope of the line $L$ is $\tan\alpha=\frac{\beta_2}{\beta_1}$, where $\alpha\in(-\pi/2,\pi/2)$. Now we are going to prove the desired inequality \eqref{distofpl} by considering the following cases.
\medskip

\noindent\textit{$\bullet$ Proving \eqref{distofpl} when $p_1=z_1$}. In this case, since $M>|\bm{\beta}|/\beta_1$ by the choice of $M$, we have
\begin{eqnarray*}
    \mathrm{dist}_{\R^2}({\bf p},L)=|p_2-z_2|\cdot|\cos\alpha|\geq\frac{\beta_1}{|\bm{\beta}|}>M^{-1}.\medskip
\end{eqnarray*}

\noindent\textit{$\bullet$ Proving \eqref{distofpl} when $0<|p_1-z_1|\leq2CM\beta_1$}. In this case, by the inequality \eqref{byliouthm}, we obtain that
\begin{eqnarray*}
    \mathrm{dist}_{\R^2}({\bf p},L)&=&\left|z_2+\frac{\beta_2}{\beta_1}(p_1-z_1)-p_2\right|\cdot |\cos\alpha|\\
    &>&\frac{C_1}{|p_1-z_1|}\cdot\frac{\beta_1}{|\bm{\beta}|}\\
    &\geq&\frac{C_1}{2CM\beta_1}\cdot\frac{\beta_1}{|\bm{\beta}|}
    = M^{-1}.
\end{eqnarray*}
\medskip

\noindent\textit{$\bullet$ Proving \eqref{distofpl} when $|p_1-z_1|>2CM\beta_1$}. Recall that by definition, the line $L$ is contained in the set
\[
\left\{\bm{x}=(x_1,x_2)\in\R^2:|x_1-z_1|\leq CM\beta_1\right\}.
\]
Therefore, in this case, since $M>\max\{2|\bm{\beta}|/(C_1\beta_1),1\}$ by the choice of $M$, we have
\begin{eqnarray*}
    \mathrm{dist}_{\R^2}({\bf p},L)\geq CM\beta_1=\frac{C_1M\beta_1}{2|\bm{\beta}|}>1>M^{-1}.\medskip
\end{eqnarray*}

By combining the above three cases, we have proved the desired inequality \eqref{distofpl}.
\end{proof}


The distance  induced by the norm $\Vert\cdot\Vert_2$ on $[0,1)^2$ is denoted by ${\rm dist}(\cdot)$.
\begin{lem}\label{distance}For any $x\in[0,1)$ and $n\in\cN$,
\begin{equation}\label{DistT2IniInj}
    \min_{i\ne j\in\mathcal{J}_n(x)}{\rm dist}(I_{n,i},I_{n,j})\gtrsim e^{-n(l_2-\tau_n)},
\end{equation}
where the implied constant in $\gtrsim$ is independent of $x$ and $n$.
\end{lem}


\begin{proof}

     Let $M_0>0$ and $C\in(0,1)$ be the constants associated with the unit eigenvector $\bm{\beta}_2$ appearing in the statement of Lemma \ref{DISTR2PL}. In the following, without loss of generality, we assume that $e^{-n\tau_n}<\min\{1/4,M_0/2,C/16\}$ for all $n\in\N$. Then, for any ${\bf{z}}_1\neq{\bf{z}}_2\in\mathbb{Z}^2$, we obtain that
    \begin{equation}\label{BZ1Z2EMPTY}
        B_{\R^2}({\bf z}_1,2e^{-n\tau_n})\cap B_{\R^2}({\bf z}_2,2e^{-n\tau_n})=\emptyset,\qquad\forall\ n\in\N.
    \end{equation}
    Fix $\bf{z}\in\mathbb{Z}^2$ and $n\in\N$, consider the ellipsoid
    \begin{equation*}
        S=S(n,{\bf{z}}):=\left\{{\bf{z}}+y_1{\bm\beta}_1+y_2{\bm\beta}_2:\frac{y_1^2}{4e^{-2n\tau_n}}+\frac{16y_2^2}{C^2 e^{2n\tau_n}}<1\right\}.
    \end{equation*}
    The above (F1) states that the unit eigenvectors ${\bm\beta}_1$ and ${\bm\beta}_1$ are orthogonal, then it is clear that the following equality holds
    \[
    B_{\R^2}({\bf{z}},2e^{-n\tau_n})=\left\{{\bf{z}}+y_1{\bm\beta}_1+y_2{\bm\beta}_2:y_1^2+y_2^2<4e^{-2n\tau_n}\right\}
    \]
    and thus, in view of the range of $e^{-n\tau_n}$, we have $B_{\R^2}({\bf{z}},2e^{-n\tau_n})\subseteq S$. Moreover, the set $S$ is contained in
\begin{equation}\label{neiboflil}
        \left\{{\bf{x}}\in\R^2:\mathrm{dist}_{\R^2}({\bf{x}},L)<2e^{-n\tau_n}\right\},
    \end{equation}
    where $L=L(n,{\bf{z}})$ denotes the line along the direction ${\bm\beta}_2$ defined by
    \begin{equation*}
        L(n,{\bf{z}}):=\left\{{\bf z}+t{\bm\beta}_2:t\in\R,\,|t|\leq Ce^{n\tau_n}/4\right\}.
    \end{equation*}
    As mentioned in (F2) above, the real number $\frac{\beta_{2,2}}{\beta_{2,1}}$ is a quadratic algebraic number. Note that $e^{-n\tau_n}<M_0/2$ as we have assumed before, then on applying Lemma \ref{DISTR2PL} to $L$ and $M=M(n):=e^{n\tau_n}/4$, we obtain that
\begin{equation*}
    \mathrm{dist}_{\R^2}({\bf z}',L)>M^{-1}=4e^{-n\tau_n},\qquad\forall\ {\bf z}'\in\mathbb{Z}^2\setminus\{{\bf z}\}.
\end{equation*}
Combining this  with the fact that $S$ is contained within the set in \eqref{neiboflil}, we have
\begin{equation*}
    S(n,{\bf{z}})\cap B_{\R^2}({\bf{z}}',2e^{-n\tau_n})=\emptyset,\qquad\forall\ {\bf{z}'}\in\mathbb{Z}^2\setminus\{{\bf{z}}\}.
\end{equation*}
Multiplying the matrix $A_n^{-1}$ in both sides, we see that
\begin{equation}\label{An-1scapemp}
    A_n^{-1}\big(S(n,{\bf{z}})\big)\cap A_n^{-1}\big(B_{\R^2}({\bf{z}}',2e^{-n\tau_n})\big)=\emptyset,\qquad\forall\ {\bf{z}'}\in\mathbb{Z}^2\setminus\{{\bf{z}}\}.
\end{equation}

Next, we are going to estimate the Euclidean distance between the sets $A_n^{-1}B_{\R^2}({\bf{z}},e^{-n\tau_n})$ and $A_n^{-1}B_{\R^2}({\bf{z}}',e^{-n\tau_n})$ for any ${\bf{z}}\neq{\bf{z}'}\in\mathbb{Z}^2$ and $n\in\cN$. Fix such $\bfz,\bfz'$ and $n$. By combining \eqref{An-1scapemp} and the fact that $$B_{\R^2}({\bf{z}},e^{-n\tau_n})\subseteq B_{\R^2}({\bf{z}},2e^{-n\tau_n})\subseteq S(n,\bf{z}),$$ we have
\begin{equation*}
\mathrm{dist}_{\R^2}\big(A_n^{-1}B_{\R^2}({\bf{z}},e^{-n\tau_n}),A_n^{-1}B_{\R^2}({\bf{z}}',e^{-n\tau_n})\big)\geq\mathrm{dist}_{\R^2}\big(A_n^{-1}B_{\R^2}({\bf{z}},e^{-n\tau_n}),A_n^{-1}S(n,{\bf{z}})^c\big).
\end{equation*}
Now, we estimate the lower bound of right-hand-side of the above inequality. For any ${\bf{x}}\in A_n^{-1}B_{\R^2}({\bf{z}},e^{-n\tau_n})$ and ${\bf{y}}\in A_n^{-1}S(n,\bf{z})^c$, we have
\begin{equation}\label{formofx}
    {\bf{x}}=A_n^{-1}({\bf{z}})+x_1{\bm{\beta}_1}+x_2{\bm{\beta}_2}
\end{equation}
for some $x_1,x_2\in\R$ such that 
\begin{equation}\label{x1la1entaun}
    \frac{x_1^2}{\lambda_1^{-2n}e^{-2n\tau_n}}+\frac{x_2^2}{\lambda_2^{-2n}e^{-2n\tau_n}}<1,
\end{equation}
and
\begin{equation}\label{formofy}
    {\bf{y}}=A_n^{-1}({\bf{z}})+y_1{\bm{\beta}_1}+y_2{\bm{\beta}_2}
\end{equation}
for some $y_1,y_2\in\R$ such that
\begin{equation}\label{clam2entaun}
    \frac{y_1^2}{4\lambda_1^{-2n}e^{-2n\tau_n}}+\frac{16y_2^2}{C^2\lambda_2^{-2n}e^{2n\tau_n}}\geq1.
\end{equation}
By \eqref{formofx} and \eqref{formofy}, 
\begin{equation}\label{x-ygemax}
    |\bfx-\bfy|\geq\max\{|y_1-x_1|,|y_2-x_2|\}.
\end{equation}
We proceed to estimate the lower bound of the right-hand-side of \eqref{x-ygemax} by considering the following cases:
\begin{itemize}
    \item[$\circ$] If $y_2>\frac{C\lambda_2^{-n}e^{n\tau_n}}{8}$, then 
    \begin{eqnarray*}
        |y_2-x_2|&>&\frac{C\lambda_2^{-n}e^{n\tau_n}}{8}-\lambda_2^{-n}e^{-n\tau_n}\\
        &>&\frac{C\lambda_2^{-n}e^{n\tau_n}}{16},
    \end{eqnarray*}
    where the first inequality follows from \eqref{x1la1entaun} and the second inequality follows from the range of $e^{-n\tau_n}$.
    \medskip

    \item[$\circ$] If $y_2\leq\frac{C\lambda_2^{-n}e^{n\tau_n}}{8}$, then by \eqref{clam2entaun}, we have $y_1>\sqrt{3}\lambda_1^{-n}e^{-n\tau_n}$. Also note that by \eqref{tannlell2}, it holds that $\lambda_1^{-n}e^{-n\tau_n}>\lambda_2^{-n}e^{n\tau_n}$. This together with   \eqref{x1la1entaun} implies
    \begin{eqnarray*}
        |y_1-x_1|&>&\sqrt{3}\lambda_1^{-n}e^{-n\tau_n}-\lambda_1^{-n}e^{-n\tau_n}\\
        &=&(\sqrt{3}-1)\lambda_1^{-n}e^{-n\tau_n}\\
        &>&(\sqrt{3}-1)\lambda_2^{-n}e^{n\tau_n}
    \end{eqnarray*}
\end{itemize}
The upshot of the above discussion is that
\begin{eqnarray}\label{distr2br2zz'}
    \mathrm{dist}_{\R^2}\big(A_n^{-1}B_{\R^2}({\bf{z}},e^{-n\tau_n}),A_n^{-1}B_{\R^2}({\bf{z}}',e^{-n\tau_n})\big)\geq\frac{C\lambda_2^{-n}e^{n\tau_n}}{16}=\frac{Ce^{-n(l_2-\tau_n)}}{16}
\end{eqnarray}
for any $\bfz\neq\bfz'\in\mathbb{Z}^2$ and $n\in\cN$. 

We are now in a position to prove \eqref{DistT2IniInj}. Fix any $x\in[0,1)$, $n\in\cN$ and $i\neq j\in\cJ_n(x)$, there exist $\bfz\neq\bfz'\in\Z^2$ such that
\begin{eqnarray}\label{inisubsetanbrz}
    I_{n,i}\subseteq A_n^{-1}B_{\R^2}({\bf{z}},e^{-n\tau_n})\m\quad\&\quad I_{n,j}\subseteq A_n^{-1}B_{\R^2}({\bf{z}'},e^{-n\tau_n})\m.
\end{eqnarray}
To bound the distance between $I_{n,i}$ and $I_{n,j}$, note that $A_n\bfw\in\Z^2$ for any $\bfw\in\Z^2$. Then, by making use of \eqref{distr2br2zz'},   we obtain that
\begin{eqnarray*}
    \|\bfx-\bfy\|&=&\inf_{\bfw\in\Z^2}|\bfx-\bfy-\bfw|\\
    &\geq&\inf_{\bfw\in\Z^2}\dist_{\R^2}\big(A_n^{-1}B_{\R^2}({\bf{z}},e^{-n\tau_n}),A_n^{-1}B_{\R^2}({\bf{z}'}+A_n\bfw,e^{-n\tau_n})\big)\\
    &\geq&\frac{Ce^{-n(l_2-\tau_n)}}{16}
\end{eqnarray*}
for any $\bfx\in A_n^{-1}B_{\R^2}({\bf{z}},e^{-n\tau_n})$ and $\bfy\in A_n^{-1}B_{\R^2}({\bf{z}'},e^{-n\tau_n})$.
This together with \eqref{inisubsetanbrz} implies that
\[
\dist(I_{n,i},I_{n,j})\geq\frac{Ce^{-n(l_2-\tau_n)}}{16}.
\]
The proof is complete.
\end{proof}

\bigskip

In the following, we give an estimate on the number of the intervals in $\{I_{n,j}\}_{j\in\cJ_n(x)}$ intersecting any given subinterval of $\{x\}\times[0,1)$. 

\begin{lem}\label{prelemma}
Given an interval $L\subset \{x\}\times[0,1)$ with length $c_L>0$, we have 
\begin{equation}\label{countupp}
    \#\{j\colon L\cap I_{n,j}\ne\emptyset\}\leq  \lceil c_L e^{n(l_2-\tau_n)}\rceil.
\end{equation}
Moreover, if $\lambda^nc_L\ge 2$, then
\begin{equation}\label{countasym}
    \#\{j\colon L\cap I_{n,j}\ne\emptyset\}\asymp c_Le^{n(l_{2}-\tau_n)}.
\end{equation}
In particular, 
\begin{equation}\label{countspec}
    \#\{I_{n,j}\subset S_x\}\asymp e^{n(l_{2}-\tau_n)}.
\end{equation}
\end{lem}

\begin{proof}
    The upper bound \eqref{countupp} is a direct consequence of Lemma \ref{distance}. Next we show that the corresponding lower bound is valid if $\lambda^nc_L\geq2$. Recall that $\bm{\beta}_1$ is the unit eigenvector associated with the eigenvalue $\lambda_1$ of the matrix $\lambda A$, and $\alpha\in(0,\pi/2)$ is the angle between $\{0\}\times\R$ and the straight line through $\bm{\beta}_1$. Let $P_L$ be the parallelogram in  $[0,1)^2$ defined by
    \begin{equation*}
        P_L:=\bigcup\Big\{(L+b\bm{\beta}_1)\m:|b|\leq\frac{\sqrt{1-\sin^2\alpha/8}}{\sin\alpha}e^{-n(l_1+\tau_n)}\Big\}.
    \end{equation*}
    A direct calculation implies that if the point ${\bf{x}}_j=(p/\lambda^n,q/\lambda^n)\in F_n$ with $(p,q)\in\Z^2$ and $1\leq j\leq \# F_n$ is contained in $P_L$, then $W_{n,j}\cap L\neq\emptyset$ and
    \[
   |W_{n,j}\cap S_x|\geq\frac{1}{2}e^{-n(l_2+\tau_n)}.
    \]
    In other words, 
    \begin{eqnarray}\label{coulow}
    \begin{aligned}
         \#\{L\cap I_{n,j}\neq\emptyset\}\ &\geq\ \#\{(p,q)\in\Z^2:(p/\lambda^n,q/\lambda^n)\in P_L\}\\[4pt]
        \ &=\ \#(\lambda^n P_L\cap\Z^2).
    \end{aligned}
    \end{eqnarray}
    Note that if $\lambda^nc_L\geq2$, then $\lambda^nP_L$ contains a quadrilateral $\widetilde{P}_L$ (in the sense of module $\lambda^n$) such that its vertices are integer points and $\mathcal{L}^2(\widetilde{P}_L)\asymp\mathcal{L}^2(\lambda^n P_L)$. By Pick's theorem \cite{pick}, 
    \begin{eqnarray*}
        \#(\lambda^nP_L\cap\mathbb{Z}^2)\geq\#(\widetilde{P}_L\cap\Z^2)\geq\mathcal{L}^2(\widetilde{P}_L)\asymp\mathcal{L}^2(\lambda^n P_L)\asymp c_L e^{n(l_2-\tau_n)}.
    \end{eqnarray*}
    This together with \eqref{countupp} and \eqref{coulow} gives the desired estimate \eqref{countasym}. Applying it to the special case $L=\{x\}\times[0,1)$, we have \eqref{countspec}. The proof is complete.
\end{proof}

Now we are ready to prove Theorem \ref{dimWx}.
\begin{proof}[Proof of Theorem \ref{dimWx}.]
By Lemma \ref{prelemma}, we have $\mathcal{L}^1(W_n(x))\asymp e^{-2n\tau_n}$.
Define
$$\mu_n=\frac{1}{\mathcal{L}^1(W_n(x))}\mathcal{L}^1|_{W_n(x)}\asymp e^{2n\tau_n}\mathcal{L}^1|_{W_n(x)}.$$
For any ball $B=B(y,r)\subset S_x$, applying Lemma \ref{prelemma}, we obtain that for $n$ large enough, which depends on $B$,
\begin{equation}\label{measurelebesgue}
\begin{split}
\mu_n(B)&\asymp e^{2n\tau_n}\mathcal{L}^1(B\cap W_n(x))\asymp e^{2n\tau_n}\#\{j\colon W_{n,j}\cap B\ne\emptyset\}\mathcal{L}^1(I_{n,j})\\
&\asymp  e^{2n\tau_n}re^{n(l_{n,2}-\tau_n)}e^{-n(\tau_n+l_{n,2})}=r=\mathcal{L}^1(B).
\end{split}
\end{equation}

Next we will estimate $\mu_n(B)$ for all $n$, $y$  and $r$. 
There are two cases.
\begin{itemize}
\item If $r<e^{n(\tau-l_2)}$, by Lemma \ref{prelemma}, $B$ intersects at most one segment of $\{I_{n,j}\}_{j\ge1}$, then
\[\mu_n(B)\lesssim e^{2n\tau_n}\min\{r,e^{-n(\tau_n+l_{n,2})}\}<r^{\frac{l_{2,n}-\tau_n}{\tau_n+l_{2,n}}}.\]
\item If $r>e^{n(\tau-l_2)}$, the number of segments of $\{I_{n,j}\}_{j\ge1}$ intersecting $B$ is at most
\[\lesssim \frac{r}{e^{n(\tau-l_2)}},\] 
then
\[\mu_n(B)\lesssim e^{2n\tau_n}\min\{r,e^{-n(\tau_n+l_{2,n})}\}\frac{r}{e^{n(\tau-l_2)}}=re^{n(l_2-l_{2,n})}.\]
Since $A$ is diagonalizable, we obtain that $\lambda_2^n\asymp \sigma_2(A^n)$, which implies that 
\[\mu_n(B)\lesssim r.\]
\end{itemize}
Combining two cases, we get
\[\mu_n(B)\lesssim r^{\frac{l_{2,n}-\tau_n}{\tau_n+l_{2,n}}}.\]
Then for
\[t<\limsup_{n\to\infty}\frac{l_{n,2}-\tau_n}{\tau_n+l_{n,2}}=s_0,\]
we have $\mu_n(B)\lesssim r^t$, combining \eqref{measurelebesgue} and Lemma \ref{lemmaforlb}, 
giving 
\[\limsup_{ n\to\infty}W_n(x)\in\mathcal{G}^{s_0}([0,1)).\]
\end{proof}

Now it suffices to prove  Theorem \ref{counterexample} for $\tau<\frac{l_2-l_1}{2}$.
 Since 
\[\limsup_{ n\to\infty}W_n(x)\subset W(\mathcal{A},\psi,\bm{0})\cap S_x,\]
applying \cite[Lemma 4.5]{HP}, we obtain that 
$$W(\mathcal{A},\psi,\bm{0})\in \mathcal{G}^{s_0+1}([0,1)^2),$$
where
\[s_0+1=\limsup_{n\to\infty}\frac{2l_{n,2}}{\tau_n+l_{n,2}}.\]
Now we finish the proof of Theorem \ref{counterexample}.
\end{proof}

\subsection{Proof of Example \ref{exam1}}

Recall that $A=\begin{bmatrix}10&5&0\\5&5&0\\0&0&k \end{bmatrix}$ is an expanding integer matrix. The eigenvalues of $A$ are 
\[\frac{15-5\sqrt{5}}{2}=:\lambda_1,\quad \frac{15+5\sqrt{5}}{2}=:\lambda_2,\quad k=:\lambda_3>0,\]
and let $l_i=\log\lambda_i$, $1\le i\le 3$.
Then
\[
  {\bm \alpha}_1=\begin{bmatrix} 1\\ -\frac{1+\sqrt{5}}{2}\\0 \end{bmatrix} \quad  {\bm \alpha}_2=\begin{bmatrix} 1\\
     \frac{\sqrt{5}-1}{2}\\0 \end{bmatrix}\quad {\rm
    and} \quad   {\bm \alpha}_3=\begin{bmatrix} 0\\  0\\1 \end{bmatrix}
\]
are eigenvectors with eigenvalues $\lambda_1$ $\lambda_2$ and $\lambda_3$ respectively, and $\{{\bm \alpha}_i\}_{i=1}^3$ are orthogonal.

Recall $F_n=\{{\bf x}\in  [0,1)^3\colon A^n{\bf x}\m={\bf 0}\}$, $n\in \N$. 
\begin{lem}
For $n\ge1$,
\[F_n=\Big\{{\left(\frac{i_1}{5^n},\frac{i_2}{5^n},\frac{i_3}{k^n}\right)}^T\colon 0\le i_1,i_2<5^n,0\le i_3<k^n\Big\},\]
and $\#F_n=(\lambda_1\lambda_2\lambda_3)^n$.
\end{lem}
\begin{proof}

 For $n\ge1$ and $\mathbf{z}=\big(\frac{i_1}{5^n},\frac{i_2}{5^n},\frac{i_3}{k^n}\big)^T$, note that $A^n\mathbf{z}\in\mathbb{Z}^3$, then
 \[F_n\supset\Big\{{\left(\frac{i_1}{5^n},\frac{i_2}{5^n},\frac{i_3}{k^n}\right)}^T\colon 0\le i_1,i_2<5^n,0\le i_3<k^n\Big\}.\]  
Applying Lemma 2.3 in \cite{EW}, we conclude that $\#F_n=(\lambda_1\lambda_2\lambda_3)^n=\#\{\big(\frac{i_1}{5^n},\frac{i_2}{5^n},\frac{i_3}{k^n}\big)\colon 0\le i_1,i_2<5^n,0\le i_3<k^n\}$, which completes the proof.

\end{proof}

\begin{proof}[Proof of Example \ref{exam1}]
Since $\mathcal{A}_1=(A^n)_n$, by Lemma \ref{lem1},  $\Gamma=\{(\log\lambda_1,\log\lambda_2,\log\lambda_3)\}$.
Without loss of generality, we assume that $\psi(n)=e^{-n\tau}$. Put
\[S_n=\big\{{\bf x}\in  [0,1)^3\colon A^n{\bf x}\m\in B({\bf 0},e^{-n\tau})\big\},\]
and  the lengths of semi-axes of $ A^{-n}B({\bf 0},e^{-n\tau})\m$ are $r_k:=e^{-n(\tau+\log\lambda_k)}$, $k=1,2,3$.

We shall show that 
\begin{equation*}
\dim_{\rm H}W(\mathcal{A}_1,\psi,\mathbf{0})=\begin{cases}
\min\{\frac{\tau+3l_2}{\tau+l_2},\frac{3l_3}{\tau+l_3}\}\quad & \text{if~$\tau\le (l_2-l_1)/2$,}\\
\min\{\frac{\tau+2l_1+l_2}{\tau+l_1},\frac{\tau+3l_2}{\tau+l_2},\frac{3l_3}{\tau+l_3}\}\quad&\text{if~$ (l_2-l_1)/2<\tau\le l_3-l_2$,}\\
\min\{\frac{\tau+2l_1+l_2}{\tau+l_1},\frac{2l_2+l_3}{\tau+l_2},\frac{3l_3}{\tau+l_3}\}\quad&\text{if~$ l_3-l_2<\tau\le l_3-l_1$,}\\
\min\{\frac{l_1+l_2+l_3}{\tau+l_1},\frac{2l_2+l_3}{\tau+l_2},\frac{3l_3}{\tau+l_3}\}\quad &\text{if~$\tau\ge l_3-l_1$.}
\end{cases},
\end{equation*}
the proof of which is divided into two parts.

(i)\, The upper bound on $\dim_{\rm H}W(\mathcal{A}_1,\psi,\mathbf{0})$. 

In this case, $m_1(\Lambda_n)=k^{-n}$ and $m_2(\Lambda_n)=m_3(\Lambda_n)=5^{-n}$.
Observe that 
\begin{equation*}
\begin{split}
\limsup_{n\to\infty}\overline{s}_n(\mathcal{A}_1,\psi,1)&=\begin{cases}
\frac{\sum_{i=1}^3l_i}{\tau+l_1}\quad&\text{if~$\tau>l_3-l_1$,}\\
\frac{\tau+2l_1+l_2}{\tau+l_1}\quad &\text{if~$(l_2-l_1)/2<\tau\le l_3-l_1$,}\\
3\quad & \text{if~$\tau\le (l_2-l_1)/2$.}
\end{cases}\\
\limsup_{n\to\infty}\overline{s}_n(\mathcal{A}_1,\psi,2)&=\begin{cases}
\frac{\tau+3l_2}{\tau+l_2}\quad&\text{if~$\tau\le l_3-l_2$,}\\
\frac{2l_2+l_3}{\tau+l_2}\quad &\text{if~$\tau> l_3-l_2$.}
\end{cases}\\
\limsup_{n\to\infty}\overline{s}_n(\mathcal{A}_1,\psi,3)&=\frac{3l_3}{\tau+l_3}.
\end{split}
\end{equation*}
It follows from Theorem \ref{application1} that
\begin{equation*}
\begin{split}
\dim_{\rm H}W(\mathcal{A}_1,\psi,\mathbf{0})&\le \min_{i=1,2,3}\left\{\limsup_{n\to\infty}\overline{s}_n(\mathcal{A}_1,\psi,i)\right\}\\
&=\begin{cases}
\min\{\frac{\tau+3l_2}{\tau+l_2},\frac{3l_3}{\tau+l_3}\}\quad & \text{if~$\tau\le (l_2-l_1)/2$,}\\
\min\{\frac{\tau+2l_1+l_2}{\tau+l_1},\frac{\tau+3l_2}{\tau+l_2},\frac{3l_3}{\tau+l_3}\}\quad&\text{if~$ (l_2-l_1)/2<\tau\le l_3-l_2$,}\\
\min\{\frac{\tau+2l_1+l_2}{\tau+l_1},\frac{2l_2+l_3}{\tau+l_2},\frac{3l_3}{\tau+l_3}\}\quad&\text{if~$ l_3-l_2<\tau\le l_3-l_1$,}\\
\min\{\frac{l_1+l_2+l_3}{\tau+l_1},\frac{2l_2+l_3}{\tau+l_2},\frac{3l_3}{\tau+l_3}\}\quad &\text{if~$\tau\ge l_3-l_1$.}
\end{cases}
\end{split}
\end{equation*}

(ii)\, The lower bound on $\dim_{\rm H}W(\mathcal{A}_1,\psi,\mathbf{0})$. 

Put 
\[\mu_n=\frac{1}{\mathcal{L}^3(S_n)}\mathcal{L}^3|_{S_n}\asymp e^{3n\tau}\mathcal{L}^3|_{S_n}. \]
By Lemma \ref{number:ellipse} (2), the sequence $(\mu_n)_n$ satisfies inequalities \eqref{lip} in Lemma \ref{lemmaforlb}.

In the following, we will show that there exist $s>0$, $C>0$ and $N\ge1$ such that 
\[\mu_n(B)\le Cr^s\]
hold for any ball $B:=B(\mathbf{w},r)\subset [0,1)^3$ and $n\ge N$. There are four cases:
\begin{itemize}
\item Case 1:\,$\tau<(l_2-l_1)/2$.
\item Case 2:\, $(l_2-l_1)/2\le \tau<l_3-l_2$,
\item Case 3:\, $l_3-l_2\le \tau<l_3-l_1$,
\item Case 4:\, $\tau\ge l_3-l_1$.
\end{itemize}

(1)\, For {\bf Case 1}, we have
\[r_3<\lambda_3^{-n}<r_2<(\lambda_1\lambda_2)^{-\frac{1}{2}n}=5^{-n}<l_1.\]
\begin{itemize}
\item Subcase (i):\, $r<\lambda_3^{-n}$. Since $B$ intersect at most one ellipsoid of $S_n$, then
\[\mu_n(B)\le e^{3n\tau}r^2\min\{r,r_3\}<r^{\min\{\frac{3l_3}{\tau+l_3},\frac{2l_2+l_3}{\tau+l_2}\}}.\]

\item Subcase (ii):\, $\lambda_3^{-n}\le r<5^{-n}$. In this case, the number of ellipsoids of $S_n$ intersecting $B$ is at most $r\lambda_3^{n}$, and it follows that 
\[\mu_n(B)\le e^{3n\tau}r^2\lambda_3^{n}r_3\min\{r,r_2\}<r^{\frac{\tau+3l_2}{\tau+l_2}}.\]

\item Subcase (iii):\,$5^{-n}\le r<e^{-nl_1}$. We need to estimate the number of ellipsoids intersecting $B$. We first give some notation. For any ${\bf x}=(x_1,x_2,x_3)^T$, we throughout denote ${\bf x}_{\rm P}:=(x_1,x_2)^T$. Let $B_{\R^d}$ denotes a ball in $\R^d$ and $\widetilde{B}=B_{\R^d}\m$. 

Then $B\subset (B_{\mathbb{R}^2}({\bf w}_{\rm P},r)\m)\times (B_{\mathbb{R}}(w_3,r)\m)$. 
Let $A_2=5\begin{bmatrix}2&1\\1&1\end{bmatrix}$. Then we have
\begin{align*}
S_n&\subset \{{\bf x}_{\rm P}\in[0,1)^2\colon A_2^n{\bf x}_{\rm P}\m\in B_{\mathbb{R}^2}({\bf 0}_{\rm P},e^{-n\tau})\m\}\\
&\quad \times \{x_3\in[0,1)\colon k^nx_3\m\in B_{\mathbb{R}}( 0,e^{-n\tau})\m\}\\
&=:S_{n,1}\times S_{n,2}.
\end{align*}
Note that 
$$S_{n,1}=\bigcup_{{\bf z}_{\rm P}\in \mathbb{Z}^2\cap A_2^n[0,1)^2 }A_2^{-n}B_{\mathbb{R}^2}({\bf z}_{\rm P},e^{-n\tau})\m.$$
It follows from the proof of  Lemma \ref{distance} that 
\begin{align*}
\inf_{{\bf z}_{\rm P}\ne {\bf z}_{\rm P}'\in \mathbb{Z}^2}{\rm dist} \left(A_2^{-n}B_{\mathbb{R}^2}({\bf z}_{\rm P},e^{-n\tau})\m,A_2^{-n}B_{\mathbb{R}^2}({\bf z}_{\rm P}',e^{-n\tau})\m\right)\gtrsim e^{-n(l_2-\tau)},
\end{align*}
which implies 
\begin{equation*}
\begin{split}
&\#\left\{{\bf z}_{\rm P}\in \mathbb{Z}^2\cap A_2^n[0,1)^2\colon A_2^{-n}B_{\mathbb{R}^2}({\bf z}_{\rm P},e^{-n\tau})\m\cap B_{\R^2}({\bf w}_{\rm P},r)\m\ne\emptyset\right\}\\
&\lesssim\, re^{n(l_2-\tau)}.
\end{split}
\end{equation*}
Therefore the number of ellipsoids intersecting $B$ is
\begin{align*}
    &\le \#\left\{{\bf z}_{\rm P}\in \mathbb{Z}^2\cap A_2^n[0,1)^2\colon A_2^{-n}B_{\mathbb{R}^2}({\bf z}_{\rm P},e^{-n\tau})\m\cap B_{\R^2}({\bf w}_{\rm P},r)\m\ne\emptyset\right\} \\
    &\quad \times\#\left\{z_3\in \mathbb{Z}\cap k^n[0,1)\colon k^{-n}B_{\mathbb{R}}(z_3,e^{-n\tau})\m\cap B_{\R}(w_3,r)\m\ne\emptyset\right\}\\
    &\lesssim re^{n(l_2-\tau)}\cdot rk^n=r^2e^{n(l_2+l_3-\tau)}.
\end{align*}
It implies that
\[\mu_n(B)\le e^{3n\tau}\cdot r^2e^{n(l_3+l_2-\tau)}\cdot r_3r_2\min\{r,r_1\}\le r^3.\]


\item Subcase (iv):\, $r\ge e^{-nl_1}$, in this case, number of ellipsoids intersecting $B$ is at most $r^3(\lambda_1\lambda_2\lambda_3)^{n}$. Hence
\[\mu_n(B)\le e^{3n\tau}r^3(\lambda_1\lambda_2\lambda_3)^{n}r_3r_2r_1=r^3.\]
Combing all subcases, when $\tau<(l_2-l_1)/2$, we have that 
\begin{equation}\label{es1}
\begin{split}
\mu_n(B)&\lesssim \exp\left\{\min\left\{\frac{3l_3}{\tau+l_3},\frac{2l_2+l_3}{\tau+l_2},\frac{\tau+3l_2}{\tau+l_2}\right\}\log r\right\}\\ 
&=\exp\left\{\min\left\{\frac{3l_3}{\tau+l_3},\frac{\tau+3l_2}{\tau+l_2}\right\}\log r\right\}.
\end{split}
\end{equation}
\end{itemize}

(2)\, We deal with {\bf Case 2}: $(l_2-l_1)/2\le \tau<l_3-l_2$. With this in mind, 
\[r_3<\lambda_3^{-n}<r_2<r_1< 5^{-n}.\]

\begin{itemize}
\item Subase (i)\, $r\le r_3$. In this case, $B$ intersects at most one ellipsoid, then 
\[\mu_n(B)\lesssim e^{3n\tau}r^3<r^{\frac{3\log\lambda_3}{\tau+\log\lambda_3}}.\]

\item Subcase (ii)\, $r_3<r\le\lambda_3^{-n}$. The ball $B$ intersects at most one ellipsoid. It follows that 
\[\mu_n(B)\lesssim e^{3n\tau}r^2r_3<r^2e^{n(2\tau-\log\lambda_3)}.\]
If $\tau\le \frac{1}{2}\log\lambda_3$, we use $r_3<r$, then 
\[\mu_n(B)<r^{\frac{3\log\lambda_3}{\tau+\log\lambda_3}}.\]
If  $\tau>\frac{1}{2}\log\lambda_3$, we use $r\le\lambda_3^{-n}<r_2$, then 
\[\mu_n(B)<r^{\frac{2\log\lambda_2+\log\lambda_3}{\tau+\log\lambda_2}}.\]

\item Subcase (iii)\, $\lambda_3^{-n}<r\le r_2$. In this case, the number of  ellipsoids intersecting $B$ is at most $r\lambda_3^{n}.$ Then 
\[\mu_n(B)\le e^{3n\tau} r\lambda_3^{n} r^2e^{-n(\tau+\log\lambda_3)}=r^3e^{2n\tau}<r^{\frac{\tau+3\log\lambda_2}{\tau+\log\lambda_2}}.\]

\item Subcase  (iv)\, $r_2<r\le r_1$. The number of  ellipsoids intersecting $B$ is at most $r\lambda_3^{n}.$ Then 
\[
\mu_n(B)\le  e^{3n\tau} r\lambda_3^{n}e^{-n(2\tau+\log\lambda_2+\log\lambda_3)}r\le r^2e^{n(\tau-\log\lambda_2)}.
\]
If $\tau\le \log\lambda_2$, we have
\[ \mu_n(B)\le r^{\frac{\tau+3\log\lambda_2}{\tau+\log\lambda_2}}.\]
If $\tau> \log\lambda_2$, then
\[ \mu_n(B)\le r^{\frac{\tau+2\log\lambda_1+\log\lambda_2}{\tau+\log\lambda_1}}.\]

\item Subcase  (v)\, $r_1<r\le (\lambda_1\lambda_2)^{-\frac{1}{2}n}$. Similarly to Case  (iv), 
\begin{equation*}
\begin{split}
\mu_n(B)&\le e^{3n\tau} r\lambda_3^{n}  e^{-3n\tau}(\lambda_1\lambda_2\lambda_3)^{-n}\\
& =re^{-n(\log\lambda_1+\log\lambda_2)}<r^{\frac{\tau+2\log\lambda_1+\log\lambda_2}{\tau+\log\lambda_1}}.
\end{split}
\end{equation*}

\item Subcase  (vi)\, $r>e^{-\frac{1}{2}(\log\lambda_1+\log\lambda_2)}$. In this case, the number of  ellipsoids intersecting $B$ is at most $r^3(\lambda_1\lambda_2\lambda_3)^n.$
 It follows that
\[\mu_n(B)\le e^{3n\tau}r^3(\lambda_1\lambda_2\lambda_3)^n e^{-3n\tau}(\lambda_1\lambda_2\lambda_3)^{-n}=r^3.\]

Combining all subcases, we have
\begin{equation}\label{es2}
\begin{split}
 \mu_n(B)&\lesssim \exp\left\{\min\left\{ \frac{3l_3}{\tau+l_3} ,\frac{2l_2+l_3}{\tau+l_2},\frac{\tau+3l_2}{\tau+l_2},\frac{\tau+2l_1+l_2}{\tau+l_1}\right\}\log r\right\} \\
 &=\exp\left\{\min\left\{ \frac{3l_3}{\tau+l_3},\frac{\tau+3l_2}{\tau+l_2},\frac{\tau+2l_1+l_2}{\tau+l_1}\right\}\log r\right\} .
 \end{split}
\end{equation}
\end{itemize}

Similarly, for Case (iii)  we get
\begin{equation}\label{es3}
 \mu_n(B)\lesssim \exp\left\{\min\left\{\frac{\tau+2l_1+l_2}{\tau+l_1},\frac{2l_2+l_3}{\tau+l_2},\frac{3l_3}{\tau+l_3}\right\}\log r\right\} ,
 \end{equation}
and for Case (iv)
\begin{equation}\label{es4}
\mu_n(B)\lesssim \exp\left\{\min\left\{\frac{\sum_{i=1}^3l_i}{\tau+l_1},\frac{2l_2+l_3}{\tau+l_2},\frac{3l_3}{\tau+l_3}\right\}\log r\right\} .
\end{equation}

We conclude from Lemma \ref{lemmaforlb} and inequalities \eqref{es1}--\eqref{es4} that 
\begin{equation*}
\dim_{\rm H}W(\mathcal{A}_1,\psi,\mathbf{0})\ge 
\begin{cases}
\min\{\frac{\tau+3l_2}{\tau+l_2},\frac{3l_3}{\tau+l_3}\}\quad & \text{if~$\tau\le (l_2-l_1)/2$,}\\
\min\{\frac{\tau+2l_1+l_2}{\tau+l_1},\frac{\tau+3l_2}{\tau+l_2},\frac{3l_3}{\tau+l_3}\}\quad&\text{if~$ (l_2-l_1)/2<\tau\le l_3-l_2$,}\\
\min\{\frac{\tau+2l_1+l_2}{\tau+l_1},\frac{2l_2+l_3}{\tau+l_2},\frac{3l_3}{\tau+l_3}\}\quad&\text{if~$ l_3-l_2<\tau\le l_3-l_1$,}\\
\min\{\frac{l_1+l_2+l_3}{\tau+l_1},\frac{2l_2+l_3}{\tau+l_2},\frac{3l_3}{\tau+l_3}\}\quad &\text{if~$\tau\ge l_3-l_1$.}
\end{cases}
\end{equation*}
\end{proof}

\section{Proofs of results in Section \ref{example}}\label{examples}

In this section, we will prove Theorems \ref{diag} and \ref{jordan} and Corollaries \ref{cor} and \ref{jorin}. 

In the following, we prove Corollary \ref{cor}.
Let a matrix $A\in GL_d(\mathbb{Z})$  with eigenvalues $\lambda_1,\dots, \lambda_d$, and $I$ be the identity matrix. Assume that $|\lambda_1|\le |\lambda_2|\le\cdots\le |\lambda_d|$.

\begin{lem}[Lemma 4.1 in \cite{HPWZ}]\label{lem1}
Let $A\in  GL_d(\mathbb{Z})$ be  non-singular, and  $\sigma_{n,1} \leq \cdots \leq \sigma_{n,d}$ be the
  singular values of $A^n$. For any $\epsilon > 0$, there
  exists a constant $c > 0$ such that
  \[
    c^{-1} e^{-\epsilon n} \leq \frac{\sigma_{n,i}}{|\lambda_i|^n}
    \leq c e^{\epsilon n}
  \]
  for all $1\le i\le d$ and $n\in \N$.
\end{lem}

\begin{lem}[Lemma 2.3 in \cite{HL2024}]\label{lem2}
Let a matrix $A\in  GL_d(\mathbb{Z})$. Assume that  the modulus of all eigenvalues are not 1.   Let $e_{n,1} \leq \cdots \leq e_{n,d}$ be the
  singular values of $A^n-I$.  Then there are constants $C > 1$ and $\tau>0$ such that
\[C^{-1}n^{-\tau} \le \frac{e_{n,i}}{|\lambda_i^n-1|}  \le C n^{\tau}\]
holds for all $n\in\N$ and $1\le i\le d$.
\end{lem}
For $x\in\R$ and ${\bm\ell}=(l_1,\cdots,l_d)\in \mathbb{R}^d$, we denote $x+{\bm \ell}=(x+l_1,\cdots,x+l_d)$.

 \begin{lem}\label{supelimsup} Let $\mathcal{A} \subset GL_d(\mathbb{R})$ and $\Gamma(\mathcal{A} )\subset (\mathbb{R}^+)^d$.  Let $\psi : \R^+\to \R^+$ be a positive and non-increasing function. Assume that $\Gamma(\mathcal{A} )$ and $\psi$ satisfy  one of the following conditions.
\begin{enumerate}[label=(\arabic*)]
\item The limit $\lim\limits_{n\to\infty}\tau_n$ exists.
\item The limits $\lim\limits_{n\to\infty}l_{n,i}$, $1\le i\le d$  exist.
\end{enumerate}
Then 
\begin{equation}\label{snk1}
\hat{s}(\mathcal{A},\psi)=\sup_{\bm{\ell}\in\Gamma(\mathcal{A} )}\min_{1\le k\le d}\Bigg\{\frac{1}{\tau+l_k}\Big(\sum_{j=1}^dl_{j}-\sum_{j:l_j<l_ki}(l_j-l_k)\Big)\Bigg\},
\end{equation}
and
\begin{equation*}\label{snk2}
\underline{s}(\mathcal{A},\psi)=\sup_{{\bm \ell}\in\Gamma(\mathcal{A} )}\min_{1\le k\le d}\{s(\bm{\ell},\tau+\bm{\ell},k)\},
\end{equation*}
where $\tau$ is the lower order at infinity of $\psi $, and $\{s(\bm{\ell},\tau+\bm{\ell},k)\}_k$ are given in Section \ref{MTP}.
\end{lem}
\begin{proof}
Firstly we show the equality \eqref{snk1}. 
Write 
\begin{align*}
\hat{s}(\bm{\ell},\tau+\bm{\ell},k)&=\frac{1}{\tau+l_k}\Big(\sum_{j=1}^dl_{j}-\sum_{j:l_j<l_k}(l_j-l_k)\Big),\\
\hat{s}_{n,k}&=\frac{1}{\tau_n+l_{n,k}}\Big\{\sum_{i=1}^dl_{n,i}-\sum_{i:l_{n,i}<l_{n,k}}(l_{n,i}-l_{n,k})\Big\}.
\end{align*}
Now we assume that condition (1) holds. Then there exists $(n_j)_{j\in\N}$ such that 
\[\lim_{j\to\infty}\hat{s}_{n_j,k}=\limsup_{n\to\infty}\hat{s}_{n,k},\quad 1\le k\le d,\]
and $\lim\limits_{j\to\infty}l_{n_j,k}$ exists, denoted by $l_k,$ $1\le k\le d.$ It follows that $\bm{\ell}=(l_1,\dots,l_d)\in \Gamma(\mathcal{A} )$.
Since $\lim\limits_{n\to\infty}\tau_n=\tau$, we have
\[\lim_{j\to\infty}\tau_{n_j}=\tau.\]
Rewrite  $s_{n_j,k}$ as
\begin{equation*}
\hat{s}_{n_j,k}=\frac{1}{\tau_{n_j}+l_{n_j,k}}\Big\{\sum_{i=1}^dl_{n_j,i}-\sum_{i:l_i<l_k}(l_{n_j,i}-l_{n_j,k})+\sum_{i:l_{n_j,i}<l_{n_j,k}\atop l_i\ge l_k}(l_{n_j,i}-l_{n_j,k})\Big\}.
\end{equation*}
Notice that for $i\in\{i:l_{n_j,i}<l_{n_j,k},~ l_i\ge l_k\}$, we have $\lim\limits_{j\to\infty}\frac{1}{\tau_{n_j}+l_{n_j,k}}(l_{n_j,i}-l_{n_j,k})=0,$ giving
that 
\begin{equation*}
\begin{split}
\lim_{j\to\infty}\hat{s}_{n_j,k}&=\lim_{j\to\infty}\frac{1}{\tau_{n_j}+l_{n_j,k}}\Big\{\sum_{i=1}^dl_{n_j,i}-\sum_{i:l_i<l_k}(l_{n_j,i}-l_{n_j,k})\Big\}\\
&=\frac{1}{\tau+l_k}\left(\sum_{i=1}^dl_i-\sum_{i:l_i<l_k}(l_i-l_k)\right)=\hat{s}(\bm{\ell},\tau+\bm{\ell},k).
\end{split}
\end{equation*}
Therefore
\[\hat{s}(\mathcal{A},\psi)=\limsup_{n\to\infty}\min_{1\le k\le d}\hat{s}_{n,k}\le \sup_{\bm{\ell}\in\Gamma(\mathcal{A} )}\min_{1\le k\le d}\hat{s}(\bm{\ell},\tau+\bm{\ell},k).\]

For any $\bm{\ell}=(l_1,\dots,l_d)\in\Gamma(\mathcal{A} )$, there exists $(\tilde{n}_j)_{j\in\N}$ such that 
\[\lim_{j\to\infty}l_{\tilde{n}_j,k}=l_k,\quad 1\le k\le d.\]
Then for $1\le k\le d$,
\[\lim_{j\to\infty}\hat{s}_{\tilde{n}_j,k}=\hat{s}(\bm{\ell},\tau+\bm{\ell},k),\]
which implies that
\[\hat{s}(\bm{\ell},\tau+\bm{\ell},k)\le \limsup_{n\to\infty}\hat{s}_{n,k},\]
giving $\sup_{\bm{\ell}\in \Gamma(\mathcal{A} )}\min_k\hat{s}(\bm{\ell},\tau+\bm{\ell},k)\le \min_k\Big\{\limsup\limits_{n\to\infty}\hat{s}_{n,k}\Big\}.$  Therefore
\[\hat{s}(\mathcal{A},\psi)=\min_k\Big\{\limsup\limits_{n\to\infty}\hat{s}_{n,k}\Big\}=\sup_{\bm{\ell}\in \Gamma(\mathcal{A} )}\min_k\hat{s}(\bm{\ell},\tau+\bm{\ell},k).\]

Then we can use a similar method to prove another equality and other cases. Now we finish the proof.
\end{proof}


\begin{proof}[Proof of Theorem \ref{cor}]
Recall  $\Gamma((A_n)_n)$ is the  set of accumulation points of sequences $$\left\{\big(\frac{1}{n} \log \sigma_1(A_n),\cdots ,\frac{1}{n} \log \sigma_d(A_n)\big)\right\}.$$ It follows from  Lemmata \ref{lem1} and \ref{lem2} that
\begin{equation*}\label{gamma}
\Gamma((A^n)_n)=\Gamma((A^n-I)_n)=\big\{(\log|\lambda_1|,\cdots,\log|\lambda_d|)\big\}.
\end{equation*}
Let $s(\bm{\ell},\tau+\bm{\ell})$ be as given in Section \ref{MTP}, by taking $\bm{\ell}=(\log|\lambda_1|,\cdots,\log|\lambda_d|)$.
Then we conclude from Lemma \ref{supelimsup} that 
\begin{multline*}
\underline{s}\left((A^n)_n,\psi\right)=\sup_{\bm{\ell}\in \Gamma((A^n)_n)}\min_k\{s(\bm{\ell},\tau+\bm{\ell},k)\}=\min_k\{s(\bm{\ell},\tau+\bm{\ell},k)\}\\
=\sup_{\bm{\ell}\in \Gamma((A^n-I)_n)}\min_k\{s(\bm{\ell},\tau+\bm{\ell},k)\}=\underline{s}\left((A^n-I)_n,\psi\right).
\end{multline*}
 We also observe that 
 \[\overline{s}\left((A^n)_n,\psi\right)\le \hat{s}\left((A^n)_n,\psi\right),\quad \overline{s}\left((A^n-I)_n,\psi\right)\le \hat{s}\left((A^n-I)_n,\psi\right),\]
 and applying Lemma \ref{supelimsup},  $ \hat{s}\left((A^n)_n,\psi\right)=\hat{s}\left((A^n-I)_n,\psi\right)$.

 It follows from the equalities above and Theorem \ref{application1} that for any ${\bf y}\in[0,1)^d$,
\[\hat{s}\left((A^n)_n,\psi\right)\ge  \dim_{\rm H}S(\psi,{\bf y}),\dim_{\rm H}R(\psi)\ge \underline{s}\left((A^n)_n,\psi\right).\]
Moreover $ \dim_{\rm H}S(\psi,{\bf y})\le \overline{s}\left((A^n)_n,\psi\right)$, $ \dim_{\rm H}R(\psi)\le\overline{s}\left((A^n-I)_n,\psi\right)$ and $\overline{s}\left(\mathcal{A},\psi\right) \le \hat{s}\left(\mathcal{A},\psi\right)$ for any $\mathcal{A}$.
This completes the proof.

\end{proof}



\subsection{Proof of Theorem \ref{diag}} 

 We start by stating a lemma for later use.
Let $P\in GL_d(\mathbb{Z})$, then it deduces a transformation $T_P\colon [0,1)^d\to[0,1)^d$ as
\[T_P{\bf x}=P{\bf x}\m.\]

\begin{lem}\label{localbi-lip}
    Let $P$ and $T_P$ be as defined above. Assume that $P$ is invertible. Then $T_P$ is locally bi-Lipschitz under the norm $\Vert \cdot\Vert_d$.
\end{lem}
\begin{proof}
    Recall that for ${\bf x}=(x_1,\dots,x_d)$ and ${\bf y}=(y_1,\dots,y_d)$, $\Vert{\bf x}-{\bf y}\Vert_d=(\sum_{i=1}^d\Vert x_i-y_i\Vert^2)^{1/2}$ and $|\cdot |$  denotes the Euclidean metric on $\mathbb{R}^d$.

We observe that for ${\bf x},\,{\bf y}\in\mathbb{R}^d$ with $|{\bf x}-{\bf y}|<\frac{1}{2}$, 
\[\Vert{\bf x}-{\bf y}\Vert_d=\left(\sum_{i=1}^d\vert x_i-y_i\Vert^2\right)^{1/2}=\left(\sum_{i=1}^d| x_i-y_i|^2\right)^{1/2}=|{\bf x}-{\bf y}|.\]
It suffices to show that for any ${\bf x}\in[0,1)^d$, there exists a neighborhood of ${\bf x}$, on which $T_P $ is  bi-Lipschitz.
Note that since $P$ is  invertible, let 
$$\sigma_{min}:=\min_{1\le i\le d}\sigma_i(P)>0\quad {\rm and}\quad \sigma_{max}:=\max_{1\le i\le d}\sigma_i(P)<\infty.$$
Denote
\begin{equation}\label{l}
L:=\max\left\{\sigma_{max}, \frac{1}{\sigma_{min}}\right\}\in(0,\infty),
\end{equation}
and $\delta:=\min\{\frac{1}{4},\frac{1}{4\sigma_{max}}\}$.  Put
$$U_{\mathbb{R}^d}\left({\bf x},\frac{\delta}{2}\right):=\left\{{\bf y}\in[0,1)^d\colon |{\bf y} -{\bf x}|<\frac{\delta}{2}\right\}.$$
For ${\bf z},\,{\bf y}\in U_{\mathbb{R}^d}\left({\bf x},\frac{\delta}{2}\right)$, we have $|{\bf z}-{\bf y}|<\frac{1}{2}$ and 
 \[|P({\bf z}-{\bf y})|\le \sigma_{max}|{\bf z}-{\bf y}|<\frac{1}{2}.\]
Since 
 \[\Vert T_P{\bf z}-T_P{\bf y}\Vert_d=\left(\sum_{i=1}^d\Vert (P({\bf z}-{\bf y}))_i\Vert^2\right)^{1/2},\]
 where $(P({\bf z}-{\bf y}))_i$ denotes the $i$-th component of $P({\bf z}-{\bf y})$, we have
 \[\Vert T_P{\bf z}-T_P{\bf y}\Vert_d=\left(\sum_{i=1}^d\vert (P({\bf z}-{\bf y}))_i\vert^2\right)^{1/2}=|P({\bf z}-{\bf y})|.\]
 It follows  that 
 \[\Vert T_P{\bf z}-T_P{\bf y}\Vert_d=|P({\bf z}-{\bf y})|\le \sigma_{max}|{\bf z}-{\bf y}|=\sigma_{max}\Vert {\bf z}-{\bf y}\Vert_d\le L\Vert {\bf z}-{\bf y}\Vert_d,\]
 and
 \[\Vert T_P{\bf z}-T_P{\bf y}\Vert_d=|P({\bf z}-{\bf y})|\ge \sigma_{min}|{\bf z}-{\bf y}|=\sigma_{min}\Vert {\bf z}-{\bf y}\Vert_d\ge \frac{1}{L}\Vert {\bf z}-{\bf y}\Vert_d.\]
 Now we finish the proof.
\end{proof}
\begin{rem} When $|\det P|>1$, $T_P$ may fail to be injective, and then its inverse map does not exist. However, $T_P$ is locally bijection. 
More precisely, for ${\bf x}\in [0,1)^d$, let  $U:=U_{\mathbb{R}^d}\left({\bf x},\frac{\delta}{2}\right)$ be as given in the proof of Lemma \ref{localbi-lip}. Consider $T_P\colon U\to T_P(U)$, which is surjective. For any  ${\bf z},\,{\bf y}\in U$, if $T_P{\bf z}=T_P{\bf y}$, then $\Vert T_P{\bf z}-T_P{\bf y}\Vert_d=0$, which gives that $|{\bf z}-{\bf y}|=0$, that is, ${\bf z}={\bf y}$. It shows that $T_P$, restricted in $U$,  is injective. Consider the inverse map $T_{P,U}^{-1}\colon T_P(U)\to U$.  Since $T_P\colon U\to T_P(U)$ is bi-Lipschitz,  $T_{P,U}^{-1}$ is also bi-Lipschitz.
\end{rem}

Now we are ready to prove Theorem \ref{diag}.
\begin{proof}[Proof of Theorem \ref{diag}]
Since $A$ is diagonalizable over $\mathbb{Q}$, there is an invertible  matrix $P_0\in GL_d(\mathbb{Q})$ such that 
\[A=P_0DP_0^{-1},\]
where $D={\rm diag}\{\lambda_1,\dots,\lambda_d\}\in GL_d(\mathbb{Q})$. 
Take $p\ge 1$ as the least common multiple of the denominators of elements of $P_0$, which is an integer such that $P:=pP_0 \in GL_d(\mathbb{Z})$. For $n\ge1$, we have $A^n(P{\bf x}\m)\m=P(D^n{\bf x}\m)\m$, that is, $T_A(T_P{\bf x})=T_P(D^n{\bf x}\m)$.

Let $\delta$ and $U_{\mathbb{R}^d}({\bf x}, \frac{\delta}{2})$ be as given in the proof of Lemma \ref{localbi-lip}. Note that $[0,1)^d\subset \bigcup_{{\bf x}\in[0,1)^d}B_{\mathbb{R}^d}({\bf x}, \frac{\delta}{2})$, then there exists a finite  subclass $\{B_{\mathbb{R}^d}({\bf x}_i, \frac{\delta}{2})\}_{i=1}^N$ such that $[0,1)^d\subset \bigcup_{i=1}^NB_{\mathbb{R}^d}({\bf x}_i, \frac{\delta}{2})$. Furthermore
\[[0,1)^d=\bigcup_{i=1}^NB_{\mathbb{R}^d}\left({\bf x}_i, \frac{\delta}{2}\right)\cap [0,1)^d=\bigcup_{i=1}^NU_{\mathbb{R}^d}\left({\bf x}_i, \frac{\delta}{2}\right),\]
and $[0,1)^d=\bigcup_{i=1}^NT_P\left(U_{\mathbb{R}^d}\left({\bf x}_i, \frac{\delta}{2}\right)\right).$

Write $U_i:=U_{\mathbb{R}^d}\left({\bf x}_i, \frac{\delta}{2}\right)$. Let $T_{P,i}^{-1}$ be the inverse map of $T_P$ restricted on $U_i$. For $1\le i\le N$, consider
\[S_i:=\{{\bf x}\in T_P(U_j)\colon A^n{\bf x}\m\in B({\bf y},\psi(n)) ~{\rm for~infinitely~many ~}n\ge1\},\]
and
\[W_i:=\{{\bf x}\in U_j\colon PD^n{\bf x}\m\in B({\bf y},\psi(n)) ~{\rm for~infinitely~many ~}n\ge1\}.\]
Now we show that $T_PW_i=S_i$. For ${\bf x}\in S_i$, we have $T_{P,i}^{-1}{\bf x}\in U_j$, and
\begin{align*}\label{left}
PD^nT_{P,i}^{-1}{\bf x}\m & =PD^nP^{-1}{\bf x}\m=A^n {\bf x}\m\in B({\bf y},\psi(n)),
\end{align*}
then $T_{P,i}^{-1}{\bf x}\in W_i$. Conversely, for ${\bf x}\in W_i$, we get $T_P{\bf x}\in T_P(U_i)$. Also
\begin{align*}
A^n (T_P{\bf x})\m&=PD^nP^{-1} (P{\bf x}\m)\m=PD^n{\bf x}\m\in B({\bf y},\psi(n)),
\end{align*}
then $T_P{\bf x}\in S_i$.

Combing the fact that for $1\le i\le N$,  $T_P$ restricted on $U_i$ is  bi-Lipschitz mappings and \cite[Proposition 3.3]{falconerbook}, we obtain that
\[\dim_{\rm H}\bigcup_{i}W_i=\max_i\{\dim_{\rm H}W_i\}=\max_i\{\dim_{\rm H}T_PW_i\}=\max_i\{\dim_{\rm H}S_i\}=\dim_{\rm H}S(\psi,{\bf y}).\]
To get $\dim_{\rm H}S(\psi,{\bf y})$, it suffices to calculate $\dim_{\rm H}\bigcup_{i}W_i$. 
We observe that 
\begin{multline*}
\quad \bigcup_{i=1}^N W_i=\{{\bf x}\in [0,1)^d\colon PD^n{\bf x}\m\in B({\bf y},\psi(n)) ~{\rm for~infinitely~many ~}n\ge1\}\\
=\bigcup_{{\bf z}\in\mathbb{Z}^d\cap P[0,1)^d}\{{\bf x}\in [0,1)^d\colon D^n{\bf x}\m\in P^{-1}\left(B({\bf y}+{\bf z},\psi(n))\right)\m\\
 ~{\rm for~infinitely~many ~}n\ge1\}.
\end{multline*}
Let $L$ be as defined in \eqref{l}. There exists  ${\bf y}_{\bf z}\in[0,1)^d$ such that $B({\bf y}_z,L^{-1}\psi(n))\m\subset P^{-1}B({\bf y}+{\bf z},\psi(n))\m\subset B({\bf y}_z,L\psi(n))\m$, which implies that 
\[\bigcup_{{\bf z}\in\mathbb{Z}^d\cap P[0,1)^d}W((D^n)_n,L^{-1}\psi,{\bf y}_{\bf z})\subset \bigcup_{i=1}^N W_i\subset \bigcup_{{\bf z}\in\mathbb{Z}^d\cap P[0,1)^d}W((D^n)_n,L\psi,{\bf y}_{\bf z}). \]
It follows  from Theorem \ref{application1} that for any constant $c>0$ and any ${\bf z}\in\mathbb{Z}^d\cap P[0,1)^d$
\[\dim_{\rm H}W((D^n)_n,c\psi,{\bf y}_{\bf z})=\underline{s}\left((D^n)_n,\psi\right) =\underline{s}\left((A^n)_n,\psi\right).\]
We conclude that for any ${\bf y}\in[0,1)^d$
\begin{align*}
\dim_{\rm H}S(\psi, {\bf y})&=\dim_{\rm H}\bigcup_{i}W_i=\underline{s}\left((A^n)_n,\psi\right).
\end{align*}
\end{proof}

\subsection{Proof of Theorem \ref{jordan} and Corollary \ref{jorin}} 

We can deduce Corollary \ref{jorin} from Theorem \ref{jordan} by applying the same method in the proof of Theorem \ref{diag}. Hence we only prove Theorem \ref{jordan}.

\begin{lem}\label{upp} Let $A\in GL_d(\mathbb{Z})$ with eigenvalues $1<|\gamma_1|\le \cdots\le|\gamma_d|$. Assume that 
 \[\lim_{n\to\infty}-\frac{1}{n}\log m_j(A^{-n}\mathbb{Z}^d)=\log|\gamma_j|,\quad 1\le j\le d.\]
 Then 
 \[\overline{s}((A^n)_n,\psi)=\underline{s}((A^n)_n,\psi).\]
 \end{lem}
 
 \begin{proof}
Recall that $h_{n,j}=\frac{1}{n}\log m_j(A^{-n}\mathbb{Z}^d)$, and $l_j=\log|\gamma_j|$, $1\le j\le d$. Here we also have 
\[\lim_{n\to\infty}l_{n,j}=\lim_{n\to\infty}\frac{1}{n}\log\sigma_j(A^n)=l_j,\quad 1\le j\le d.\]
Notice that for $1\le i\le d$,
 \begin{equation*}
 \begin{split}
\overline{s}_{n}((A^n)_n,\psi,i)&=\sum_{j=1}^d\frac{l_{n,j}}{\tau_n+l_{n,i}}-\sum_{j:l_{n,j}<l_{n,i}}\frac{l_{n,j}-l_{n,i}}{\tau_n+l_{n,i}}+\sum_{j\in\Gamma_n(i)}\frac{h_{n,j}+\tau_n+l_{n,i}}{\tau_n+l_{n,i}}\\
&=\sum_{j=1}^d\frac{l_{n,j}}{\tau_n+l_{n,i}}-\sum_{j:l_j<l_i}\frac{l_{n,j}-l_{n,i}}{\tau_n+l_{n,i}}+\sum_{j:l_j>\tau+l_i}\frac{h_{n,j}+\tau_n+l_{n,i}}{\tau_n+l_{n,i}}\\
&\quad +\sum_{j:l_{n,j}<l_{n,i}\atop l_j\ge l_i}\frac{l_{n,j}-l_{n,i}}{\tau_n+l_{n,i}}+\sum_{j\in\Gamma_n(i)\atop l_j\le \tau+l_i}\left(-\frac{h_{n,j}}{\tau_n+l_{n,i}}-1\right)\\
 \end{split}
 \end{equation*}
 We observe that
 \begin{itemize}
 \item for $j\in \{j:l_{n,j}<l_{n,i},~l_j\ge l_i\}$, 
 \[\lim_{n\to\infty}\frac{l_{n,j}}{l_{n,i}}=1.\]
 \item  for $j\in\{j:j\in\Gamma_n(i),~ l_j\le \tau+l_i\}$, 
 \[-h_{n,j}\ge \tau_n+l_{n,i},\quad  l_j\le \tau+l_i,\]
 then
 \[\limsup_{n\to\infty}\frac{-h_{n,j}}{ \tau_n+l_{n,i}}=1.\]
 \end{itemize}
These together with Lemma \ref{supelimsup}, we have that 
 \begin{equation*}
 \begin{split}
 \limsup_{n\to\infty}\overline{s}_{n}((A^n)_n,\psi,i)&=\sum_{j=1}^d\frac{l_j}{\tau+l_i}-\sum_{j:l_j<l_i}\frac{l_j-l_{i}}{\tau+l_{i}}+\sum_{j:l_j>\tau+l_i}\frac{l_j+\tau+l_{i}}{\tau+l_{i}}\\
 &= \limsup_{n\to\infty}\underline{s}_{n}((A^n)_n,\psi,i),
  \end{split}
 \end{equation*}
 which gives this lemma.
 \end{proof}

\begin{proof}[Proof of Theorem \ref{jordan}]
It follows from Theorem \ref{application1} that for any ${\bf y}\in[0,1)^d$,
\[\dim_{\rm H}W(\psi,{\bf y})\ge \underline{s}((A^n)_n,\psi).\]
Hence  we just need to prove that $\underline{s}((A^n)_n,\psi)$ is also the upper bound of $\dimh W(\psi)$.


    Given $j\in\{1,...,s\}$, $\mathbf{x}=(x_1,...,x_{n_j})\in\R^{n_j}$, denote $|\cdot|_j$ as the $L^{\infty}$ norm on $\R^{n_j}$, that is,
    \[|\mathbf{x}|_j:=\max\{|x_1|,...,|x_{n_j}|\}.\] 
     Let $\epsilon>0$ be a small enough number, then after a straightforward calculation we have
 \begin{eqnarray}
 \begin{split}\label{norm}
     (|\lambda_j|-\epsilon)^n|\mathbf{v}|_j\leq&|J(\lambda_j,n_j)^n\mathbf{v}|_j\leq (|\lambda_j|+\epsilon)^n|\mathbf{v}|_j,\\
     (|\lambda_j|+\epsilon)^{-n}|\mathbf{v}|_j\leq&|J(\lambda_j,n_j)^{-n}\mathbf{v}|_j\leq (|\lambda_j|-\epsilon)^{-n}|\mathbf{v}|_j
 \end{split}
 \end{eqnarray}
 for all $j=1,...,s$, $\mathbf{v}\in\R^{n_j}$ and all $n\gg1$. 
 
 Therefore for any $p\ge 1$ with $\sum_{k=1}^{j-1}n_k<p\le \sum_{k=1}^{j}n_k$, we have
 \[\log(|\lambda_j|-\epsilon)\le -\frac{1}{n}\log m_p(\Lambda_n)\le\log (|\lambda_j|+\epsilon).\]
 It follows that 
 \[\lim_{n\to\infty}-\frac{1}{n}\log m_p(\Lambda_n)=\log|\lambda_j|.\]
By Lemma \ref{upp}, we have $\overline{s}((A^n)_n,\psi)=\underline{s}((A^n)_n,\psi)$. 

The following observation finishes the proof which can be verified directly
\[\underline{s}((A^n)_n,\psi)=\min_{1\le j\le s}\left\{\sum_{j\in\K'_1(i)}n_j+\sum_{j\in\K'_2(i)}\frac{n_j\log|\lambda_i|}{\log|\lambda_i|+\tau}+\sum_{j\in\K'_3(i)}\frac{n_j\log|\lambda_j|}{\log|\lambda_i|+\tau}\right\},\]
where
\begin{eqnarray*}
\K'_1(i)&=&\left\{1\leq j\leq s:\log|\lambda_j|>\log|\lambda_i|+\tau\right\},\\
\K'_2(i)&=&\left\{1\leq j\leq s:|\lambda_j|<|\lambda_i|\right\},\\
\K'_3(i)&=&\{1,...,s\}\backslash(\K_1(i)\cup\K_2(i)).
\end{eqnarray*}

\end{proof}

\subsection*{Acknowledgements}

This work was supported by National Key R$\And$D Program of China (No. 2024YFA1013700), NSFC 12271176, NSFC 12501115, Beijing Natural Science Foundation 1254048,  Guangdong Natural Science Foundation 2024A1515010946.

\vspace*{10ex}

\noindent Z. N. Hu: College of Science, China University of Petroleum, Beijing,

\noindent\phantom{Zhang-nan Hu: }  Fuxue Road 18, Changping District, Beijing, P. R. China 

\noindent\phantom{Zhang-nan Hu: }  hnlgdxhzn@163.com

\vspace{5mm}

\noindent J. J. Huang:  School of Mathematics, South China University of Technology, 

\noindent\phantom{Jun-jie Huang: } Wushan Road 381, Tianhe District, Guangzhou, P. R. China

\noindent\phantom{Jun-jie Huang: } 	h1135778478@163.com

\vspace{5mm}

\noindent B. Li:  School of Mathematics, South China University of Technology, 

\noindent\phantom{Bing Li: } Wushan Road 381, Tianhe District, Guangzhou, P. R. China

\noindent\phantom{Bing Li: } scbingli@scut.edu.cn

\vspace{5mm}
\noindent J. Wu:  School of Mathematics and Statistics, Huazhong University of Science 

\noindent\phantom{Jun Wu: }  and Technology, Luoyu Road 1037,  Wuhan, P. R. China

\noindent\phantom{Jun Wu: } jun.wu@hust.edu.cn

\end{sloppypar}
\end{document}